\documentclass{article}

\usepackage[T1]{fontenc}
\usepackage{amsfonts}
\usepackage{makeidx}
\usepackage{amsmath}
\usepackage{amssymb}
\usepackage{amsthm}
\usepackage[final]{graphicx}
\usepackage[all]{xy}
\usepackage{float}
\usepackage{fancybox}
\usepackage{color}
\usepackage[normalem]{ulem}
\usepackage{authblk}

\usepackage{multicol}

\newcommand{\OV}{\mathcal{O}_{\tilde{V}_0^{(n+1)},\xi _0}}

\theoremstyle{plain}
\newtheorem{Thm}{Theorem}[subsection]         

\newtheorem{Lemma}[Thm]{Lemma}
\newtheorem{Cor}[Thm]{Corollary}
\newtheorem{Prop}[Thm]{Proposition}

\theoremstyle{definition}
\newtheorem{Def}[Thm]{Definition}  
\newtheorem{Par}[Thm]{ }
\newtheorem{Rem}[Thm]{Remark}

\theoremstyle{remark}

\newtheorem{Notation}[Thm]{Notation}
\newtheorem{Ex}[Thm]{Example}

\numberwithin{equation}{Thm}

\begin{document}

\title{Nash multiplicities and resolution invariants}
\author{A. Bravo, S. Encinas, B. Pascual-Escudero
\footnote{The authors were partially supported by MTM2012-35849. The third autor was supported by BES-2013-062656.
}} 

\AtEndDocument{\bigskip{\footnotesize
  \textsc{Depto. Matem\'aticas,
Facultad de Ciencias, Universidad Aut\'onoma de Madrid 
and ICMAT (Instituto de Ciencias Matematicas CSIC-UAM-UC3M-UCM), Cantoblanco 28049 Madrid, Spain} \par  
  \textit{E-mail address}, A. Bravo: \texttt{ana.bravo@uam.es} \par
  \textit{E-mail address}, B. Pascual-Escudero: \texttt{beatriz.pascual@uam.es} \par
  \addvspace{\medskipamount}
  \textsc{Depto. Matem\'atica Aplicada,
and IMUVA, Instituto de Matem\'aticas.
Universidad de Valladolid. Paseo de Bel\'en 7. 47011 Valladolid, Spain.
} \par
  \textit{E-mail address}, S. Encinas: \texttt{sencinas@maf.uva.es}
}}

\date{}

\maketitle

\begin{abstract}
The Nash multiplicity sequence was defined by M. Lejeune-Jalabert as a non-increasing sequence of integers attached to a germ of a curve inside a germ of a hypersurface. M. Hickel generalized this notion and described a sequence of blow ups which allows us to compute it and study its behavior. \\
\\
In this paper, we show how this sequence can be used to compute some invariants that appear in algorithmic resolution of singularities.
Moreover, this indicates that these invariants from constructive resolution are intrinsic to the variety since they can be read in terms of its space of arcs.
This result is a first step connecting explicitly arc spaces and algorithmic resolution of singularities.
\end{abstract}


\normalsize

\section*{Introduction}
Consider a variety $X$ of dimension $d$ over a field $k$. By a \textit{resolution of singularities} we mean a proper birational morphism 
$$X\stackrel{\phi }{\longleftarrow }X'$$
such that $X'$ is regular. In addition we require that $\phi $ induces an isomorphism on the set of regular points of $X$, and that the exceptional divisor $\phi ^{-1}(\mathrm{Sing}(X))$ has normal crossing support.

In \cite{Hir}, Hironaka proved that given a variety over a field of charateristic zero it is possible to find a resolution of singularities of $X$ defined by a sequence of blow ups at smooth centers. Moreover, it is possible to construct such a sequence by means of some \textit{invariants} attached to the points of $X$ (see \cite{V1}, \cite{V2}, \cite{B-M}). The study of those invariants becomes interesting as soon as they provide an algorithm for the construction of a resolution of singularities for any variety over a field of characteristic zero. Furthermore, they may also give insight into the resolution phenomenon, in order to solve the problem for more general fields. Through these invariants, one can define \textit{resolution functions}, which stratify $X$ in locally closed sets, so that there is a canonical (regular) center to blow up at each step of the resolution sequence. Then resolution is achieved via the construction of a finite sequence of blow ups.\\
\\
One of the ingredients that one may take into account for this stratification is the \textit{multiplicity function} (see \cite{V}). The multiplicty is an upper semi-continuous function defined at each point $\xi $ of a variety $X$. If $X$ is defined over $\mathbb{C}$ then the multiplicity at $\xi $ is the smallest rank of the generic fiber for all possible local morphisms $(X,\xi )\longrightarrow (\mathbb{C}^{d},0)$. If $X$ is a reduced equidimensional scheme, then $X$ is regular if and only if the multiplicity equals one at every point.

\subsubsection*{Constructive resolution of singularities}
In short, a \textit{constructive resolution of singularities} of $X$ is given by an upper semi-continuous function 
$$f:X\longrightarrow (\Lambda ,\geq )\mbox{,}$$
where $(\Lambda ,\geq )$ is some well ordered set. The maximum value of $f$  determines the first smooth center $C\subset X$ to blow up: $X\stackrel{\pi _1}{\longleftarrow} X_1$.
Right after this blow up, a new upper semi-continuous function $f_1:X_1\longrightarrow (\Lambda ,\geq )$ is defined, in such a way that $f_1(\pi _1 ^{-1}(\xi ))=f(\xi )$ for any $\xi \in X\setminus C$, and $f_1(\xi ')<f(\xi )$ whenever $\pi _1(\xi ')=\xi \in C$. If $f$ is appropiately constructed so that it is constant if and only if $X$ is smooth, then resolution is achieved after a finite number of steps. One way to construct such a function is to associate a string of invariants to each point.\\ 
\\
Looking at the multiplicity function on $X$ may be a good starting point when attempting to construct a resolution of singularities of $X$. But unfortunately, the strata defined by the multiplicity function may be non smooth. Thus, the use of other invariants becomes necessary in order to refine the stratification so that one can have a smooth stratum to choose as the center of the first blow up. The most important of these invariants is the so called Hironaka's \textit{order function} (see \cite{E_V97} or Definition \ref{def:order} in this paper). From it, many other invariants may be defined (see Section \ref{subsec:alg_res}). If we choose the multiplicity function as the first coordinate of $f$, $C$ is contained in $\mathrm{\underline{Max}\; mult}(X)$, the closed subset of $X$ where the multiplicity reaches its highest value. Now fix some point $\xi \in \mathrm{\underline{Max}\; mult}(X)$. Locally, in a neighbourhood of $\xi $,  a finite local projection $p$ to some smooth scheme   of dimension $d=\mathrm{dim}(X)$ can be defined, inducing a bijection between $\mathrm{\underline{Max}\; mult}(X)$ and its image (see \cite{Br_V}, \cite{Br_V1}). There, Hironaka's order function can be defined at each point in the image of $\mathrm{\underline{Max}\; mult}(X)$. This function, which we will for the moment denote by $\mathrm{ord}^{(d)}_{\xi }(X)$, does not depend on the projection (if it is general enough). Moreover, it can be shown that lowering the maximum multiplicity in a neighbourhood of $\xi $ is equivalent to solving a suitable problem in a $d$-dimensional smooth scheme. This gives the possibility of constructing a resolution of singularities of $X$ by resolving such problems, which simplifies the process. We will refer to $\mathrm{ord}^{(d)}_{\xi }(X)$ as \textit{Hironaka's order function in dimension $d$} (see Section \ref{subsec:elim} for full details). It can be shown that $\mathrm{ord}^{(d)}_{\xi }(X)$ is the next relevant coordinate of $f$, refining the stratum $\mathrm{\underline{Max}\; mult}(X)$ (see Section \ref{subsec:alg_res}), so we will consider
$$f(\xi )=(\mathrm{mult}_{\xi }(X),\mathrm{ord}^{(d)}_{\xi }(X),\ldots )\mbox{.}$$
Surprisingly, $\mathrm{ord}^{(d)}_{\xi }(X)$ can readily be read by looking at a suficiently general \textit{arc} in ${\mathcal L}(X)$, as our main result, Theorem \ref{thm:main_order-r}, shows.

\subsubsection*{Arcs}
There are many other approaches to the study of singularities. Jet and arc spaces of varieties often appear among them. Many properties of the jet schemes and the arc scheme of a variety are linked to its singularities. See for instance the works of Ein, Ishii, Musta\c{t}\u{a} and Yasuda where some singularity types are characterized through topological or geometrical properties of the associated arc schemes (\cite{Mus1}, \cite{Mus2}, \cite{E_M_Y}, \cite{E_M}, \cite{E_M2}, \cite{I}).\\
\\
It is in this context of arc spaces where the \textit{Nash multiplicity sequence} appears. It was defined by M. Lejeune-Jalabert \cite{L-J} as a non-increasing sequence of positive integers attached to a germ of a curve inside a germ of a hypersurface. M. Hickel generalized this notion to arbitrary codimension \cite{Hickel05} and defined a sequence of blow ups (at points) that allows us to compute Nash multiplicity sequences and study their behaviour. Given a variety $X$, fix an arc through a point $\xi \in \mathrm{\underline{Max}\; mult}(X)$ (not necessarily closed). By means of its graph $\Gamma \subset X\times \mathbb{A}^1$, the arc $\varphi $ defines a sequence of blow ups at points:
\begin{equation*}
\xymatrix@R=0pt@C=30pt{
X_0=X\times \mathbb{A}^1 & X_1 \ar[l]_>>>>>{\pi _1} & \ldots \ar[l]_{\pi _2} & X_r \ar[l]_{\pi _r}\mbox{,}\\
\xi _0=(\xi ,0) & \xi _1 & \ldots & \xi _r 
}\end{equation*}
where $\xi _i$ is the intersection of the exceptional divisor of $\pi _i$ and the strict transform of the graph $\Gamma $ in $X_i$ for $i=1,\ldots ,r$. The Nash multiplicity sequence is then the sequence
$$m_0\geq m_1\geq \ldots \geq m_r\geq 1\mbox{,}$$
in which $m_i$ is the multiplicity of $X_i$ at $\xi _i$ for $i=0,\ldots ,r$ (see Section \ref{subsec:RAandNash} for details).

\subsubsection*{Our results}
In this work, we analyze a connection of arc spaces with the problem of resolution of singularities. We study the Nash multiplicity sequence for arcs in varieties, and find a relation between the structure of this sequence and some invariants of resolution. In particular, for an algebraic variety $X$ of dimension $d$, we are in position to give a relation between the length $\rho _{X,\varphi }$ of the first step of the sequence (before the Nash multiplicity decreases for the first time) and Hironaka's order function in dimension $d$. We introduce an invariant for $X$ and $\varphi $ at $\xi $ which is sharper than $\rho _{X,\varphi }$ and which we will denote by $r_{X,\varphi }$. More precisely, we will see that $\rho _{X,\varphi }=\left[ r_{X,\varphi }\right] $. For this invariant, we prove the following result:\\
\\
\textbf{\textit{Main Theorem (\ref{thm:main_order-r}):}} Let $X$ be a variety of dimension $d$. Let $\xi $ be a point in $\mathrm{\underline{Max}\; mult}(X)$. Then, 
$$\mathrm{ord}_{\xi }^{(d)}(X)=\min _{\varphi }\left\{ \frac{r_{X,\varphi }}{\mathrm{ord}(\varphi )}\right\} \mbox{,}$$
where $\varphi $ runs through all arcs in $X$ through $\xi $.\\
\\

As we mentioned before, this minimum is achieved for any arc which is generic enough with respect to the tangent cone of $X$ at $\xi $.\\
\\
When we work with a hypersurface $X$, computing invariants and giving a local expression of the equation of $X$ is much easier than when we deal with a variety of higher codimension. To avoid this difficulty, we rely on the results on \textit{local presentations} attached, in this case, to the multiplicity (see \cite{V}). They allow us to work locally with a set of hypersurfaces with weights.\\
\\
\textit{Rees algebras} happen to provide a useful tool for the study of these local presentations and their behaviour under blow ups. They keep track locally of the behaviour of resolution functions before and after blowing up at smooth centers. We will also see that our problem can be translated into a problem of resolution of Rees algebras.\\
\\
Our work is organized as follows. In section 1, we present some preliminary definitions and results on Rees algebras, as well as some examples motivating their use and their connection to algorithmic resolution. We also include some comments about the resolution invariants we want to focus on. Section 2 is devoted to arcs and the Nash multiplicity sequence. It is in section 3 where we finally connect all the previous concepts and state our main result (Theorem \ref{thm:main_order-r}). The proof of the main result is given in section 4 where we first prove it in the simpler case of a hypersurface. Then we deduce the general case from this one, making use of what we know from \cite{V} about local presentations attached to the multiplicity.

\subsubsection*{Acknowledgements}
The authors profited from conversations with F. Aroca, R. Docampo, S. Ishii and H. Kawanoue, as well as from the support of their research team. We also thank C. Abad, J. Conde-Alonso and I. Efraimidis for their help with the presentation,  and the referee for many useful suggestions.

\section{Rees algebras and their use in resolution of singularities}
\subsection{Rees algebras}
\begin{Def}
Let $R$ be a Noetherian ring. A \textit{Rees algebra $\mathcal{G}$ over $R$} is a graded ring\footnote{$W$ is just a variable in charge of the degree of the $I_i$.}, that is:
$$\mathcal{G}=\bigoplus _{l\in \mathbb{N}}I_{l}W^l\subset R[W]$$
for some ideals $I_i\in R$, $i\in \mathbb{N}$ such that $I_0=R$ and $I_iI_j\subset I_{i+j}\mbox{,\; } \forall i,j\in \mathbb{N}$, which is also a finitely generated $R$-algebra. That is, there exist some $f_1,\ldots ,f_r\in R$ and weights $n_1,\ldots ,n_r\in \mathbb{N}$ such that
\begin{equation}\label{def:Rees_alg_generadores}
\mathcal{G}=R[f_1W^{n_1},\ldots ,f_rW^{n_r}]\mbox{.}
\end{equation}
\end{Def}

\begin{Rem}
Rees algebras can be defined  over a Noetherian scheme $V$  in the obvious manner, that is, $\mathcal{G}$ will be locally at each $\xi \in V$ as in (\ref{def:Rees_alg_generadores}), with $\mathrm{Spec}(R)\subset V$ an open affine subset.
\end{Rem}

\begin{Def}\label{def:amalgama_op}
Let $\mathcal{G}_1$ and $\mathcal{G}_2$ be two Rees algebras. We denote by $\mathcal{G}_1\odot \mathcal{G}_2$ the smallest Rees algebra containing both of them. If $\mathcal{G}_1=R[f_1W^{n_1},\ldots ,f_rW^{n_r}]$ and $\mathcal{G}_2=R[g_1W^{m_1},\ldots ,g_lW^{m_l}]$, then $\mathcal{G}_1\odot \mathcal{G}_2=R[f_1W^{n_1},\ldots ,f_rW^{n_r},g_1W^{m_1},\ldots ,g_lW^{m_l}]$. If $\mathcal{G}_2'=R'[g_1W^{m_1},\ldots ,g_lW^{m_l}]$, where $R'\subset R$ is a subring, by abuse of notation we will sometimes denote by $\mathcal{G}_1\odot \mathcal{G}_2'$ the Rees algebra $\mathcal{G}_1\odot \mathcal{G}_2$, where $\mathcal{G}_2$ is the extension of $\mathcal{G}_2'$ to a Rees algebra over $R$.
\end{Def}

\begin{Par}\textbf{Notations and Conventions}
From now on we will assume $k$ to be a field of characteristic zero, unless otherwise stated. We will also assume $R$ to be a smooth  $k$-algebra, or $V$ to be a smooth scheme     over  $k$.
\end{Par}

\begin{Def}
Let $\mathcal{G}$ be a Rees algebra over $R$. The \textit{singular locus} of $\mathcal{G}$, Sing$(\mathcal{G})$, is the closed set given by all the points $\xi \in \mathrm{Spec}(R)$ such that $\nu _{\xi }(I_l)\geq l$, $\forall l\in \mathbb{N}$.\footnote{Here $\nu _{\xi}(I)$ denotes the order of the ideal $I$ in the regular local ring $R_{\mathcal{M}_{\xi }}$, where $\mathcal{M}_{\xi }$ is the ideal defining the point $\xi $.} Equivalently, if $\mathcal{G}=R[f_1W^{n_1},\ldots ,f_rW^{n_r}]$, then it can be shown (\cite[Proposition 1.4]{E_V}) that
$$\mathrm{Sing}(\mathcal{G})=\left\{ \xi \in \mathrm{Spec}(R):\, \nu _{\xi }(f_i)\geq n_i,\; \forall i=1,\ldots ,r\right\} \mbox{.}$$
Note that the singular locus of the Rees algebra over $V$ generated by $f_1W^{n_1}, \ldots ,f_rW^{n_r}$ does not coincide with the usual definition of the singular locus of the subvariety of $V$ defined by $f_1,\ldots ,f_r$.
\end{Def}
We will sometimes refer to the singular locus of a Rees algebra as the \textit{closed set attached to} it.

\begin{Def}\label{def:order}
We define the \textit{order of an element  $fW^{n}\in \mathcal{G}$ at $\xi \in \mathrm{Sing}(\mathcal{G})$} as
$$\mathrm{ord}_{\xi }(fW^{n})=\frac{\nu _{\xi }(f)}{n}\mbox{.}$$
We define the \textit{order of the Rees algebra $\mathcal{G}$ at $\xi \in \mathrm{Sing}(\mathcal{G})$} as the infimum of the orders of the elements of $\mathcal{G}$ at $\xi$, that is
$$\mathrm{ord}_{\xi }(\mathcal{G})=\inf _{fW^{n}\in \mathcal{G}}\left\{ \mathrm{ord}_{\xi }(fW^{n})\right\} \mbox{.}$$
\end{Def}

Actually, one could define the order of an ${\mathcal O}_V$-Rees algebra ${\mathcal G}$ at any point $\xi \in V$, but for our purposes, only the order at points in $ \mathrm{Sing}(\mathcal{G})$ will be needed. 

\begin{Thm}\cite[Proposition 6.4.1]{E_V} Let $\mathcal{G}=R[f_1W^{n_1},\ldots ,f_rW^{n_r}]$ be a Rees algebra and let
 $\xi\in \mathrm{Sing}(\mathcal{G})$. Then
$$\mathrm{ord}_{\xi }(\mathcal{G})=\min _{i=1\ldots r}\left\{ \mathrm{ord}_{\xi }(f_iW^{n_i})\right\} \mbox{.}$$
\end{Thm}

\begin{Def}
Let $\mathcal{G}$ be a Rees algebra over $R$. Let $\mathcal{P}\subset R$ be a prime ideal. We say that $\mathcal{P}$ is a \textit{permissible center for $\mathcal{G}$} if $R/\mathcal{P}$ is a regular ring and $\nu _{\mathcal{P}}(\mathcal{G})\geq 1$. That is, $\mathcal{P}$ is permissible for $\mathcal{G}$ if it defines a smooth closed set in $\mathrm{Spec}(R)$ which is also contained in $\mathrm{Sing}(\mathcal{G})$. If $\mathcal{G}$ is a Rees algebra over $V$, a closed set $Y\subset V$ is a permissible center for $\mathcal{G}$ if it is a regular subvariety contained in $\mathrm{Sing}(\mathcal{G})$.\\
\end{Def}

\begin{Def}\label{def:transf_law}\cite[Definition 6.1]{V07}
Let $\mathcal{G}$ be a Rees algebra on $V$. A \textit{$\mathcal{G}$-permissible transformation}
$$V\stackrel{\pi }{\leftarrow} V_1\mbox{,}$$
is the blow up of $V$ at a permissible center $Y\subset V$. We denote then by $\mathcal{G}_1$ the transform of $\mathcal{G}$ by $\pi $, which is defined as
$$\mathcal{G}_1:=\bigoplus _{l\in \mathbb{N}}I_{l,1}W^l\mbox{,}$$
where 
\begin{equation}\label{eq:transf_law}
I_{l,1}=I_l\mathcal{O}_{V_1}\cdot I(E)^{-l}
\end{equation}
for $l\in \mathbb{N}$ and $E$ the exceptional divisor of the blow up $V\longleftarrow V_1$.
\end{Def}

\begin{Def}\label{def:res_RA}
Let $\mathcal{G}$ be a Rees algebra over $V$. A \textit{resolution of $\mathcal{G}$} is a finite sequence of  blow ups, 
\begin{equation}\label{diag:res_Rees_algebra}
\xymatrix@R=0pt@C=30pt{
V=V_0 & V_1 \ar[l]_>>>>>{\pi _1} & \ldots \ar[l]_{\pi _2} & V_l \ar[l]_{\pi _l}\\
\mathcal{G}=\mathcal{G}_0 & \mathcal{G}_1 \ar[l] & \ldots \ar[l] & \mathcal{G}_l \ar[l]
}\end{equation}
at permissible centers $Y_i\subset \text{Sing} ({\mathcal G}_i)$, $i=0,\ldots, l-1$, such that $\mathrm{Sing}(\mathcal{G}_l)=\emptyset$,  and the exceptional divisor of the composition $V_0\longleftarrow V_l$ is a union of hypersurfaces with normal crossings. Recall that  a set of hypersurfaces $\{H_1,\ldots, H_r\}$ in a smooth $n$-dimensional $V$  has normal crossings at a point $\xi\in V$ if there is a regular system of parameters $x_1,\ldots, x_n\in {\mathcal O}_{V, \xi}$ such that if $\xi\in H_{i_1}\cap \ldots \cap H_{i_s}$,  and $\xi \notin H_l$ for $l\in \{1,\ldots, r\}\setminus \{i_1,\ldots,i_s\}$, then
 ${\mathcal I}(H_{i_j})_{\xi}=\langle x_{i_j}\rangle$  for $i_j\in \{i_1,\ldots, i_s\}$;
we say that $H_1,\ldots, H_r$ have normal crossings in V if they have normal crossings at each point of $V$. 
\end{Def}

 In \cite{Hir}, H. Hironaka proves resolution of singularities of varieties over fields of characteristic zero by showing that the maximum value of the Hilbert Samuel function can be lowered after a sequence of blow ups at suitable regular centers. To this end, he proceeds as follows. Let $X$ be an algebraic variety over a (perfect) field $k$, let $\text{max  HS}(X)$ be the maximum value of the Hilbert Samuel function on $X$,  let $\underline{\text{Max}} \text{ HS}(X)$ be the maximum stratum of this function, and let $\xi \in \underline{\text{Max}} \text{ HS}(X)$. Then in  some  (\'etale) neighborhood of $\xi$ there is an immersion of $X$ in some smooth $V$ and a Rees algebra ${\mathcal G}$ strongly attached to  $\underline{\text{Max}} \text{ HS}(X)$ (see Example \ref{ex:HS_repr_local_pres}  below; see also  \cite{Hir1}).    Then he shows that 
 a resolution of ${\mathcal G}$ induces a sequence of blow ups over X that ultimately leads to a lowering of $\text{max  HS}(X)$. To conclude, he  proves that such resolution exists when the characteristic is zero:

\begin{Thm}\label{thm:existence_res_Rees_algebras}\cite{Hir}
Let $k$ be a field of characteristic zero, and let $R$ be a smooth $k$-algebra. Given a Rees algebra $\mathcal{G}$ over $R$, there exists a resolution of $\mathcal{G}$.
\end{Thm}

The previous result is existencial. The following theorem says that, in fact, resolution of Rees algebras can be constructed; i.e., given a Rees algebra ${\mathcal G}$ in a smooth $V$ defined over a field of characteristic zero, there is a procedure that indicates how to actually construct a sequence of blow ups that leads to a resolution. 
See also \cite{V1} and \cite{B-M}.

\begin{Thm}\label{thm:res_Rees_algebras}\cite[Theorem 3.1]{E_V97}
Let $k$ be a field of characteristic zero, and let $R$ be a smooth $k$-algebra. Given a Rees algebra $\mathcal{G}$ over $R$, it is possible to construct a resolution of $\mathcal{G}$.
\end{Thm}

For more details about transformations and resolution of Rees algebras, we refer to \cite{E_V97} and \cite{Br_V2}.

\begin{Rem}
To construct a resolution of $\mathcal{G}$, we use the so called \textit{resolution invariants}. The most important resolution invariant is Hironaka's order function, $\mathrm{ord}_{\xi }\mathcal{G}$, at a point $\xi \in \mathrm{Sing}(\mathcal{G})$ (\cite{Hir1}). All other invariants derive from it (see Section \ref{subsec:alg_res} and \cite[9, 14, 18]{Br_V2}).
\end{Rem}

\begin{Rem}
For some purposes,  during the resolution, one may need to keep track of more information  than that given by the Rees algebra itself. We refer to $(V^{(n)},\mathcal{G}^{(n)})$ as a \textit{pair}, where  $V^{(n)}$ is an $n$-dimensional smooth scheme of finite type, and $\mathcal{G}^{(n)}$ a Rees algebra over $V^{(n)}$. We understand by \textit{basic object} a triple $(V^{(n)},\mathcal{G}^{(n)},E)$, where $(V^{(n)},\mathcal{G}^{(n)})$ is a pair and $E$ is a set of smooth hypersurfaces in $V^{(n)}$ (possibly empty) so that their union has normal crossings. For more details and the definition of transformations and resolution of pairs and basic objects, we refer to \cite{E_V97}.
\end{Rem}

\subsection{Motivation I}
In general, Rees algebras represent a very interesting tool, since many problems in resolution of singularities can be codified by them. We mention here a few examples that may help getting an overall impression of their use.\\

\begin{Ex}\label{ex:resol_hypersurfaces_via_resol_RA} \textbf{Resolution of singularities of a hypersurface:} Consider a hypersurface $X\subset V$. Then $I(X)$ is locally principal. Set $\mathcal{G}=\mathcal{O}_{V}[I(X)W^b]$, where $b$ is the maximum multiplicity of $X$ (see Example \ref{ex:mult}), which we will denote by $\mathrm{max\, mult}(X)$. A resolution of $\mathcal{G}$ as (\ref{diag:res_Rees_algebra}) gives a simplification of the points of multiplicity $b$ of $X$, that is, the induced sequence $X\longleftarrow X_l$ will be such that $\mathrm{max\; mult}(X_l)<b$. One can resolve the singularities of $X$ by iterating this process until $X_r$ is such that $\mathrm{max\; mult}(X_r)=1$.\\
\end{Ex}

\begin{Ex}\label{ex:resol_via_resol_RA}\textbf{Resolution of $\mathcal{G}=\mathcal{O}_{V}[I(X)W]$:} Let $V$ be a smooth scheme over a field of characteristic zero. Let now $X\subset V$ be a closed reduced equidimensional subscheme, defined by $I(X)\subset \mathcal{O}_V$. Let $\mathcal{G}=\mathcal{O}_{V}[I(X)W]$. By Theorem \ref{thm:res_Rees_algebras}, one can construct a resolution of Rees algebras for $\mathcal{G}$:
\begin{equation}
\xymatrix@R=0pt@C=30pt{
V=V_0 & V_1 \ar[l]_>>>>>{\pi _1} & \ldots \ar[l]_{\pi _2} & V_r \ar[l]_{\pi _r}\\
\mathcal{G}=\mathcal{G}_0 & \mathcal{G}_1 \ar[l] & \ldots \ar[l] & \mathcal{G}_r \ar[l]
}\end{equation}
such that $\mathrm{Sing}(\mathcal{G}_r)=\emptyset $, and so that the exceptional locus of $V\longleftarrow V_r$ is a union of smooth hypersurfaces with normal crossings. Let us show now how a resolution of singularities of $X$ can be obtained: For any $i\in \{ 1,\ldots ,r\} $, the transform $I(X)^{(i)}$ of $I(X)$ in $\mathcal{O}_{V_i}$, defined by $I(X)^{(i)}:=I_{1,i}$ as in (\ref{eq:transf_law}), is supported in the exceptional locus (which has normal crossings) as well as in the strict transform of $X$ by $V\longleftarrow V_i$. The condition $\mathrm{Sing}(\mathcal{G}_r)=\emptyset $ implies that the maximum order of $I(X)^{(r)}$ is less than one, so for some $j\in \{ 1,\ldots ,r\} $, the strict transform $X_{j-1}$ of $X$ in $V_{j-1}$ is a connected component of the center of the transform $\pi _j$, and hence is permissible. In particular, this implies that $X_{j-1}$ is regular and has normal crossings with the exceptional divisor. Therefore
\begin{equation}
\xymatrix@R=0pt@C=30pt{
V=V_0 & V_1 \ar[l]_>>>>>{\pi _1} & \ldots \ar[l]_{\pi _2} & V_j \ar[l]_{\pi _j}\\
\cup \; \; \; \; \; \cup & \cup & & \cup \\
X=X_0 & X_1 \ar[l] & \ldots \ar[l] & X_j \ar[l]
}\end{equation}
is a resolution of singularities of $X$ (see \cite[proof of Theorem 1.5]{E_V03} for a precise proof of this result in the language of basic objects).\\
\end{Ex}

\begin{Ex}\textbf{Log-resolution of ideals:} A \textit{Log-resolution} of an ideal $I$ on a smooth scheme $V$ is a proper birational morphism of smooth schemes, say $V'\longrightarrow V$, so that the total transform of $I$, $I\mathcal{O}_{V'}$, is an invertible ideal in $V'$ supported on smooth hypersurfaces having only normal crossings. A resolution of $\mathcal{G}=R[IW]$ gives a Log-resolution of $I$. In \cite{E_V}, Encinas and Villamayor proved, by using Rees algebras, that for two ideals with the same integral closure, one obtains the same algorithmic Log-resolution.\\
\end{Ex}

In this work, we use Rees algebras to give an answer to a problem of computing a sequence of multiplicities. As we will see, we translate our problem into a resolution of some specific Rees algebras (see Section \ref{par:def_Nash_m}).\\

\subsection{Motivation II: local presentations}

When one tries to study certain closed subsets of a variety $X$, one often needs to consider some equations $\{ f_1,\ldots ,f_r\} \subset R$ with weights $\{ n_1,\ldots ,n_r\} \subset \mathbb{Z}_{>0}$ that describe the closed set in question:
$$C=\cap _{i=1}^{r} \{ \eta \in V: \nu _{\eta }(f_i)\geq n_i\} \mbox{,}$$
in a way that the expression is stable under blow ups at suitably chosen centers.
We call such a representation a \textit{local presentation}. Example \ref{ex:resol_hypersurfaces_via_resol_RA} is a particular case of this representation. Let us see another example:\\

\begin{Ex}\label{ex:HS_repr_local_pres} Let $X$ be a variety over a perfect field $k$. Let $\text{HS}(X)$ be the Hilbert-Samuel function on $X$. 
This is an upper semicontinuous function\footnote{Actually, the Hilbert-Samuel function has to be modified in order to be semicontinuous (see \cite{Bennett}). Here we refer to this modification of the Hilbert-Samuel function.}  on $X$, 
$$\begin{array}{rrcl} 
\text{HS}(X): & X & \longrightarrow & ({\mathbb N}^{\mathbb N}, \leq )\\
  & \xi & \mapsto & \text{HS}(X)(\xi),
  \end{array}$$
where ${\mathbb N}^{\mathbb N}$ is ordered lexicographically.
Let $\mathrm{max\, HS}(X)$ and $\mathrm{\underline{Max}\, HS}(X)$ denote the maximum value of $\text{HS}(X)$ in $X$ and the closed subset of points where $\text{HS}(X)$ reaches this value respectively. Pick $\xi \in \mathrm{\underline{Max}\, HS}(X)$. Then (see \cite{Hir1}), it is possible to find, locally in an \'etale neighbourhood of $\xi $, an immersion of $X$ in a smooth scheme $V$ and equations $f_1,\ldots ,f_r$ such that $I(X)=<f_1,\ldots ,f_r>$,
$$\mathrm{\underline{Max}\, HS}(X)=\cap _{i=1}^{r}\mathrm{\underline{Max}\, HS}(\{f_i=0\})\mbox{,}$$
and such that this condition is preserved by blow ups with smooth centers in $\mathrm{\underline{Max}\, HS}(X)$ and by smooth morphisms, in terms of the strict transforms of $X$ and of the $f_i$. Let us translate this  it into  the language of Rees algebras: let $\mathcal{G}=\mathcal{O}_{V,\xi }[f_1W^{\mu _1},\ldots ,f_rW^{\mu _r}]$, where $\mu _i$ is the maximum order of $f_i$ for $i=1,\ldots ,r$. Then
$$\mathrm{Sing}(\mathcal{G})=\mathrm{\underline{Max}\, HS}(X)\mbox{,}$$
and for any sequence of $\mathcal{G}$-permissible transformations
\begin{equation}
\xymatrix@R=0pt@C=30pt{
V=V_0 & V_1 \ar[l]_>>>>>{\pi _1} & \ldots \ar[l]_{\pi _2} & V_l \ar[l]_{\pi _l}\\
\mathcal{G}=\mathcal{G}_0 & \mathcal{G}_1 \ar[l] & \ldots \ar[l] & \mathcal{G}_l \ar[l]\mbox{,}
}\end{equation} 
we have 
$$\mathrm{Sing}(\mathcal{G}_l)=\mathrm{\underline{Max}\, HS}(X_l)\mbox{.}$$
Resolving the Rees algebra $\mathcal{G}$ is equivalent to making $\mathrm{max\, HS}(X)$ decrease.\\
\\
\end{Ex}

The previous example shows that Rees algebras appear as an appropriate language to represent such a set of equations and weights, and allow us to describe the transformations on the subset $C$ we are interested in via well defined transformations of the associated Rees algebra (see (\ref{eq:transf_law})). It is very important to understand to which extent a given algebra can represent $C$. For this purpose we will consider the following transformations:

\begin{Def}\label{def:loc_seq_V}
A \textit{local sequence on} a variety \textit{$V$} is a sequence of morphisms
$$V=V_0\stackrel{\phi _1}{\longleftarrow} V_1\stackrel{\phi _2}{\longleftarrow}\ldots \stackrel{\phi _l}{\longleftarrow} V_l$$
where each $\phi _i$ is either a blow up at a regular center or a smooth morphism, such as   an open inmersion or a projection from a product by some affine space (see for example (\ref{diag:prod_line})).
\end{Def}

\begin{Def}\label{def:loc_seq_G}
Let $\mathcal{G}$ be a Rees algebra over $\mathcal{O}_V$. A \textit{$\mathcal{G}$-local sequence over $V$} is a local sequence over $V$ as in Definition \ref{def:loc_seq_V},
\begin{equation}\label{diag:seq_transf}
\xymatrix@R=0pt@C=30pt{
V=V_0 & V_1 \ar[l]_>>>>>{\phi _1} & \ldots \ar[l]_{\phi _2} & V_l \ar[l]_{\phi _l}\\
\mathcal{G}=\mathcal{G}_0 & \mathcal{G}_1 \ar[l]_>>>>>{\psi _1} & \ldots \ar[l]_{\psi _2} & \mathcal{G}_l \ar[l]_{\psi _l}\mbox{,}}
\end{equation}
such that whenever $\phi _i$ is a blow up, it is in particular a blow up at a permissible center $Y_{i-1}\subset \mathrm{Sing}(\mathcal{G}_{i-1})\subset V_{i-1}$, and then $\mathcal{G}_i$ is the transform of $\mathcal{G}_{i-1}$ by the rule in Definition \ref{def:transf_law}; if $\phi _i$ is a smooth morphism, then $\mathcal{G}_i$ is the pullback of $\mathcal{G}_{i-1}$ by $\phi _i$ (see \cite[Definition 3.2]{Br_G-E_V}).
\end{Def}

\begin{Def}\label{def:tree} Let $\mathcal{G}$ be a Rees algebra over $V$, and consider a $\mathcal{G}$-local sequence over $V$ as in (\ref{diag:seq_transf}). This sequence determines a collection of closed sets, namely $\left\{ \mathrm{Sing}(\mathcal{G}), \mathrm{Sing}(\mathcal{G}_1),\ldots ,\mathrm{Sing}(\mathcal{G}_l)\right\} $. We will refer to this collection (or \textit{branch}) of closed sets as the one defined by or attached to the sequence (\ref{diag:seq_transf}). If we consider all possible $\mathcal{G}$-local sequences over $V$, we obtain a \textit{tree of closed sets} for $\mathcal{G}$, which we denote by $\mathcal{F}_V(\mathcal{G})$ (see \cite[Section 3]{Br_G-E_V}).\\
\end{Def}

For the next examples, let us recall a few concepts and notations:\\

\begin{Notation}
Let $F$ be an upper semicontinuous function defined on varieties, that is, for each variety, there is
\begin{equation}\label{eq:up_semic_funct}F(X)=F_X:X\longrightarrow (\Lambda , \geq ) \mbox{,}\end{equation}
where $(\Lambda ,\geq )$ is a well ordered set. We will denote by max $F(X)$ the maximum value achieved by $F_X$ in $X$. We will use $\underline{\mathrm{Max\, }}F(X)$ to denote the set of points of $X$ in which $F$ achieves this maximum value, that is:
$$\underline{\mathrm{Max\, }}F(X)=\{ \eta \in X:F_X(\eta )\geq \mathrm{max\, }F(X)\} =\{ \eta \in X:F_X(\eta )=\mathrm{max\, }F(X)\} \mbox{.}$$
Note that $\underline{\mathrm{Max\, }}F(X)$ is a closed set.
\end{Notation}

\begin{Def}\label{def:F-loc_seq}
Let $F$ be an upper semicontinuous function defined on varieties. An \textit{$F_X$-local sequence} is a local sequence on $X$ (Definition \ref{def:loc_seq_V}) such that, whenever $\phi _i$ is a blow up, the center is contained in $\underline{\mathrm{Max\, }}F_{X_{i-1}}$.
\end{Def}

\begin{Def}\label{representable}(see \cite[Definition 28.4]{Br_V2})
An upper semicontinuous function $F$ defined on varieties as (\ref{eq:up_semic_funct}) is said to be \textit{representable via local embeddings} if, for each $X$ and each $\xi \in X$, in an \'etale neighbourhood of $\xi $, we can find a closed immersion $X\hookrightarrow V$ and a Rees algebra $\mathcal{G}$ over $\mathcal{O}_{V,\xi }$ such that
\begin{enumerate}
	\item the Rees algebra $\mathcal{G}$ satisfies:
	\begin{equation}\label{eq:condition_G_X:F}
\mathrm{Sing}(\mathcal{G})=\mathrm{\underline{Max}}\, F_X\mbox{;}
\end{equation}
	\item any $F_X$-local sequence 
	\begin{equation}\label{loc_seq_F}X=X_0\leftarrow X_1\leftarrow \ldots \leftarrow X_r\end{equation}
	such that 
	\begin{equation}\label{cond_loc_sec_F}m=\mathrm{max\, }F_X=\mathrm{max\, }F_{X_1}=\ldots =\mathrm{max\, }F_{X_{r-1}}\geq \mathrm{max\, }F_{X_r}\end{equation}
	induces a $\mathcal{G}$-local sequence of Rees algebras over $V$
\begin{align*}
V= & \, V_{0}\leftarrow V_{1}\leftarrow \ldots \leftarrow V_{r}\\
X= & \, X_{0}\leftarrow X_{1}\leftarrow \ldots \leftarrow X_{r}\\
\mathcal{G}= & \, \mathcal{G}_{0}\leftarrow \mathcal{G}_{1}\leftarrow \ldots \leftarrow \mathcal{G}_{r}
\end{align*}
such that for $i=1,\ldots ,r$, 
	$$\mathrm{Sing}(\mathcal{G}_{i})=\left\{ \eta \in X_i:F_{X_i}(\eta )=m\right\} \mbox{,}$$
	 with  $\mathrm{Sing}(\mathcal{G}_r)=\emptyset$ if  and only if $\mathrm{max \; }F_{X_r}<m$ and 
	\item any $\mathcal{G}$-local sequence over $V$ induces an $F_X$-local sequence as (\ref{loc_seq_F}) satisfying (\ref{cond_loc_sec_F}).\\
	\\
\end{enumerate}

\end{Def}

\begin{Ex}\label{ex:HS}\textbf{Hilbert-Samuel}
The results of Hironaka (\cite{Hir}, \cite{Hir1}) show that it is possible to resolve the singularities of a variety (over a perfect field) if we know how to lower the maximum value of the Hilbert-Samuel function of the variety through a finite sequence of blow ups. Then, to construct a resolution of the singularities of a given variety $X$, one just needs to iterate the process a finite number of times.\\
\\
On the other hand, the Hilbert-Samuel function is upper semicontinuous, and it is representable for any variety $X$ via local embeddings (see \cite{Hir1} and Example \ref{ex:HS_repr_local_pres}). Thus, for each point $\xi \in X$ we can find, in an \'etale neighbourhood of $\xi $, an immersion of $X$ into a smooth scheme $V$ and an $\mathcal{O}_{V,\xi }$-Rees algebra $\mathcal{G}_X$ such that $\mathrm{Sing}(\mathcal{G}_{X})=\mathrm{\underline{Max}\, HS}(X)$ and this identity is preserved by $\mathcal{G}$-local sequences over $V$ as long as the maximum value of the Hilbert-Samuel function of $X$ does not decrease. From this, it will follow that finding a sequence of blow ups
$$X=X_0\leftarrow X_1\leftarrow \ldots \leftarrow X_r$$ 
such that $\mathrm{max\, HS}(X_0)=\ldots =\mathrm{max\, HS}(X_{r-1})>\mathrm{max\, HS}(X_r)$ is equivalent to finding a resolution of the Rees algebra $\mathcal{G}_X$.\\
\\
\end{Ex}

A similar statement holds for the multiplicity of a variety defined over a perfect field, see Example \ref{ex3} and \cite{V}:\\

\begin{Ex}\label{ex:mult}\textbf{Multiplicity} 
The multiplicity of an equidimensional variety $X$ at a point $\eta \in X$ is given by an upper semicontinuous function
\begin{align*}
\mathrm{mult} (X): X & \longrightarrow \mathbb{N} \\
\eta & \longmapsto \mathrm{mult} (X)(\eta )=\mathrm{mult}_{\eta }(X)=\mathrm{mult}(\mathcal{O}_{X,\eta })\mbox{}
\end{align*}
where $\mathrm{mult}(\mathcal{O}_{X,\eta })$ stands for the multiplicity of the local ring at the maximal ideal $\mathcal{M}_{\eta }$.
Let $m$ be the maximum multiplicity of $X$. The set
$$\mathrm{\underline{Max}\, mult}(X)=\left\{ \eta \in X:\mathrm{mult}_{\eta }(X)\geq m\right\} =\left\{ \eta \in X:\mathrm{mult}_{\eta }(X)=m\right\} $$
is closed, and the multiplicity is representable via local embeddings for $X$ (see \cite[Proposition 5.7 and Theorem 7.1]{V}).\\
\\
Therefore, just as for the Hilbert-Samuel function in Example \ref{ex:HS}, we can attach a Rees algebra $\mathcal{G}$ to $\mathrm{mult}(X)$ so that resolving $\mathcal{G}$ is equivalent to decreasing the maximum value of $\mathrm{mult}(X)$.
\end{Ex}

By Theorem \ref{thm:res_Rees_algebras}, the resolution for such an algebra can be constructed whenever $k$ is a field of characteristic zero. It is not known if this is true for fields of positive characteristic.

\subsection{Equivalence of Rees algebras}\label{subsec:eq_RA}

Given an upper semicontinuous function $F$ as in (\ref{eq:up_semic_funct}) which is representable via local embeddings, the choice of a Rees algebra satisfying the properties of Definition \ref{representable} is not unique. To begin with, for a given $X$, there are many possible choices for the immersion $X\hookrightarrow V$, but we will mention this problem later. On the other hand, once an immersion is fixed, we can attach a different Rees algebra to a neighbourhood of each point $\xi \in X$. This choice is not unique either. Therefore, given two possible choices of Rees algebras, $\mathcal{G}$ and $\mathcal{G}'$, attached to a fixed point $\xi \in \mathrm{\underline{Max}\, F}(X)$, it would be desirable to compare the algorithmic resolution of $\mathcal{G}$ to that of $\mathcal{G}'$, and vice versa. To deal with this problem, we need the notion of weak equivalence of Rees algebras.

\begin{Def}\cite[Definition 3.5]{Br_G-E_V}
We say that two $\mathcal{O}_{V}$-Rees algebras $\mathcal{G}$ and $\mathcal{H}$ are \textit{weakly equivalent} if:

\begin{enumerate}
	\item $\mathrm{Sing}(\mathcal{G})=\mathrm{Sing}(\mathcal{H})$,
	\item Any $\mathcal{G}$-local sequence over $V$
	$$\mathcal{G}=\mathcal{G}_0\longleftarrow \mathcal{G}_1\longleftarrow \ldots \longleftarrow \mathcal{G}_l$$
	induces an $\mathcal{H}$-local sequence over $V$
	$$\mathcal{H}=\mathcal{H}_0\longleftarrow \mathcal{H}_1\longleftarrow \ldots \longleftarrow \mathcal{H}_l$$
	and vice versa, and moreover the equality in (1.) is preserved, that is
	\item $\mathrm{Sing}(\mathcal{G}_j)=\mathrm{Sing}(\mathcal{H}_j)$ for $j=0,\ldots ,l$.\\
\end{enumerate}

\end{Def}

\begin{Ex}
Let $V$ be a smooth scheme over a field $k$ of characteristic zero. Let $X$ be a hypersurface in $V$. Denote now by $b$ the maximum multiplicity of $X$. Then, locally at each point, there exists a Rees algebra $\mathcal{G}$ representing $\mathrm{mult}(X)$ via local embeddings (see Example \ref{ex:HS_repr_local_pres}, Definition \ref{representable} and Example \ref{ex3}). This algebra $\mathcal{G}$ is unique up to weak equivalence.\\
\end{Ex}

The following definitions and results give a flavour of what this equivalence relation means:

\begin{Def}
A Rees algebra over $V$, $\mathcal{G}=\oplus _{n\geq 0}I_nW^n$ is \textit{integrally closed} if it is integrally closed as an $\mathcal{O}_V$-ring in $\mathrm{Quot}(\mathcal{O}_V)[W]$. We denote by $\overline{\mathcal{G}}$ the integral closure of $\mathcal{G}$.
\end{Def}

\begin{Def}
Two Rees algebras are \textit{integrally equivalent} if their integral closure in $\mathrm{Quot}(\mathcal{O}_V)[W]$ coincides.
\end{Def}

\begin{Def}
A Rees algebra $\mathcal{G}=\oplus _{n\geq 0}I_nW^n$ over $V$ is \textit{differentially closed} (or a \textit{Diff-algebra}) if there is an affine open covering of $V$, $\{ U_i\} $ such that for every $D\in \mathrm{Diff}^{(r)}(U_i)$ and $h\in I_n(U_i)$, we have $D(h)\in I_{n-r}(U_i)$ whenever $n\geq r$, where $\mathrm{Diff}^{(r)}(U_i)$ is the locally free sheaf of $k$-linear differential operators of order $r$ or less. In particular, $I_{n+1}\subset I_n$ for $n\geq 0$. We denote by $\mathrm{Diff}(\mathcal{G})$ the smallest differential Rees algebra containing $\mathcal{G}$ (its \textit{differential closure}). (See \cite[Theorem 3.4]{V07} for the construction.)
\end{Def}

\begin{Thm}\cite[Theorem 3.11]{Br_G-E_V}
Let $\mathcal{G}_1$ and $\mathcal{G}_2$ be two Rees algebras over $V$. Then $\mathcal{G}_1$ and $\mathcal{G}_2$ are weakly equivalent if and only if $\overline{\mathrm{Diff}(\mathcal{G}_1)}=\overline{\mathrm{Diff}(\mathcal{G}_2)}$.
\end{Thm}

\begin{Cor}\label{cor:ord_weak_eq}
Let $\mathcal{G}_1$ and $\mathcal{G}_2$ be two weakly equivalent Rees algebras over $V$. Then for all $\eta \in \mathrm{Sing}(\mathcal{G}_1)=\mathrm{Sing}(\mathcal{G}_2)$, we have $\mathrm{ord}_{\eta }\mathcal{G}_1=\mathrm{ord}_{\eta }\mathcal{G}_2$.
\end{Cor}

\begin{Cor}
Let $\mathcal{G}_1$ and $\mathcal{G}_2$ be two weakly equivalent Rees algebras. Then a constructive resolution of $\mathcal{G}_1$ induces a constructive resolution of $\mathcal{G}_2$ and vice versa (see \cite[Remark 11.8]{Br_V2}).
\end{Cor}

\begin{Rem}
Let $X$ be a variety, and fix an immersion $X\hookrightarrow V$. Any two local presentations of $X$ attached to the multiplicity or to the Hilbert-Samuel function are weakly equivalent by definition, and therefore Corollary \ref{cor:ord_weak_eq} applies: fixed an immersion for $X$, the order of a Rees algebra attached to a local presentation at any point of its singular locus does not depend on the local presentation, and neither does the resolution. The previous results give an answer to the problem of compatibility of Rees algebras over $V$.
\end{Rem}

\subsection{Elimination algebras}\label{subsec:elim}

In the following examples, one can observe that, in some cases, the relevant information regarding the simplification of the multiplicity of a variety $X^{(d)} \hookrightarrow V^{(n)}$ can be reflected in a lower dimensional version of $V^{(n)}$. In order to generalize this idea, we have the concept of elimination, which we introduce next.

\subsubsection*{Example I: Hypersurface case}
\addcontentsline{toc}{subsubsection}{Example I: Hypersurface case}

\begin{Ex}\label{ex1} 
Let $S$ be a regular $d$-dimensional $k$-algebra of finite type, with $d>0$. Let $V^{(n)}=\mathrm{Spec}(S[x])$, where $n=d+1$. There is an injective morphism
$$S\stackrel{\beta ^*}{\longrightarrow } S[x]\mbox{,}$$
and an induced smooth projection
\begin{equation}\label{ex:proj}
V^{(n)}\stackrel{\beta}{\longrightarrow }V^{(d)}=\mathrm{Spec}(S)\mbox{.}
\end{equation}
Let $X$ be a hypersurface in $V^{(n)}$, $X=\mathrm{Spec}(S[x]/f(x))$, where $f$ is a polynomial in $x$ of degree $b>1$ with coefficients in $S$. Let $\xi ^{(n)}$ be a point in the closed set of multiplicity $b$ of $X$. We can suppose that the maximal ideal $\mathcal{M}_{\xi ^{(n)}}$ of $\xi ^{(n)}$ in $S[x]$ is given by $<x,z_1,\ldots ,z_d>$ for a regular system of parameters $\left\{ z_1,\ldots ,z_d\right\} $ in $S$. The image $\xi ^{(d)}$ of $\xi ^{(n)}$ by the projection (\ref{ex:proj}) is defined by the maximal ideal $\mathcal{M}_{\xi ^{(d)}}=<z_1,\ldots ,z_d>$. Then, the Rees algebra $\mathcal{G}_X^{(n)}$ over $S[x]$ 
$$\mathcal{G}_X^{(n)}=\mathrm{Diff}(S[x][fW^{b}])\subset S[x][W]\mbox{}$$
represents the multiplicity function on $X$ locally at $\xi ^{(n)}$. \\
\\
\end{Ex}
Let us suppose that, in addition, $f$ has the form of Tschirnhausen:
\begin{equation}\label{ec:f_Tsch}
f(x)=x^b+B_{b-2}x^{b-2}+\ldots + B_ix^i+\ldots +B_0\in S[x]\mbox{,}
\end{equation}
where $B_i\in S$ for $i=0,\ldots ,b-2$ and\footnote{For simplicity, we will sometimes write $\xi $ when we refer to the image of $\xi ^{(n)}$ by most of the maps we use in  this article. In particular, we will often write $\mathrm{ord}_{\xi }$ meaning $\mathrm{ord}_{\xi ^{(n)}}$ or $\mathrm{ord}_{\xi ^{(d)}}$.} $\mathrm{ord}_{\xi }(B_i)\geq b-i$. \\
\\
The following lemma shows that for $X$ as in Example \ref{ex1}, the meaningful part of $f\in S[x]$ (regarding the multiplicity) is given by the coefficients $B_i$, which are already in $S$.

\begin{Lemma}\label{lema:generadores_algebra}
	Let $X$ be given by $f$ as in (\ref{ec:f_Tsch}). Then
	$$\mathcal{G}_X^{(n)}=S[x][xW]\odot \mathrm{Diff}(S[x][B_{b-2}W^2,\ldots ,B_iW^{b-i},\ldots ,B_0W^b])\mbox{.}$$
\end{Lemma}
	
\begin{proof} In order to compute the differential closure of $S[x][fW^b]$, let us start by computing the $(b-1)$-th derivative of $fW^b$ with respect to $x$: one can see that $xW\in \mathcal{G}_X^{(n)}$. Therefore $f_2W^b=fW^b-(xW)^b\in \mathcal{G}_X^{(n)}$ and, if we consider $xW$ and $f_2W^b$ among the generators of $\mathcal{G}_X^{(n)}$, there is no need to include $fW^b$. To continue, we compute the $(b-2)$-th derivative of $f_2W^b$ with respect to $x$ obtaining, up to a nonzero constant, $B_{b-2}W^2\in \mathcal{G}_X^{(n)}$. Just like in the previous step, it is possible to verify that $f_3W^b=f_2W^b-(B_{b-2}W^2)(xW)^{b-2}\in \mathcal{G}_X^{(n)}$, and that $f_2W^b$ can be generated by $xW$, $B_{b-2}W^2$ and $f_3W^b$. By iterating this argument, one  concludes that the set consisting of $xW$ and $B_iW^{b-i}$ for $i=0,\ldots b-2$ is contained in $\mathcal{G}_X^{(n)}$ and, in addition, the differential closure of the $S[x]$-Rees algebra generated by this set\footnote{Note that it is already differentially closed with respect to $x$.} corresponds exactly to $\mathcal{G}_X^{(n)}$.\\
\end{proof}

\begin{Ex}\label{ex2}Instead of (\ref{ec:f_Tsch}), suppose now that $f$ is of the form
\begin{equation}\label{ec:f_Weierstr}
f(x)=x^b+D_{b-1}x^{b-1}+\ldots + D_ix^i+\ldots +D_0\in S[x]\mbox{,}
\end{equation}
where $D_i\in S$, $D_{b-1}\neq 0$ and $\mathrm{ord}_{\xi }(D_i)\geq b-i$ for $i=0,\ldots b-1$. After a suitable change, namely $\tilde{x}=x+\frac{D_{b-1}}{b}$, we obtain
$$f(x)=\tilde{f}(\tilde{x})=\tilde{x}^b+B_{b-2}\tilde{x}^{b-2}+\ldots +B_0\in S[\tilde{x}]\mbox{,\; }B_i\in S\mbox{,\; }\mathrm{ord}_{\xi }(B_i)\geq b-i\mbox{.}$$
\end{Ex}

\vspace{0.2cm}

\subsubsection*{Example II: Multiplicity of a variety}
\addcontentsline{toc}{subsubsection}{Example II: Multiplicity of a variety}

\begin{Ex}\label{ex3}(see \cite[7.1]{V}) 
Let $X$ be a variety of dimension $d$ over $k$ of maximum multiplicity $b$, and let $\xi \in X$ be a point in $\mathrm{\underline{Max}\, mult}(X)$. We have, after possibly replacing $X$ by an \'etale neighbourhood of $\xi $, a smooth $k$-algebra $S$ of dimension $d$ and a finite and \textit{transversal} projection
\begin{equation}\label{eq:beta_X}\beta _X :X\longrightarrow \mathrm{Spec}(S)=V^{(d)}\mbox{,}\end{equation}
that is, a finite projection of generic rank $b$. Note that $\beta _X$ induces a homeomorphism between $\mathrm{\underline{Max}\, mult}(X)$ and its image (\cite[Appendix A]{Br_V2}, \cite[4.8]{V}), and an injective finite morphism
$$S\longrightarrow B=S[\theta _1,\ldots ,\theta _{n-d}]\cong S[x_1,\ldots ,x_{n-d}]/I(X)\mbox{.}$$
As a consequence, we have a local immersion of $X$ in a smooth $n$-dimensional space 
$$V^{(n)}=\mathrm{Spec}(S[x_1,\ldots ,x_{n-d}])$$
in a neighbourhood of $\xi $, and it can be shown that there exist $f_1,\ldots ,f_{n-d}\in I(X)\subset S[x_1,\ldots ,x_{n-d}]$ such that for some positive integers $b_1,\ldots ,b_{n-d}$ the Rees algebra
\begin{equation}\label{loc_pres}
\mathcal{G}_{X}^{(n)}=\mathrm{Diff}(\mathcal{O}_{V^{(n)},\xi }[f_1W^{b_1},\ldots ,f_{n-d}W^{b_{n-d}}])
\end{equation}
represents $\mathrm{mult}(X):X\longrightarrow \mathbb{N}$ locally at $\xi $. In addition, for $i=1,\ldots ,n-d$, 
\begin{equation}\label{sep_var}
f_i\in S[x_i]
\end{equation}
and it is the minimal polynomial of $\theta _i$ over $S$ (see \cite[7.1]{V} for more details). Note that $S[x_1,\ldots ,x_{n-d}]\longrightarrow B\cong S[x_1,\ldots ,x_{n-d}]/I(X)$ is a surjective map and that for any $i=1,\ldots n-d$ the following diagram commutes:
\begin{equation}\label{diag:fact_hyp}
\xymatrix{
S[x_1,\ldots ,x_{n-d}] \ar[r] & S[x_1,\ldots ,x_{n-d}]/(f_1,\ldots ,f_{n-d}) \ar[r] & B \ar[r] & 0\\ 
S[x_i] \ar[u] \ar[r] & S[x_i]/(f_i) \ar[u] & & \\
S \ar[u] \ar[ur] & & & \\
}
\end{equation}
Due to (\ref{sep_var}), we can perform changes of variables for all of the $x_i$ as in \ref{ex2} in order to obtain an expression as in (\ref{ec:f_Tsch}) for each of the $f_i$. We will therefore assume that, when we consider a local presentation attached to the multiplicity for $X$ as (\ref{loc_pres}), the $f_i$ have the form of Tschirnhausen.

\begin{Rem}
In the particular case in which, locally at $\xi $, $B=S[\theta _1]$, necessarily $I(X)=(f_1)$, $B\cong k[x_1]/(f_1)$, and hence $X$ is a hypersurface in $V^{(n)}$.\\
\end{Rem}

\end{Ex}

Given an $n$-dimensional smooth scheme of finite type $V^{(n)}$, and a Rees algebra $\mathcal{G}^{(n)}$ over $V^{(n)}$, which we will consider as a pair from now on, it would be useful to find a new pair $(V^{(n-e)},\mathcal{G}^{(n-e)})$ of dimension $n-e<n$, so that a resolution of $\mathcal{G}^{(n-e)}$ induces a resolution of $\mathcal{G}^{(n)}$, since the first one could be easier to find.

\begin{Def}\label{def:elim_alg}
Let $\mathcal{G}^{(n)}$ be a differential Rees algebra over $V^{(n)}$, and let $\xi \in \mathrm{Sing}(\mathcal{G}^{(n)})$ be a closed point. For  a suitable\footnote{No larger than the invariant $\tau $ at $\xi $, see \cite{B} for more details.} $e\geq 1$  and a smooth transversal\footnote{This condition just means that the intersection of $\mathrm{Ker}(d\beta )$ and the tangent space of $\mathcal{G}^{(n)}$ at $\xi $ is $0$. This guarantees that $\beta $ induces a homeomorphism between $\mathrm{Sing}(\mathcal{G}^{(n)})$ and $\beta (\mathrm{Sing}(\mathcal{G}^{(n)}))$.} projection (also admissible\footnote{For this, it suffices to have $\mathcal{G}^{(n)}$ differentially closed with respect to $\beta $, that is, closed under the action of the sheaf of relative differential opperators $\mathrm{Diff}_{V^{(n)}/V^{(n-e)}}$.}),
$$\beta :V^{(n)}\longrightarrow V^{(n-e)}$$
in a neighbourhood of $\xi $, we define an \textit{elimination algebra} $\mathcal{G}^{(n-e)}$ of   $\mathcal{G}^{(n)}$ as $\mathcal{G}^{(n)}\cap \mathcal{O}_{V^{(n-e)}}$ up to integral closure.
\end{Def}

For a complete description of the properties asked to the projections, and of elimination algebras, we refer to \cite{Br_V}, \cite{Br_V1}, \cite[16 and Appendix A]{Br_V2}, \cite{V} and \cite[Theorem 4.11 and Theorem 4.13]{V07}.\\

\begin{Par}\label{Prop_elim}\textbf{Properties}
\begin{enumerate}
	\item The projection $\beta $ induces a homeomorphism between $\mathrm{Sing}(\mathcal{G}^{(n)})$ and $\beta (\mathrm{Sing}(\mathcal{G}^{(n)}))=\mathrm{Sing}(\mathcal{G}^{(n-e)})$.
	\item Any $\mathcal{G}^{(n)}$-local sequence over $V^{(n)}$ induces a $\mathcal{G}^{(n-e)}$-local sequence over $V^{(n-e)}$ and a commutative diagram
	\begin{equation}\label{diag:def-elim}
\xymatrix@R=0.7pc@C=3pt{
\mathcal{G}^{(n)}=\mathcal{G}_0^{(n)} & & & \mathcal{G}_1^{(n)} & & & \ldots & & & \mathcal{G}_r^{(n)}  \\
V^{(n)}=V_0^{(n)} \ar[dd]_{\beta } & & & V_1^{(n)} \ar[lll] \ar[dd]_{\beta _1} & & & \ldots \ar[lll] & & & V_r^{(n)} \ar[lll] \ar[dd]_{\beta _r}  \\
& & & & & & & & &\\
V^{(n-e)}=V_0^{(n-e)} & & & V_1^{(n-e)} \ar[lll] & & & \ldots \ar[lll] & & & V_r^{(n-e)} \ar[lll]  \\
\mathcal{G}^{(n-e)}=\mathcal{G}_0^{(n-e)} & & & \mathcal{G}_1^{(n-e)} & & & \ldots & & & \mathcal{G}_r^{(n-e)}  \\
}
\end{equation}
	where $\mathcal{G}_i^{(n-e)}$ is an elimination algebra of $\mathcal{G}_i^{(n)}$ for $i=0,\ldots ,r$, and the $\beta _i$ are smooth $\mathcal{G}^{(n)}$-admissible projections inducing therefore homeomorphisms between $\mathrm{Sing}(\mathcal{G}_i^{(n)})$ and $\mathrm{Sing}(\mathcal{G}_i^{(n-e)})$.
	\item Any $\mathcal{G}^{(n-e)}$-local sequence over $V^{(n-e)}$ induces a $\mathcal{G}^{(n)}$-local sequence over $V^{(n)}$ and a commutative diagram as above where $\mathcal{G}_i^{(n-e)}$ is an elimination algebra of $\mathcal{G}_i^{(n)}$ for $i=0,\ldots ,r$, and with $\beta _i$ smooth $\mathcal{G}^{(n)}$-admissible projections inducing homeomorphisms between $\mathrm{Sing}(\mathcal{G}_i^{(n)})$ and $\mathrm{Sing}(\mathcal{G}_i^{(n-e)})$.
	\item Properties 1-3 characterize the elimination algebra $\mathcal{G}^{(n-e)}$ up to weak equivalence.
	\item Any resolution of $\mathcal{G}^{(n)}$ induces a resolution of $\mathcal{G}^{(n-e)}$ and vice versa.
	\item For any two elimination algebras $\mathcal{G}^{(n-e)}$ and $\breve{\mathcal{G}}^{(n-e)}$ of $\mathcal{G}^{(n)}$, given by projections $V^{(n)}\stackrel{\beta }{\longrightarrow } V^{(n-e)}$ and $V^{(n)}\stackrel{\breve{\beta }}{\longrightarrow }\breve{V}^{(n-e)}$ respectively, we have the same order at the image of $\xi $ (see \cite[Theorem 10.1]{Br_V}). That is,
	$$\mathrm{ord}_{\xi }\mathcal{G}^{(n-e)}=\mathrm{ord}_{\xi }\breve{\mathcal{G}}^{(n-e)}\mbox{.}$$
	Let us define 
	$$\mathrm{ord}_{\xi }^{(n-e)}(\mathcal{G}^{(n)})$$
	as the order $\mathrm{ord}_{\xi }\mathcal{G}^{(n-e)}$ (the order at the image of $\xi $) for any elimination algebra $\mathcal{G}^{(n-e)}$ of $\mathcal{G}^{(n)}$ of dimension $n-e$. Hence $\mathrm{ord}_{\xi }^{(n-e)}(\mathcal{G}^{(n)})$ is an invariant for $\mathcal{G}^{(n)}$ at $\xi $.
\end{enumerate}
\end{Par}

In particular, given $X\subset V^{(n)}$ and a Rees algebra $\mathcal{G}^{(n)}$ representing the multiplicity of $X$, as in Example \ref{ex3}, we wish to find a Rees algebra in dimension $d=dim(X)$ which is an elimination algebra of $\mathcal{G}^{(n)}$. The reason for this will be explained in Section \ref{subsec:alg_res}. The following theorem guarantees that this is possible:

\begin{Thm}\label{thm:existe_elim_d}
Let $X\subset V^{(n)}$ be a $d$-dimensional variety over a field of characteristic zero, and $\mathcal{G}_X^{(n)}$ a Rees algebra over $V^{(n)}$ representing the multiplicity of $X$. Then it is possible to find a smooth projection $\beta :V^{(n)}\longrightarrow V^{(d)}$ inducing an elimination algebra $\mathcal{G}^{(d)}_X$ of $\mathcal{G}_X^{(n)}$. Moreover, the order $\mathrm{ord}_{\xi }^{(d)}(\mathcal{G}^{(n)}_X):= \mathrm{ord}_{\beta {\xi }}\mathcal{G}^{(d)}_X$ does not depend on the choice of the embedding or of the algebra $\mathcal{G}_X^{(n)}$.
\end{Thm}

\begin{proof} This fact follows from \cite[Section 21, Theorem 28.8, Theorem 28.10 and Example 28.2]{Br_V2}.\\ 
\end{proof}

\begin{Ex}\ Let us suppose that $X$ is a hypersurface of dimension $d$, and consider the Rees algebra $\mathcal{G}_X^{(n)}$ representing the multiplicity of $X$, as in Example \ref{ex1}. There is a Rees algebra $\mathcal{G}_X^{(d)}$ over $S$, the elimination algebra of $\mathcal{G}_X^{(n)}$, given by
\begin{equation}\label{eq:elim_algebra}\mathcal{G}_X^{(d)}=\mathrm{Diff}(S[x][fW^b])\cap S[W]\end{equation}
describing the image by (\ref{ex:proj}) of $\mathrm{\underline{Max}\, mult}(X)$ (or equivalently, the set of points of maximum multiplicity of the image of $X$ by (\ref{ex:proj})). For a description of this elimination algebra see Lemma \ref{lema:generadores_algebra2} below.\\
\end{Ex}

\begin{Ex} Let us go back to Example \ref{ex2}. It is worth noting that $\mathcal{G}_X^{(d)}$ is invariant under translations of the variable $x$, see \cite{V00}, and hence the $S[x]$-Rees algebra generated by $fW^b\in S[x][W]$ and the $S[\tilde{x}]$-Rees algebra generated by $\tilde{f}W^b\in S[\tilde{x}][W]$ give equivalent elimination algebras $\mathrm{Diff}(S[x][fW^b])\cap S[W]$ and $\mathrm{Diff}(S[\tilde{x}][\tilde{f}W^b])\cap S[W]$ respectively
 (and now we are in the situation of Example \ref{ex1}). \\
\end{Ex}

\begin{Lemma}\label{lema:generadores_algebra2}
	Let $X$ be given by $f$ as in Example \ref{ex1}. Then the elimination algebra of $\mathcal{G}^{(n)}_X$ relative to (\ref{ex:proj}) is (up to integral closure)
	\begin{equation}\label{eq:elim_hyp}
	\mathcal{G}_X^{(d)}=\mathrm{Diff}(S[B_{b-2}W^2,\ldots ,B_iW^{b-i},\ldots ,B_0W^b])\mbox{.}
	\end{equation}
\end{Lemma}
	
\begin{proof} Considering the expression given by Lemma \ref{lema:generadores_algebra}, (\ref{eq:elim_hyp}) follows from the facts that $B_i\in S$ for $i=0, \ldots ,b-2$, and that $\mathcal{G}^{(n-e)}=\mathcal{G}^{(n)}\cap \mathcal{O}_{V^{(n-e)}}$.
\end{proof}

\begin{Rem}\label{rem:ex1}
One can see $\mathcal{G}_X^{(n)}$ as the smallest $S[x]$-Rees algebra containing $xW$ and $\mathcal{G}_X^{(d)}$. By abuse of notation, we will simply write
$$\mathcal{G}_X^{(n)}=S[x][xW]\odot \mathcal{G}^{(d)}_X\mbox{,}$$
meaning that we extend both algebras to Rees algebras over the same ring and apply $\odot $ afterwards (see Definition \ref{def:amalgama_op}).
\end{Rem}
	
	\begin{Lemma}\label{lema:ord_algebra_eliminacion}
	Let $X$ be a hypersurface, given by $f$ as in Example \ref{ex1}. Let $\mathcal{G}_X^{(d)}$ be the elimination algebra of $\mathcal{G}_X^{(n)}$ as in (\ref{eq:elim_algebra}). Then for $\xi \in \mathrm{Sing}(\mathcal{G}_X^{(n)})$,
	\begin{equation}\label{eq:orden_algebra_eliminacion}
	\mathrm{ord}_{\xi }(\mathcal{G}_X^{(d)})=\min _{i=0,\ldots ,b-2} \left\{ \frac{\mathrm{ord}_{\xi }(B_i)}{(b-i)}\right\} \mbox{.}
	\end{equation}
	\end{Lemma}
	
	\begin{proof}
By the expression of $\mathcal{G}_X^{(d)}$ given in Lemma \ref{lema:generadores_algebra2}, it is clear that it is enough to prove that, for any $i$, the element $B_iW^{b-i}$ has lower order than any of its derivatives in $\xi $. The element $B_iW^{b-i}$ has order $\frac{\mathrm{ord}_{\xi }(B_i)}{b-i}$ in $\xi $, while the order of its $j$-th derivative (for $j<b-i$) is greater than or equal to $\frac{\mathrm{ord}_{\xi }(B_i)-j}{b-i-j}$, and for any pair of positive integers $A\geq A'$, $\frac{A}{A'}\leq \frac{A-k}{A'-k}$ for any $k<A'$. On the other hand, any element generated by the $B_i$ and their derivatives has greater order (see \cite[Proposition 3.11]{Bl_E11}).
	\end{proof}
	
\begin{Rem}
Let $X$ be a hypersurface given by $f$ as in Example \ref{ex2}. Then the result in Lemma \ref{lema:ord_algebra_eliminacion} can be applied after a variable change.\\
\end{Rem}

\begin{Ex}\label{ex:elim_gral} If $X$ is as in Example \ref{ex3}, for any $i\in \left\{ 1,\ldots ,n-d\right\} $, $f_i\in S[x_i]$ is the equation of a hypersurface $H_i$ in a scheme of dimension $\bar{n}=d+1$, $\mathrm{Spec}(S[x_i])$ and, by Remark \ref{rem:ex1}: 
$$\mathcal{G}_{H_i}^{(\bar{n})}=\mathrm{Diff}(S[x_i][f_iW^{b_i}])=S[x_i][x_iW]\odot \mathcal{G}_{H_i}^{(d)}\mbox{.}$$
By extending this algebra to $\mathcal{O}_{V^{(n)},\xi }$, we obtain
$$\mathcal{G}_{H_i}^{(n)}=\mathrm{Diff}(\mathcal{O}_{V^{(n)},\xi }[f_iW^{b_i}])=\mathcal{O}_{V^{(n)},\xi }[x_iW]\odot \mathcal{G}_{H_i}^{(d)}\mbox{.}$$
Hence, (\ref{loc_pres}) can be written as 
$$\mathcal{G}^{(n)}_{X}=\mathcal{G}_{H_1}^{(n)}\odot \ldots \odot \mathcal{G}_{H_{n-d}}^{(n)}=\mathcal{O}_{V^{(n)},\xi }[x_1W,\ldots ,x_{n-d}W]\odot \mathcal{G}_{H_1}^{(d)}\odot \ldots \odot \mathcal{G}^{(d)}_{H_{n-d}}\mbox{.}$$
This gives an easy expression for the elimination algebra of $\mathcal{G}_{X}^{(n)}$ relative to the projection 
$$\mathrm{Spec}(S[x_1,\ldots ,x_{n-d}])=V^{(n)} \longrightarrow V^{(d)}=\mathrm{Spec}(S)\mbox{,}$$
namely
$$\mathcal{G}_{X}^{(d)}=\mathcal{G}^{(d)}_{H_1}\odot \ldots \odot \mathcal{G}^{(d)}_{H_{n-d}}\mbox{.}$$
An explanation of this elimination can be found in \cite[Remark 16.10]{Br_V2}. The elimination algebra $\mathcal{G}_X^{(d)}$ will be differentially closed (see \cite[Proposition 5.1]{V3}). See Theorem \ref{thm:triv_orders} for the role of $\mathcal{G}^{(d)}_X$ in algorithmic resolution.\\
\end{Ex}

\subsection{Algorithmic resolution}\label{subsec:alg_res}

A variety $X$ of dimension $d$ over a field of characteristic zero can be desingularized by a sequence of blow ups at smooth centers \cite{Hir}. \textit{Algorithmic resolutions} provide a way to construct such sequences, attending to suitable invariants associated to the points of $X$ \cite{V1}, \cite{V2}, \cite{B-M}, \cite{E_V97}.

\subsubsection*{Resolution functions}
\addcontentsline{toc}{subsubsection}{Resolution functions}

For the construction of an algorithm of resolution \cite{E_V97}, consider a well ordered set $(\Lambda ,\geq )$ and an upper semicontinuous function defined on varieties $F(X)=F_X$, $F_X:X\longrightarrow (\Lambda, \geq )$ such that for any $X$, $\mathrm{\underline{Max}\, } F_X\subset X$ is a closed and smooth subset, and $F_X$ is constant on $X$ if and only if $X$ is smooth. Set $\mathrm{\underline{Max}\, } F_X$ as the center of the first blow up $X\stackrel{\pi _1}{\leftarrow} X_1$. The function $F_X$ must satisfy $F_X(\xi )>F_{X_1}(\xi ')$ whenever $\xi =\pi _1(\xi ')\in \mathrm{\underline{Max}\, } F_X$. Given a variety $X$, the algorithm will give us a sequence of blow ups by iterating the process, that is,
$$X=X_0\stackrel{\pi _1}{\longleftarrow}X_1\stackrel{\pi _2}{\longleftarrow}\ldots \stackrel{\pi _r}{\longleftarrow}X_r\mbox{,}$$
with  $\pi _i$ being  the blow up at $\mathrm{\underline{Max}\, }F_{X_{i-1}}$ for $i=1,\ldots ,r$.

\subsubsection*{Invariants}

\addcontentsline{toc}{subsubsection}{Invariants}

When it comes to the construction of the resolution function, we use invariants of the varieties in order to assign a value (in fact, a set of values) to each point reflecting the complexity of the singularities. Examples \ref{ex:HS} and \ref{ex:mult} give upper semicontinuous functions which are often useful for this construction.\\
\\
As a first coordinate of the resolution function $F_X$, we can consider the Hilbert-Samuel function or the multiplicity at each point. In particular, we will be interested in considering the multiplicity. We will compare the values of $F_X$ at different points using the lexicographical order, and this first coordinate will allow us to focus already on the stratum of maximum value of the multiplicity in $X$.\\
\\
For each $\xi \in \mathrm{\underline{Max}\, mult}(X)$, we know that we can attach a local presentation and an algebra $\mathcal{G}^{(n)}_X$ for the multiplicity. We have already said  that the order of $\mathcal{G}_X^{(n)}$ at $\xi $ is the most important resolution invariant at $\xi $. Therefore, let us take it as the second coordinate of $F_X$.\\
\\
If $X$ is a $d$-dimensional variety, then it can be shown that there are suitable admissible projections to smooth $(n-i)$- dimensional schemes $V^{(n-i)}$, and elimination algebras ${\mathcal G}^{(n-i)}$, $i=1,\ldots, n-d$.
For the following coordinates, we will use the orders $\mathrm{ord}_{\xi }\mathcal{G}_X^{(n-i)}$ of the eliminations as in \ref{Prop_elim} (6), for $i=1,\ldots ,n-d$ (see \ref{thm:existe_elim_d}):
\begin{equation}\label{eq:sat_fun_coords}
F_{X}(\xi )=\left( \mathrm{mult}_{\xi }(X),\; \mathrm{ord}_{\xi }\mathcal{G}_X^{(n)},\; \mathrm{ord}_{\xi }^{(n-1)}\mathcal{G}_X^{(n)},\ldots ,\; \mathrm{ord}_{\xi }^{(d+1)}\mathcal{G}_X^{(n)},\; \mathrm{ord}_{\xi }^{(d)}\mathcal{G}_X^{(n)},\ldots  \right) \mbox{.}
\end{equation}

These invariants behave well under weak equivalence of Rees algebras. More precisely:

\begin{Rem}
Two weakly equivalent Rees algebras $\mathcal{G}$ and $\mathcal{G}'$ share their resolution invariants and hence the constructive resolution of each of them induces the constructive resolution of the other one. This follows from the fact that all invariants that we consider for the construction or the resolution functions derive from Hironaka's order function (\cite[10.3]{Br_G-E_V},\cite[4.11, 4.15]{E_V97}) together with Corollary \ref{cor:ord_weak_eq}. In particular, this is the case for Rees algebras coming from different local presentations once we have fixed an immersion (see \ref{subsec:eq_RA}).
\end{Rem}

Among the orders in (\ref{eq:sat_fun_coords}), the next theorem will tell us that $\mathrm{ord}_{\xi }^{(d)}\mathcal{G}_X^{(n)}$ is the first interesting one, since all the previous are necessarily equal to $1$, and therefore this will be the coordinate we will focus on for our results.

\begin{Thm}\label{thm:triv_orders}\cite[16.7]{Br_V2}
Let $X$ be a $d$-dimensional variety, and let $(V^{(n)},\mathcal{G}^{(n)})$ be an $n$-dimensional pair attached to $X$ at a point $\xi \in \mathrm{\underline{Max}\; mult}(X)$. Then for any $e<n-d$ we have $\mathrm{ord}_{\xi }^{(n-e)}\mathcal{G}^{(n)}=1$.
\end{Thm}
Thus, $F_X$ can actually be constructed as
\begin{equation}\label{eq:sat_fun_coords_abr}
F_{X}(\xi )=\left( \mathrm{mult}_{\xi }(X), \mathrm{ord}_{\xi }^{(d)}\mathcal{G}_X^{(n)},\ldots  \right) \mbox{.}
\end{equation}
It follows from \ref{Prop_elim} that $\mathrm{ord}_{\xi }^{(d)}\mathcal{G}_X^{(n)}$ does not depend on the choice of the elimination algebra. It neither depends on the immersion, by Theorem \ref{thm:existe_elim_d}. Our main result (Theorem \ref{thm:main_order-r}) will show that this invariant, $\mathrm{ord}_{\xi }^{(d)}\mathcal{G}_X^{(n)}$, can be obtained from the arcs in $X$ through $\xi $.

\vspace{0.5cm}
 
\section{Arc spaces and Nash multiplicity sequences}

\subsection{The  space of arcs of $X$}
Let $X$ be an algebraic variety over a field $k$ of characteristic zero. Let us suppose, for simplicity, that $X$ is affine. Otherwise, since we will work locally, it would be enough to consider open affine subsets of $X$. Thus, say $X=\mathrm{Spec}(R)$ for some $k$-algebra of finite type $R$.

\begin{Def}
The \textit{space of arcs of $X$}, $\mathcal{L}(X)$, is a $k$-scheme whose $K$-valued points are the morphisms 
\begin{equation}
\varphi :R \longrightarrow K[[t]]
\end{equation}
for any extension $K$ of $k$. We say that the prime  $\varphi^{-1}(\langle t \rangle)\subset R$ is the {\em center of the arc $\varphi$} in $X$. 
We denote by $\mathcal{L}_{\xi }(X)$ the \textit{space of arcs of $X$ through a (not necessarily closed) point $\xi \in X$}, i.e., those arcs in $\mathcal{L}(X)$ with center $\xi$.
\end{Def} 

\begin{Rem}
It should be noticed that, for a given $X$,
 $\mathcal{L}(X)$ is not necessarily of finite type, and therefore it is not an algebraic variety over $k$.
\end{Rem}

There is a long bibliography  where one can find the basics of arc spaces. For instance, we refer to \cite{Voj} for more details on the construction of $\mathcal{L}(X)$.

\begin{Def}\label{def:order_arc} We define the \textit{order of an arc} $\varphi \in \mathcal{L}(X)$ through $\xi \in X$, $\varphi : \mathcal{O}_{X,\xi }\longrightarrow K[[t]]$ as the largest positive integer $n$ such that $\varphi (\mathcal{M}_{\xi })\subset (t^n)$, where $\mathcal{M}_{\xi }$ is the maximal ideal of the local ring $\mathcal{O}_{X,\xi }$, and denote it by $\mathrm{ord}(\varphi )$ if $\xi $ is clear from the context.
\end{Def}

\subsection{Rees algebras and Nash multiplicity sequences}\label{subsec:RAandNash}

In \cite{L-J}, M. Lejeune-Jalabert introduced a sequence of positive integers attached to an arc in a germ of a hypersurface at a point, and she called it the \textit{Nash multiplicity sequence}. This sequence is non increasing:
$$m_0\geq m_1\geq m_2 \geq \ldots \geq m_k\geq 1$$
and stabilizes  for some $k\in \mathbb{N}$.\\
\\
Later, in \cite{Hickel05}, M. Hickel generalized this sequence for varieties of higher codimension. The way in which he constructs the sequence, involves a sequence of blow ups determined by the chosen arc. For this construction, Hickel works with arcs inside of a germ of a variety at a point (analytic context). We will work with arcs inside of a local neighbourhood of the variety at the point (local algebraic context). We will explain now this construction carefully, to show the computation of the Nash multiplicity sequence from this local algebraic point of view.

\begin{Par}\textbf{Nash multiplicity sequence}
Let $X^{(d)}$ be an irreducible algebraic variety of dimension $d$ over a perfect field $k$. Let $\xi $ be a point contained in $\mathrm{\underline{Max}\, mult}(X^{(d)})$, the closed set of points of maximum multiplicity of $X^{(d)}$.\footnote{Note that we can always assume this situation for any $\xi \in X$, since one can always consider a neighbourhood of $\xi $ where this is true.} For simplicity, in this paper we will assume that $\xi $ is a closed point. This will allow us to consider the blow up at $\xi $, since $\xi $ is a smooth center in this case. In case one wants to consider non  closed points, one needs just to localize $X$ at $\xi $ before performing the sequences that we will construct in this Section.\\
\end{Par}
Consider the product of $X^{(d)}$ with an affine line. Then, we have a surjective morphism
\begin{equation}\label{diag:prod_line}
X^{(d)}\stackrel{p}{\longleftarrow } X_{0}^{(d+1)}=X^{(d)}\times \mathbb{A}_{k}^{1}\mbox{,}
\end{equation}
given by the projection onto the first component. Let us write $\xi _0=(\xi ,0)$, which is a point in $X^{(d+1)}_0$.\\
\\
Consider the blow up of $X_0^{(d+1)}$ at $\xi _0$, which we will denote by $\pi _1$. We will write $X_1^{(d+1)}$ for the transform of $X_0^{(d+1)}$ under $\pi _1$.
After performing this blow up, we can choose a new point $\xi _1\in X_1^{(d+1)}$, and call $\pi _2$ the blow up of $X^{(d+1)}_1$ at $\xi _1$. \\
\\
Next, we will establish a criterion for the choice of each $\xi _i \in X_i^{(d+1)}$ using an arc, so that we can perform a sequence of permissible blow ups at points in this way.
 \begin{equation}\label{diag:perm_seq1}
\xymatrix{ 
(X_0^{(d+1)},\xi _0) & & (X_1^{(d+1)},\xi _1) \ar[ll]_{\pi _1} & & \ldots \ar[ll]_>>>>>>>>>>{\pi _2} & & (X_r^{(d+1)},\xi _r) \ar[ll]_<<<<<<<<<{\pi _r} \mbox{.}\\
}
\end{equation}
Let $\varphi \in \mathcal{L}(X^{(d)})$ be an arc in $X^{(d)}$ through $\xi $. That is, we have a local homomorphism of local rings
\begin{align*}
\varphi : \mathcal{O}_{X^{(d)},\xi } & \longrightarrow K[[t]]\\
\mathcal{M}_{\xi } & \longrightarrow <t>\mbox{,}
\end{align*}
or, equivalently, a morphism $\varphi ^*: \mathrm{Spec}(K[[t]]) \longrightarrow X^{(d)}$, mapping the closed point to $\xi $. This, together with the inclusion map $i: k[t]\rightarrow K[[t]]$ gives an arc $\Gamma _0 $ in $X^{(d+1)}_0$ through $\xi _0$ 
\begin{align*}
\Gamma _0: \mathcal{O}_{X_0^{(d+1)},\xi_0} & \stackrel{\varphi \otimes i}{\longrightarrow} K[[t]] \\
\mathcal{M}_{\xi _0} & \longmapsto <t>
\end{align*}
where $\Gamma_0^*$ is the morphism given by the universal property of the fiber product:
\begin{equation}
\xymatrix{
  \mathrm{Spec}(K[[t]]) \ar@/^2pc/[drr]^{i^*} \ar@/_1.5pc/[ddr]_{\varphi ^*} \ar@{|-->}[dr]^{\Gamma _0^*} &  & \\
 & (X^{(d)},\xi )\times _{k}\mathrm{Spec}(K[t])=(X_0^{(d+1)},\xi _0) \ar[d] \ar[r] & \mathrm{Spec}(K[t]) \ar[d]\\
 & (X^{(d)},\xi ) \ar[r] &  \mathrm{Spec}(k)
}
\end{equation}
Note that $\Gamma _0$ is in fact the graph of $\varphi $.\\
\\
Consider the blow up $\pi _1$ of $X_0^{(d+1)}$ at $\xi _0$. The initial \textit{Nash multiplicity of $X$ at $\xi $} is defined as
$$m=m_0=\mathrm{mult}_{\xi _0}(X_0^{(d+1)})=\mathrm{mult}_{\xi }(X^{(d)})\mbox{,}$$
where the last identity follows from the faithful flatness of (\ref{diag:prod_line}).\\
\\
After blowing up $X_0^{(d+1)}$ at $\xi _0$ (as in \ref{diag:perm_seq1}), the valuative criterion of properness ensures that we can lift $\Gamma _0^*$ to a unique arc in $X_1^{(d+1)}$, which we will denote by $\Gamma _1^*$. Now $\Gamma _1^*$ maps the closed point of $\mathrm{Spec}(K[[t]])$ to some closed point $\xi _1\in X_1^{(d+1)}$. Furthermore, $\xi _1\in E_1=\pi _1^{-1} (\xi _0)$ and $\xi _1\in \mathrm{Im}(\Gamma _1^*)$. This point $\xi _1$ will be the center of the blow up $\pi _2$. We iterate this process: for $i=1,\ldots ,r$, let $\Gamma _i$ be the lifting of the arc $\Gamma _{i-1}\in \mathcal{L}(X_{i-1}^{(d+1)})$ through $\xi _{i-1}$ by the blow up $\pi _i$ of $X^{(d+1)}_{i-1}$ with center $\xi _{i-1}$. Then $\Gamma _i$ is an arc in $\mathcal{L}(X^{(d+1)}_i)$ through a point $\xi _i$ in the exceptional divisor $E_i=\pi _i^{-1}(\xi _{i-1})$. We will say that the sequence of transformations at points chosen in this way is the sequence \textit{directed by $\varphi $} (or that the blow ups themselves are directed by $\varphi $), meaning that $\xi _0=(\varphi (0),0)=(\xi ,0)$ and $\xi _i=\mathrm{Im}(\Gamma _i^*)\cap E_i$ for $i=1,\ldots ,r$:
 \begin{equation}\label{diag:perm_seq1b}
\xymatrix{ 
(X_0^{(d+1)},\xi _0) & & (X_1^{(d+1)},\xi _1) \ar[ll]_{\pi _1} & & \ldots \ar[ll]_>>>>>>>>>>{\pi _2} & & (X_r^{(d+1)},\xi _r) \ar[ll]_<<<<<<<<<{\pi _r} \\
(\mathrm{Spec}(K[[t]]),0) \ar@{^{(}->}[u]^{\Gamma _0^*} & & (\mathrm{Spec}(K[[t]]),0) \ar[ll]_{id} \ar@{^{(}->}[u]^{\Gamma _1^*} & & \ldots \ar[ll]_>>>>>>>>>>>>{id} & & (\mathrm{Spec}(K[[t]]),0)\mbox{.} \ar[ll]_<<<<<<<<<<<{id} \ar@{^{(}->}[u]^{\Gamma _r^*} \\
}
\end{equation}
For this  sequence, the multiplicity of $X_i^{(d+1)}$ at $\xi _i$, will be the \textit{$i$-th Nash multiplicity}, $m_i$. The sequence $m_0, m_1,\ldots ,m_r$ is non-increasing  (see \cite[Theorem 4.1]{Hickel05} or  \cite{Dade}: the blow up at regular equimultiple centers does not increase the multiplicity) \  and eventually decreasing whenever   the generic point of the   initial arc $\varphi $ is not contained in $\mathrm{\underline{Max}\, mult}(X)$. Indeed, if $\varphi $ is contained in the stratum of $X$ of multiplicity $m'$ but not totally contained in any stratum of multiplicity greater than $m'$, then the sequence stabilizes at the value $m'$.\footnote{Therefore, for our purpose, we need to choose arcs in a way such that they are not contained in the set of points of highest multiplicity of $X$ (that is, $\varphi ^*(<0>)\nsubseteq \mathrm{\underline{Max}\; mult}(X)$).} Thus, we can find some $r$ so that for the diagram above the  sequence of Nash multiplicities   is such that $m_0=\ldots =m_{r-1}>m_r$. Our interest is in finding this $r$, namely the minimum number of blow ups at points directed by the arc $\varphi $ as above which  is necessary to perform in order to lower the  Nash multiplicity of $X$ at $\xi $.\\
\\
Since this can be done for any arc $\varphi \in \mathcal{L}(X)$ through $\xi $, let us define:
\begin{Def}\label{def:rho}
Let $\varphi $ be an arc in $X$ through $\xi $. We denote by $\rho _{X,\varphi }$ the minimum number of blow ups directed by $\varphi $ which are needed to lower the  Nash multiplicity of  $X$ at $\xi $. That is, $\rho _{X,\varphi }$ is such that $m=m_0=\ldots =m_{\rho _{X,\varphi }-1}>m_{\rho _{X,\varphi }}$. We will call $\rho _{X,\varphi }$ the \textit{persistance of $\varphi $ in $X$}. We denote by $\rho _{X}(\xi )$ the infimum of the number of blow ups directed by some arc in $X$ through $\xi $ needed to lower the Nash multiplicity   at $\xi $:
\begin{align*}
\rho _X: \mathrm{\underline{Max}\; mult}(X) & \longrightarrow \mathbb{N}\\
\xi & \longmapsto \rho _X(\xi )=\inf _{\varphi \in \mathcal{L}_{\xi }(X)}\left\{ \rho _{X,\varphi }\right\} \mbox{.}
\end{align*}
\end{Def}

To keep the notation as simple as possible, $\rho _{X,\varphi }$ does not contain a reference to the point $\xi $, even though it is clear that it is local. However, the point is determined by $\varphi $, and hence it is implicit, although not explicit in the notation. Similarly, we may refer to $\rho _X(\xi )$ as $\rho _X$ once the point is fixed.\\
\\
Let us define normalized versions of $\rho _{X,\varphi }$ and $\rho _X$ in order to avoid the influence of the order of the arc in the number of blow ups needed to lower the  Nash multiplicity. 

\begin{Def}\label{def:rho_bar}
For a given arc $\varphi $ in $X$, we will write
$$\bar{\rho }_{X,\varphi }=\frac{\rho _{X,\varphi }}{\mathrm{ord}(\varphi )}\mbox{,}$$
and similarly, we will denote
$$\bar{\rho }_{X}(\xi )=\inf _{\varphi \in \mathcal{L}_{\xi }(X)}\left\{ \bar{\rho }_{X,\varphi }\right\} \mbox{.}$$
\end{Def}
\vspace{1cm}

Let us state our main theorem now, and develop afterwards the tools needed for the proof of this and some related results. Recall that $\mathrm{ord}_{\xi }\mathcal{G}_X^{(d)}$ is the first interesting coordinate of our resolution function (see section \ref{subsec:alg_res}). Theorem \ref{thm:main_order-rho} at the end of Section \ref{sec:proof} gives a relation between this invariant and the Nash multiplicity sequence.\\
\\
In the following section, we will show that for $X$, $\xi \in X$ and $\varphi \in \mathcal{L}_{\xi }(X)$, we can attach a Rees algebra to the sequence of blow ups directed by $\varphi $. From this algebra, we will define a new quantity, $r_{X,\varphi }$ (see Definition \ref{def:r}) which is a refinement of $\rho _{X,\varphi }$. In particular, $\rho _{X,\varphi }$ is obtained by taking the integral part of $r_{X,\varphi }$ (see \ref{thm:r_elim_amalgama}). With this notation, the following result holds:
\vspace{0.2cm}
\begin{Thm}\label{thm:main_order-r} \textbf{(Main Theorem)}
Let $X$ be a $d$-dimensional variety defined over a field of characteristic zero $k$. Let $\xi $ be a point in $\mathrm{\underline{Max}\, mult}(X)$. Then, 
$$\mathrm{ord}_{\xi }\mathcal{G}_X^{(d)}=\min _{\varphi \in \mathcal{L}_{\xi }(X)}\left\{ \frac{r_{X,\varphi }}{\mathrm{ord}(\varphi )}\right\} \mbox{.}$$
\end{Thm}

This result will be reformulated in \ref{thm:main} and the proof will be addressed in section \ref{sec:proof}.

\vspace{0.5cm}

\section{Rees algebras attached to Nash multiplicity sequences}\label{par:def_Nash_m}

In this section, the situation we consider for all constructions and results is always the same, specified in \ref{subsec:setting}.
\subsection{Setting: notation and hypothesis}\label{subsec:setting}
Let $X$ be a $d$-dimensional variety over $k$. Locally in an \'etale neighbourhood $\mathcal{U}_{\eta }$ of each point $\eta \in X$, we can find an immersion  of $\eta $, $\mathcal{U}_{\eta }\hookrightarrow V^{(n)}$, and a Rees algebra $\mathcal{G}_X^{(n)}$ over $\mathcal{O}_{V^{(n)},\xi }$ such that
\begin{equation}\label{eq:condition_G_X}
\mathrm{Sing}(\mathcal{G}^{(n)}_X)=\mathrm{\underline{Max}}\, \mathrm{mult}(X)\mbox{,}
\end{equation}
and the equality is preserved by $\mathcal{G}_X^{(n)}$-local sequences over $V^{(n)}$ as long as the maximum multiplicity does not decrease (see \cite{V}). In other words, the multiplicity is represented by $\mathcal{G}_X^{(n)}$ (see Definition \ref{representable}). Let us recall that $\mathcal{G}^{(n)}_X$ can be chosen to be differentially closed (see \ref{loc_pres}). For simplicity of the notation, we will also write $X$ for this neighbourhood $\mathcal{U}_{\eta }$ from now on. \\

Let us choose a point $\xi \in \mathrm{\underline{Max}\, mult}(X)$. If we go back to (\ref{diag:prod_line}), after the product $X^{(d)}\times \mathbb{A}_k^1$, we also have an immersion, and thus a commutative diagram
\begin{equation}\label{diag:p}
\xymatrix{
V^{(n)} & & V_0^{(n+1)}=V^{(n)}\times \mathbb{A}_k^1 \ar[ll]_{p} \\
X^{(d)} \ar@{^{(}->}[u] & & X_0^{(d+1)}=X^{(d)}\times \mathbb{A}_k^1 \ar[ll]_{\left. p\right| _{X_0^{(d+1)}}} \mbox{.} \ar@{^{(}->}[u] \\
}
\end{equation}
In particular, $p$ is a local sequence on $V^{(n)}$ and preserves (\ref{eq:condition_G_X}), and thus the smallest $\mathcal{O}_{V_0^{(n+1)},\xi _0}$-Rees algebra contaning $\mathcal{G}^{(n)}_X$ (the extended algebra) represents the function $\mathrm{mult}(X_0^{(d+1)})$. We will refer to this algebra as the $\mathcal{O}_{V_0^{(n+1)},\xi _0}$-Rees algebra $\mathcal{G}_{X_0}^{(n+1)}$.\\
\\
Fix an arc $\varphi \in \mathcal{L}(X)$ through $\xi $ not contained in $\mathrm{\underline{Max}\, mult}(X)$. The sequence of blow ups at points directed by $\varphi $ defined in (\ref{diag:perm_seq1b}) induces a sequence\footnote{For simplicity of the notation, we will often identify the points $\xi _i$ in $X_i^{(d+1)}$ with their images in $V_i^{(n+1)}$.} of blow ups for $V_0^{(n+1)}$:
 \begin{equation}\label{diag:perm_seq1c}
\xymatrix{ 
(V_0^{(n+1)},\xi _0) & & (V_1^{(n+1)},\xi _1) \ar[ll]_{\pi _1} & & \ldots \ar[ll]_>>>>>>>>>>{\pi _2} & & (V_r^{(n+1)},\xi _r) \ar[ll]_<<<<<<<<<{\pi _r} \\
(X_0^{(d+1)},\xi _0) \ar@{^{(}->}[u] & & (X_1^{(d+1)},\xi _1) \ar[ll]_{\left. \pi _1\right| _{X_1^{(d+1)}}} \ar@{^{(}->}[u] & & \ldots \ar[ll]_>>>>>>>>>>{\left. \pi _2\right| _{X_2^{(d+1)}}} & & (X_r^{(d+1)},\xi _r) \ar[ll]_<<<<<<<<<{\left. \pi _r\right| _{X_r^{(d+1)}}} \ar@{^{(}->}[u] \\
(\mathrm{Spec}(K[[t]]),0) \ar@{^{(}->}[u]^{\Gamma _0^*} & & (\mathrm{Spec}(K[[t]]),0) \ar[ll]_{id} \ar@{^{(}->}[u]^{\Gamma _1^*} & & \ldots \ar[ll]_>>>>>>>>>>>>{id} & & (\mathrm{Spec}(K[[t]]),0)\mbox{.} \ar[ll]_<<<<<<<<<<<{id} \ar@{^{(}->}[u]^{\Gamma _r^*} \\
}
\end{equation}
Consider now the ring $\mathcal{O}_{V^{(n)},\xi }\otimes _{k} K[[t]]$, and the localization at $\xi _0=(\xi ,0)$:\footnote{We use the same notation for the image of $\xi $ by $p*$ and by $\bar{\delta }$.}
\begin{equation}\label{diag:completion}
\delta:\mathcal{O}_{V^{(n)},\xi } \longrightarrow (\mathcal{O}_{V^{(n)},\xi }\otimes _{k} K[[t]])_{\xi _0}\mbox{,}
\end{equation}
and let us denote $\tilde{V}_0^{(n+1)}=\mathrm{Spec}(\mathcal{O}_{V^{(n)},\xi }\otimes _{k} K[[t]])_{\xi _0}$ and $\tilde{X}_0^{(d+1)}=\mathrm{Spec}(\mathcal{O}_{X^{(d)},\xi }\otimes _{k} K[[t]])_{\xi _0}$. Let us choose a regular system of parameters $y_1,\ldots y_n\in \mathcal{O}_{V^{(n)},\xi }$, so that  $\{ y_1,\ldots y_n,t\} $ is a regular system of parameters in $(\mathcal{O}_{V^{(n)},\xi }\otimes _{k} K[[t]])_{\xi _0}$.\\

Note that if $\beta_X:X\to \mathrm{Spec}(S)=V^{(d)}$ is a finite morphism as in (\ref{eq:beta_X}) then after the natural base extension, $\tilde{X}_0^{(d+1)}\to \tilde{V}_0^{(d+1)}$ is also a finite morphism. We
will need this fact in the proof of Proposition \ref{prop:eq_amalgama_ordenes} below.

Now $\Gamma _0$ can be described by the images of $t$ and the classes $\overline{y_i}$ of the $y_i$ in $\mathcal{O}_{X_0^{(d+1)},\xi _0}$, for $i=1, \ldots ,n$:
\begin{align*}
\Gamma _0: \mathcal{O}_{X_0^{(d+1)},\xi _0} & \longrightarrow K[[t]]\\
\overline{y_i} & \longmapsto \varphi_{y_i}=\varphi (\overline{y_i}) \mbox{\; \; i=1,\ldots ,n}\\
t & \longmapsto t\mbox{.}
\end{align*}
Since both $\varphi $ and $\delta $ are continuous, there is a $k$-morphism $\tilde{\Gamma }_0:(\mathcal{O}_{V^{(n)},\xi }\otimes _{k} K[[t]])_{\xi _0}\longrightarrow K[[t]]$ which is completely determined by the images of the $\overline{y_i}$ and $t$. The following commutative diagram provides an overview of the situation:

\begin{equation}\label{diag:triangle}
 \xymatrix@C=0.5pc@R=1pc{\mathcal{O}_{V_0^{(n+1)},\xi _0} \ar[d] & & & \mathcal{O}_{V^{(n)},\xi }  \ar[lll]_{p^*} \ar[rrrrrrr]^{\delta } \ar[d] & & & & & & & (\mathcal{O}_{V^{(n)},\xi }\otimes _{k} K[[t]])_{\xi _0} \ar[d]\\
\mathcal{O}_{X_0^{(d+1)},\xi _0} \ar[ddddrrr]_{\Gamma _0} & & & \mathcal{O}_{X^{(d)},\xi } \ar[lll] \ar[rrrrrrr] \ar[dddd]^{\varphi } & & & & & & & (\mathcal{O}_{X^{(d)},\xi }\otimes _{k} K[[t]])_{\xi _0} \ar@/^3pc/[ddddlllllll]^{\tilde{\Gamma }_0}\\
\overline{y_i},t \ar@{|->}[dddrr] & & & & \overline{y_i} \ar@{|->}[dd] \ar@{|->}[rrrr] & & & & y_i \ar@/^1pc/@{|->}[ddllll] & & \\
& & & & & & & & & & \\
& & & & \varphi _{y_i} & & & & & & \\
& & \varphi _{y_i}=\varphi (\overline{y_i}),t & K[[t]] & & & & & & & \\
}
\end{equation}

\vspace{1cm}
Note that $\tilde{\Gamma }_0$ is an arc in $\tilde{X}_0^{(d+1)}$ defining a curve $C_0$ which is smooth, since it is given in $\tilde{V}_0^{(n+1)}$ by the equations\footnote{$C_0$ is a smooth curve in a local ring, and hence a complete intersection.} $y_i-\varphi _{y_i}=0$ for $i=1,\ldots ,n$ where $\varphi _{y_i}\in K[[t]]$ for $i=1,\ldots n$. This curve is the closure of the image of $\tilde{\Gamma }_0^*:\mathrm{Spec}(K[[t]]) \rightarrow \tilde{V}_0^{(n+1)}$, induced by $\tilde{\Gamma}_0$. We get an analogous diagram to that in (\ref{diag:perm_seq1c}):
 \begin{equation}\label{diag:perm_seq2}
\xymatrix{ 
(\tilde{V}_0^{(n+1)},\xi _0) & & (\tilde{V}_1^{(n+1)},\xi _1) \ar[ll]_{\tilde{\pi }_1} & & \ldots \ar[ll]_>>>>>>>>>>{\tilde{\pi }_2} & & (\tilde{V}_r^{(n+1)},\xi _r) \ar[ll]_<<<<<<<<<{\tilde{\pi }_r} \\
(\tilde{X}_0^{(d+1)},\xi _0) \ar@{^{(}->}[u] & & (\tilde{X}_1^{(d+1)},\xi _1) \ar[ll]_{\left. \tilde{\pi }_1\right| _{\tilde{X}_1^{(d+1)}}} \ar@{^{(}->}[u] & & \ldots \ar[ll]_>>>>>>>>>>{\left. \tilde{\pi }_2\right| _{\tilde{X}_2^{(d+1)}}} & & (\tilde{X}_r^{(d+1)},\xi _r) \ar[ll]_<<<<<<<<<{\left. \tilde{\pi }_r\right| _{\tilde{X}_r^{(d+1)}}} \ar@{^{(}->}[u] \\
(C_0,\xi _0) \ar@{^{(}->}[u] & & (C_1,\xi _1) \ar[ll]_{\left. \tilde{\pi }_1\right| _{C_1}} \ar@{^{(}->}[u] & & \ldots \ar[ll]_>>>>>>>>>>>>{\left. \tilde{\pi }_2\right| _{C_2}} & & (C_r,\xi _r) \ar[ll]_<<<<<<<<<<<{\left. \tilde{\pi }_r\right| _{C_r}} \ar@{^{(}->}[u] \\
}
\end{equation}
where we can see that the preimage  $\tilde{E}_i$ of $\xi _{i-1}$ by $\tilde{\pi }_i$ always intersects $C_i$ at a single point. 
This point is $\xi _i$, the center of the blow up $\tilde{\pi }_{i+1}$.\\

\subsection{Contact algebras}
With the notation in Section \ref{subsec:setting}, let us look now at the closed set $C_0\subset \tilde{V}_0^{(n+1)}$ defined by the arc $\varphi $. We can find an $(\mathcal{O}_{V^{(n)},\xi }\otimes _{k} K[[t]])_{\xi _0}$-Rees algebra $\mathcal{G}_{\varphi }^{(n+1)}$ representing $C_0$ in the sense of Definition \ref{representable}. That is, $\mathcal{G}_{\varphi }^{(n+1)}$ will satisfy $\mathrm{Sing}(\mathcal{G}_{\varphi }^{(n+1)})=C_0$, and for any local sequence as in (\ref{diag:seq_transf}), $\mathrm{Sing}(\mathcal{G}_{\varphi ,i}^{(n+1)})=C_i$, where $C_i$ is the strict transform of $C_{i-1}$ by $\phi _i$ if it is a blow up at a smooth center, or the pullback of $C_{i-1}$ if $\phi _i$ is a smooth morphism. It can be shown that
\begin{equation}\label{ec:RA_arc}\mathcal{G}_{\varphi }^{(n+1)}=\mathcal{O}_{\tilde{V}_0^{(n+1)},\xi _0}[h_1W,\ldots ,h_nW]\mbox{,}\end{equation}
where $h_i=(y_i-\varphi _{y_i})$ for $i=1,\ldots ,n$.
\vspace{0.5cm}
Consider now the closed set
\begin{equation}\label{form:intersection}
Z_0=C_0 \cap \left\{ \eta \in \tilde{X}_0^{(d+1)}: \mathrm{mult}_{\eta }(\tilde{X}_0^{(d+1)})=m\right\} \subset \tilde{V}_0^{(n+1)}\mbox{.}
\end{equation}
For any local sequence
\begin{equation}\label{eq:seq_Z}\tilde{V}_0^{(n+1)}\stackrel{\pi _1}{\longleftarrow}\tilde{V}_1^{(n+1)}\stackrel{\pi _2}{\longleftarrow}\ldots \stackrel{\pi _r}{\longleftarrow}\tilde{V}_r^{(n+1)}\end{equation}
we define $Z_i$, for $i=1, \ldots ,r$, as the closed set 
\begin{equation}
Z_i=C_i \cap \left\{ \eta \in \tilde{X}_i^{(d+1)}: \mathrm{mult}_{\eta }(\tilde{X}_i^{(d+1)})=m\right\} \mbox{,}
\end{equation}
where $C_i$ is the transform of $C_{i-1}$ by $\pi _i$ (that is, the pullback if $\pi _{i-1}$ is a smooth morphism, and the strict transform if it is a blow up at a smooth center contained in $Z_{i-1}$) and $\tilde{X}_i^{(d+1)}$ is the transform of $\tilde{X}_{i-1}^{(d+1)}$.\\
\\

\begin{Def}
Let us suppose now that one can find an $(\mathcal{O}_{V^{(n)},\xi }\otimes _{k} K[[t]])_{\xi _0}$-Rees algebra $\mathcal{H}$ whose singular locus is $Z_0$, and such that this is preserved by local sequences as in (\ref{eq:seq_Z}) (and in particular for sequences of blow ups of $\tilde{X}_0^{(d+1)}$ directed by $\varphi $). We will say that such an algebra, if it exists, is an \textit{algebra of contact of $\varphi $ with $\mathrm{\underline{Max}\; mult}(X)$}.
\end{Def}

\begin{Rem}
Lowering the Nash multiplicity of  $X$ at $\xi $, $m$, is therefore equivalent to resolving this $\mathcal{H}$, and consequently $\rho _{X,\varphi }$ as in Definition \ref{def:rho} is the number of induced transformations of this Rees algebra $\mathcal{H}$ which are necessary to resolve it (see Definition \ref{def:res_RA}).
\end{Rem}

\begin{Rem}
Note that, by the way in which it has been defined, the algebra of contact of $\varphi $ with $\mathrm{\underline{Max}\; mult}(X)$, if it exists, is unique up to weak equivalence.
\end{Rem}

Denote
\begin{equation}\label{eq:thm_amalgama}
\mathcal{G}^{(n+1)}_{X_0,\varphi }:=\mathcal{G}^{(n+1)}_{\tilde{X}_0}\odot \mathcal{G}_{\varphi }^{(n+1)}\mbox{,}
\end{equation}
where $\mathcal{G}_{\tilde{X}_0}^{(n+1)}$ is the extension of $\mathcal{G}_{X}^{(n)}$ to $(\mathcal{O}_{V^{(n)},\xi }\otimes _{k} K[[t]])_{\xi _0}$ (see (\ref{diag:p}) and (\ref{diag:completion})) and $\mathcal{G}_{\varphi }^{(n+1)}$ is as in (\ref{ec:RA_arc}).\footnote{Note that $\mathcal{G}_{\varphi }$ and $\mathcal{G}_{\tilde{X}_0}^{(n+1)}$ are differentially closed by definition.}

\begin{Prop}\label{prop:eq_amalgama_ordenes}
Let $X$ be a variety, let $\xi $ be a point in $\mathrm{\underline{Max}\; mult}(X)$, and let $\varphi $ be an arc in $X$ through $\xi $ with the hypothesis and notation from Section \ref{subsec:setting}. Then the Rees algebra $\mathcal{G}^{(n+1)}_{X_0,\varphi }$ from (\ref{eq:thm_amalgama}) is an algebra of contact of $\varphi $ with $\mathrm{\underline{Max}\; mult}(X)$. Moreover, the restriction $\mathcal{G}_{X_0,\varphi }^{(1)}$ of the same Rees algebra to the curve $C_0$ defined by $\varphi $ is also an algebra of contact of $\varphi $ with $\mathrm{\underline{Max}\; mult}(X)$. In particular, resolving $\mathcal{G}_{X_0,\varphi }^{(n+1)}$ is equivalent to resolving $\mathcal{G}_{X_0,\varphi }^{(1)}$.
\end{Prop}

\begin{proof}
By definition of $\mathcal{G}_{X_0,\varphi }^{(n+1)}$, 
$$\mathcal{F}_{\tilde{V}_0}\left( \mathcal{G}_{X_0,\varphi }^{(n+1)}\right) =\mathcal{F}_{\tilde{V}_0}\left( \mathcal{G}_{\tilde{X}_0}^{(n+1)}\right) \cap \mathcal{F}_{\tilde{V}_0}\left( \mathcal{G}_{\varphi }^{(n+1)}\right)$$
(see Definition \ref{def:tree}). Then, $\mathcal{G}_{X_0,\varphi }^{(n+1)}$ is an algebra of contact of $\varphi $ with $\mathrm{\underline{Max}\; mult}(X)$ as long as $\mathcal{G}_{\tilde{X}_0}^{(n+1)}$ represents $\mathrm{\underline{Max}\; mult}(\tilde{X_0})$ and $\mathcal{G}^{(n+1)}_{\varphi }$ represents $C_0$ in the sense of Definition \ref{representable}. The latter was already shown at the begining of this section. For the first assertion, we may assume that locally we are in the situation of Example \ref{ex3}, and with the notation there, we have now that $S\otimes _kK[[t]] \subset B\otimes _kK[[t]]=S[\theta _1,\ldots ,\theta _{n-d}]\otimes _kK[[t]]$ is a finite extension of rings satisfying the properties in \cite[4.5]{V}, and therefore the argument in \cite[Proposition 5.7]{V} is also valid for them: $\xi \in \mathrm{\underline{Max}\; mult}(\tilde{X_0})$ if and only if $\mathrm{ord}_{\xi }f_i\geq n_i$ for $i=1,\ldots ,n-d$, so the $f_i$ are also the minimal polynomials of the $\theta _i$ over $S\otimes K[[t]]$.\\
\\
On the other hand, by \cite[Proposition 6.6]{Br_G-E_V}
$$\mathcal{F}_{C_0}\left( \left. \mathcal{G}_{\tilde{X}_0}^{(n+1)}\right| _{C_0}\right) =\mathcal{F}_{\tilde{V}_0}\left( \mathcal{G}_{\tilde{X}_0}^{(n+1)}\right) \cap \mathcal{F}_{\tilde{V}_0}\left( \mathcal{G}_{\varphi }^{(n+1)}\right) \mbox{,}$$
since $\mathcal{G}_{\tilde{X}_0}^{(n+1)}$ is differentially closed, and $C_0$ is smooth. Hence, it is clear that the Rees algebra $\left. \mathcal{G}_{\tilde{X}_0}^{(n+1)}\right| _{C_0}$ defines the same tree of closed sets as $\mathcal{G}_{X_0,\varphi }^{(n+1)}$. In addition, the restriction of $\mathcal{G}_{X_0,\varphi }^{(n+1)}$ to $C_0$ defines the very same tree, since
$$\mathcal{F}_{\tilde{V}_0}\left( \mathcal{G}_{X_0,\varphi }^{(1)}\right) :=\mathcal{F}_{\tilde{V}_0}\left( \left. \mathcal{G}_{\tilde{X}_0}^{(n+1)}\right| _{C_0}\odot \left. \mathcal{G}_{\varphi }^{(n+1)}\right| _{C_0}\right) =\mathcal{F}_{\tilde{V}_0}\left( \left. \mathcal{G}_{\tilde{X}_0}^{(n+1)}\right| _{C_0}\right) \cap \mathcal{F}_{\tilde{V}_0}\left( \left. \mathcal{G}_{\varphi }^{(n+1)}\right| _{C_0}\right) =\mathcal{F}_{C_0}\left( \left. \mathcal{G}_{\tilde{X}_0}^{(n+1)}\right| _{C_0}\right) \mbox{,}$$
and the proposition is proved.\\
\end{proof}

The following definition will give us a tool to compute the algebra $\mathcal{G}^{(1)}_{X_0,\varphi }$ that appears in the last Proposition. This will become quite useful in Section \ref{sec:proof}:

\begin{Def}\label{def:image_G_by_arc}
With the notation in Section \ref{subsec:setting}, let $\mathcal{G}$ be a Rees algebra over $V^{(n)}$ given as
$$\mathcal{G}=\mathcal{O}_{V^{(n)},\xi }[g_1W^{c_1},\ldots ,g_{s}W^{c_s}]$$
locally at $\xi $. Then, for any arc $\varphi \in \mathcal{L}_{\xi }(V^{(n)})$, we define
$$\varphi (\mathcal{G})=K[[t]][\varphi (g_1)W^{c_1},\ldots ,\varphi (g_{s})W^{c_s}]\mbox{.}$$
\end{Def}

\begin{Rem}
With the notation in Section \ref{subsec:setting}, we may define the image by $\tilde{\Gamma }_0$ of the Rees algebra $\mathcal{G}^{(n+1)}_{X_0,\varphi }$ from (\ref{eq:thm_amalgama}). This algebra $\tilde{\Gamma }_0(\mathcal{G}^{(n+1)}_{X_0,\varphi })$ happens to be the restriction of the algebra $\mathcal{G}^{(n+1)}_{X_0,\varphi }$ to the curve $C_0$ defined by $\varphi $, and the proof of Proposition \ref{prop:eq_amalgama_ordenes} shows that if $\mathcal{G}^{(n+1)}_{X_0}={\mathcal O}_{V_0^{(n+1)}}[g_1W^{c_1},\ldots ,g_{s}W^{c_s}]$, then 
$$\tilde{\Gamma }_0(\mathcal{G}^{(n+1)}_{X_0,\varphi })=K[[t]][\varphi (g_1)W^{c_1},\ldots ,\varphi (g_s)W^{c_{s}}]\mbox{,}$$
since $\tilde{\Gamma }_0(h_i)=0$ for $i=1,\ldots ,n$.\\
\end{Rem}

Our goal now is to define an invariant for $X$, $\xi $ and $\varphi $ using the algebra of contact of $\varphi $ with $\mathrm{\underline{Max}\; mult}(X)$. However, Proposition \ref{prop:eq_amalgama_ordenes} shows that it would also make sense to define it from the restriction $\mathcal{G}_{X_0,\varphi }^{(1)}$ to $C_0$. In addition, from the way in which $\mathcal{G}_{X_0,\varphi }^{(n+1)}$ is constructed, we know that it has elements of order $1$ in weight $1$, and hence has order $1$ itself\footnote{Note that $\mathcal{G}_{\varphi }^{(n+1)}$ has order one (see (\ref{ec:RA_arc})).} at all points of its singular locus. On contrary, the order of $\mathcal{G}_{X_0,\varphi }^{(1)}$ will be much more interesting, as we will see in Proposition \ref{thm:r_elim_amalgama}.

\begin{Def}\label{def:r}
Let $X$ be a variety, and let $\varphi $ be an arc in $X$ through $\xi \in \mathrm{\underline{Max}\; mult}(X)$. We define the \textit{order of contact} of $\varphi $ with $\mathrm{\underline{Max}\; mult}(X)$ as the order\footnote{As we have done already, we will write $\xi $ for the image of $\xi $ under most of the morphisms we use, as long as the identification between both points is clear.} at $\xi $ of the restriction $\mathcal{G}_{X_0,\varphi }^{(1)}$ to $C_0$ of the algebra of contact of $\varphi $ with $\mathrm{\underline{Max}\; mult}(X)$, and we write it by
\begin{equation*}
r_{X,\varphi }=\mathrm{ord}_{\xi }(\mathcal{G}_{X_0,\varphi }^{(1)})\in \mathbb{Q}\mbox{.}
\end{equation*}
We denote by $r_X$ the infimum of the orders of contact of $\mathrm{\underline{Max}\; mult}(X)$ with all arcs in $X$ through $\xi $:
\begin{equation*}
r_X=\inf _{\varphi \in \mathcal{L}_{\xi }(X)}\left\{ \mathrm{ord}_{\xi } (\mathcal{G}_{X_0,\varphi }^{(1)})\right\} \in \mathbb{R}\mbox{.}
\end{equation*}
\end{Def}

\begin{Rem}
We have defined an invariant $r_{X,\varphi }$ for the pair $(X,\varphi )$ and another invariant $r_X$ for $X$: by Hironaka's trick (see \cite[Section 7]{E_V97}), it can be shown that $r_{X,\varphi }$ depends only on $X$, $\xi $ and $\varphi $, not on the choice of the algebra of contact (which is not unique). For the same reason $r_X$ depends only on $X$ and on the point $\xi $ we are looking at.
\end{Rem}

\begin{Def}
Normalizing $r_{X,\varphi }$ and $r_X$ by the order of the respective arcs (see Definition \ref{def:order_arc}) we define new invariants. We denote
\begin{equation*}
\bar{r}_{X,\varphi }=\frac{\mathrm{ord}_{\xi }(\mathcal{G}_{X_0,\varphi }^{(1)})}{\mathrm{ord}(\varphi )}\in \mathbb{Q}\mbox{,}
\end{equation*}
and
\begin{equation*}
\bar{r}_X=\inf _{\varphi \in \mathcal{L}_{\xi }(X)}\left\{ \frac{\mathrm{ord}_{\xi } (\mathcal{G}_{X_0,\varphi }^{(1)})}{\mathrm{ord}(\varphi )}\right\} \in \mathbb{R}\mbox{.}
\end{equation*}
\end{Def}

We give now a more complete version of Theorem \ref{thm:main_order-r}, which we will prove in Section \ref{sec:proof}:

\begin{Thm}\label{thm:main}
Let $X$ be an algebraic variety of dimension $d$ and $\xi $ a point in $\mathrm{\underline{Max}\; mult}(X)$. Then
\begin{equation*}
\bar{r}_X=\mathrm{ord}_{\xi }\mathcal{G}^{(d)}_X\in \mathbb{Q}\mbox{.}
\end{equation*}

Moreover, the infimum $\bar{r}_X$ is indeed a minimum.\\
\end{Thm}

Equivalently, for every arc $\varphi \in \mathcal{L}(X)$ through $\xi $, 
$$\bar{r}_{X,\varphi }\geq \mathrm{ord}_{\xi }\mathcal{G}^{(d)}_X\mbox{,}$$
and in addition, one can find an arc $\varphi _0\in \mathcal{L}(X)$ through $\xi $ such that
$$\bar{r}_{X,\varphi _0}=\mathrm{ord}_{\xi }\mathcal{G}^{(d)}_X\mbox{.}$$

We already mentioned at the end of Section \ref{subsec:RAandNash} that $r_{X,\varphi }$ is a refinement of $\rho _{X,\varphi }$. The following proposition shows that in fact $\rho _{X,\varphi }$ may be obtained from $r_{X,\varphi }$.

\begin{Prop}\label{thm:r_elim_amalgama} 
Let $X$ be a variety, let $\xi $ be a point in $\mathrm{\underline{Max}\; mult}(X)$ and let $\varphi $ be an arc in $X$ through $\xi $. Then
\begin{equation}
\rho _{X,\varphi }=\left[ r_{X,\varphi }\right] \mbox{.}
\end{equation}
That is, the persistance of $\varphi $ in $X$ (Definition \ref{def:rho}) equals the integral part of the order of contact of $\varphi $ with $\mathrm{\underline{Max}\; mult}(X)$.
\end{Prop}

\begin{proof}
Since $\mathcal{G}_{X_0,\varphi }^{(1)}$ is a Rees algebra over a smooth curve, it is defined over a regular local ring $ \mathcal{O}_{C_0,\xi}$ of dimension one. 
If the maximal ideal $\mathcal{M}_{\xi }$ of $\xi $ in $\mathcal{O}_{C_0,\xi }$ is $\mathcal{M}_{\xi }=<T>$ for some regular parameter $T$, then $\mathcal{G}_{X_0,\varphi }^{(1)}$ is necesarily generated by a finite set of elements of the form $T^{\alpha }W^{l_{\alpha}}$, where $\alpha ,l_{\alpha }$ are positive integers. Observe also that $\mathcal{G}_{X_0,\varphi }^{(1)}$ is integrally equivalent to a Rees algebra generated by $JW^l$ for some principal ideal $J\subset \mathcal{O}_{C_0,\varphi }$ and some positive integer $l$, at least in a neighbourhood of $\xi $ (see \cite[Lemma 1.7]{Br_G-E_V}). Therefore, we can suppose that $\mathcal{G}_{X_0,\varphi }^{(1)}=\mathcal{O}_{C_0,\xi }[T^{\alpha }W^l]$. In this case, the order of $\mathcal{G}_{X_0,\varphi }^{(1)}$ at $\xi $ will be given by 
$$\mathrm{ord}_{\xi }(\mathcal{G}_{X_0,\varphi }^{(1)})=\frac{\alpha }{l}\mbox{.}$$
By the transformation law (\ref{eq:transf_law}), the first transform of $\mathcal{G}_{X_0,\varphi }^{(1)}$ by blowing up at the closed point is
$$\mathcal{G}_{X_0,\varphi ,1}^{(1)}=\mathcal{O}_{C_0,\xi }[T^{\alpha -l}W^l]\mbox{.}$$
The order of the $k$-th transform will therefore be
$$\frac{\alpha -k\cdot l}{l}\mbox{,}$$
and the number $\rho _{X,\varphi }$ of blow ups needed to resolve $\mathcal{G}_{X_0,\varphi }^{(1)}$ must satisfy:
$$0\leq \alpha -\rho _{X,\varphi }\cdot l<l\mbox{.}$$
But this implies
$$0\leq \frac{\alpha }{l}-\rho _{X,\varphi }<1\mbox{,}$$
which means that $\rho _{X,\varphi }$ is the integral part of $\frac{\alpha }{l}=\mathrm{ord}_{\xi }(\mathcal{G}_{X_0,\varphi }^{(1)})$, which is precisely the order of contact of $\varphi $ with $\mathrm{\underline{Max}\; mult}(X)$.
\end{proof}

\begin{Cor}
For any variety $X$, 
$$\rho _X=[ r_{X}] \mbox{,}$$
$$[\bar{r}_X]\leq \bar{\rho }_X\leq \bar{r}_X\mbox{.}$$
\end{Cor}

The proof follows solely from the definitions of $r_X$, $\bar{r}_X$, $\rho _X$ and $\bar{\rho }_X$ together with Proposition \ref{thm:r_elim_amalgama}, by means of algebraic manipulations of their integral parts.

In what follows, we will give the proof of Theorem \ref{thm:main} by focusing first on the hypersurface case and generalizing then to arbitrary codimension.\\

\section{Proof of the main result}\label{sec:proof}

For the proof of Theorem \ref{thm:main}, we assume first that $X$ is a hypersurface in Theorems \ref{thm:desigualdad} and \ref{thm:igualdad}. Later on, we will see that we can deduce the proof of the general case from the hypersurface one in Theorems \ref{thm:desigualdad_gral} and \ref{thm:igualdad_gral}.

\subsection{Rees algebras and orders for a hypersurface}\label{subsec:Rees_algebras_orders_hypers}

For any variety $X$ which is locally a hypersurface, we can always find a nice expression for $X$ in an \'etale neighbourhood of each point. Using this expression, we will prove Theorem \ref{thm:main} for the hypersurface case by dividing it into two theorems: Theorem \ref{thm:desigualdad} states that $\mathrm{ord}_{\xi }\mathcal{G}_X^{(d)}$ from Section \ref{subsec:alg_res} is a lower bound of $\bar{r}_{X,\varphi }$ for any arc $\varphi \in \mathcal{L}_{\xi }(X)$, and Theorem \ref{thm:igualdad} shows that in fact we can find an arc giving the equality, so that $\bar{r}_X$ is actually a minimum. For the proof of these two, we will define diagonal arcs, which will help us analyzing the orders of contact and the order $\mathrm{ord}_{\xi }\mathcal{G}_X^{(d)}$ (see \ref{Prop_elim} 1 to 5, and Theorems \ref{thm:existe_elim_d} and \ref{thm:triv_orders}), and giving some conclusions and lemmas about them.\\

\begin{Par}\textbf{Notation and hypothesis\\}\label{subsubsec:setting_hyp}
Let $X=X^{(d)}$ be a $d$-dimensional variety over $k$ of maximum multiplicity $b$, and let $\xi \in \mathrm{\underline{Max}\, mult}(X)$. Let us suppose that $X$ at $\xi $ is locally a hypersurface, given by $\mathcal{O}_{X,\xi }\cong S[x]/(f)$ for a regular local $k$-algebra $S$ and a variable $x$, as in Example \ref{ex1}. As we did in (\ref{ec:f_Tsch}), we can suppose that $f$ has an expression of the form
\begin{equation}\label{ec:f_Tsch2}
f(x)=x^b+B_{b-2}x^{b-2}+\ldots + B_ix^i+\ldots +B_0
\end{equation}
in some \'etale neighbourhood of $\xi \in X$, with $B_0,\ldots B_{b-2}\in S$, and where we write $n=d+1$ for the dimension of the ambient space $V^{(n)}=\mathrm{Spec}(S[x])$. 
Consider $\mathcal{G}_{X}^{(d)}$, the elimination algebra of $\mathcal{O}_{V^{(n)},\xi }[fW^b]$ in $\mathcal{O}_{V^{(d)},\xi ^{(d)}}$ induced by the projection $\beta :V^{(n)}\longrightarrow V^{(d)}=\mathrm{Spec}(S)$ (see Theorem \ref{thm:existe_elim_d}), as the diagram shows:
\begin{equation}\label{diag:comm_compl-elim}
\xymatrix@R=0.7pc@C=3pt{
\mathcal{G}_{X_0}^{(n+1)}  &  & & & & & \ar[lllll] & \mathcal{G}_X^{(n)} & \ar[rrrrr] & & & & & & \mathcal{G}_{\tilde{X}_0}^{(n+1)} \\
\mathcal{O}_{V^{(n+1)}_0,\xi _0} & & & & & & & \mathcal{O}_{V^{(n)},\xi } \ar[lllllll]_{p^*} \ar[rrrrrrr]^{\delta } & & & & & & & (\mathcal{O}_{V^{(n)},\xi }\otimes _{k} K[[t]])_{\xi _0}\\
& & & & & \\
\mathcal{O}_{V_0^{(d+1)},\xi _0^{(d+1)}} \ar@{^{(}->}[uu] & & & & & & & \mathcal{O}_{V^{(d)},\xi ^{(d)}}  \ar@{^{(}->}[uu]^{\beta ^*} \ar[lllllll] \ar[rrrrrrr] & & & & & & & (\mathcal{O}_{V^{(d)},\xi _0^{(d)}}\otimes _{k} K[[t]])_{\xi _0^{(d+1)}}   \ar@{^{(}->}[uu]\\
\mathcal{G}_{X_0}^{(d+1)} & & & & & & \ar[lllll]  & \mathcal{G}_X^{(d)} & \ar[rrrrr] & & & & & & \mathcal{G}_{\tilde{X}_0}^{(d+1)} 
}
\end{equation}
where $\mathcal{G}_{X_0}^{(d+1)}$ is an elimination of $\mathcal{G}_{X_0}^{(n+1)}$. We have the following expression:
\begin{equation}\label{eq:expr_elim_prev}
\mathcal{G}_{X}^{(n)}=\mathrm{Diff}(\mathcal{O}_{V^{(n)},\xi } [fW^b]) =\mathcal{O}_{V^{(n)},\xi } [xW]\odot \mathcal{G}_{X}^{(d)}
\end{equation} 

(see Lemma \ref{lema:generadores_algebra2} for $\mathcal{G}_X^{(d)}$). Let $\varphi $ be an arc in $X$ through $\xi $, not contained in $\mathrm{\underline{Max}\, mult}(X)$. Suppose that $\varphi $ is such that $\varphi _{x}=u_0t^{\alpha _0}$ and $\varphi _{z_i}=u_it^{\alpha _i}$ for a regular system of parameters $\left\{ z_1,\ldots ,z_d\right\} \in S$, as in (\ref{diag:triangle}), where $u_0,\ldots ,u_d$ are units in $K[[t]]$ and $\alpha _0, \ldots \alpha _d$ are positive integers. This gives the following expressions for the algebra of contact of $\varphi $ with $\mathrm{\underline{Max}\, mult}(X)$ (see Proposition \ref{prop:eq_amalgama_ordenes}):
\begin{align}\label{eq:expr_amalgama}
\mathcal{G}_{X_0,\varphi }^{(n+1)}=\mathrm{Diff}(\OV [fW^b])\odot \OV [(x-u_0t^{\alpha _0})W,(z_i-u_{i}t^{\alpha _{i}})W; i=1,\ldots ,d]= \notag \\
=\OV [xW]\odot \mathcal{G}_{X}^{(d)}\odot \OV [(x-u_0t^{\alpha _0})W, (z_i-u_{i}t^{\alpha _{i}})W; i=1,\ldots ,d]\mbox{.}
\end{align} 
This expression will allow us to know the order of contact of $\varphi $ with $\mathrm{\underline{Max}\, mult}(X)$ (see Definition \ref{def:r}), which is our real interest.\\
\\
Let us recall that Properties \ref{Prop_elim} 1-4 guarantee that $\mathcal{G}_{X}^{(d)}$ represents $\beta (\mathrm{\underline{Max}\, mult}(X))$. Note now that the corresponding projection of $\varphi $ by $\beta $ gives also an arc $\varphi ^{(d)}$ in $V^{(d)}$ according to the following diagram
\begin{equation}\label{diag:proj_arc}
\xymatrix{
\mathcal{O}_{V^{(n)},\xi } \ar[r]^{\varphi } & K[[t]] \\
\mathcal{O}_{V^{(d)},\xi ^{(d)}}  \ar@{^{(}->}[u]^{\beta ^*} \ar[ur]_{\varphi ^{(d)}}
}
\end{equation}
Consider then the elimination algebra $\mathcal{G}_{X_0}^{(d+1)}$ above.  We can construct an algebra of contact of $\varphi ^{(d)}$ with $\beta (\mathrm{\underline{Max}\, mult}(X))$ by an analogous construction to that in (\ref{eq:thm_amalgama}), using the fact that $\mathcal{G}_{X}^{(d)}$ represents $\beta (\mathrm{\underline{Max}\, mult}(X))$. Then we obtain the $\mathcal{O}_{\tilde{V}^{(d+1)},\xi ^{(d+1)}}$-Rees algebra
\begin{equation}\label{eq:amalgama_sinx}
\mathcal{G}^{(d+1)}_{X_0,\varphi ^{(d)}}=\mathcal{G}_{X}^{(d)}\odot \mathcal{G}_{\varphi ^{(d)}}^{(d+1)}\mbox{.}
\end{equation}
Also $\mathcal{G}_{X_0,\varphi ^{(d)}}^{(1)}$ will be the restriction of $\mathcal{G}^{(d+1)}_{X_0,\varphi ^{(d)}}$ to the image of $C_0$ in $\tilde{V}_0^{(d+1)}$ (which we will denote by $C_0^{(d)}$). Note that $\mathcal{G}_{X_0,\varphi ^{(d)}}^{(1)}=\tilde{\Gamma }_0^{(d)}(\mathcal{G}^{(d+1)}_{X_0,\varphi ^{(d)}})$, where $\tilde{\Gamma }_0^{(d)}: (\mathcal{O}_{V^{(d)},\xi ^{(d)}} \otimes _{k}K[[t]])_{\xi _0^{(d)}}\rightarrow K[[t]]$ is given by $\varphi ^{(d)}:\mathcal{O}_{V^{(d)},\xi ^{(d)}}\rightarrow K[[t]]$ as in (\ref{diag:triangle}). With this notation we can write,
$$\mathcal{G}_{X_0,\varphi }^{(n+1)}=\OV [xW,t^{\alpha _0}W] \odot \mathcal{G}_{X}^{(d)}\odot \mathcal{G}_{\varphi ^{(d)}}^{(d+1)}=\OV [xW]\odot K[[t]][t^{\alpha _0}W]\odot \mathcal{G}_{X_0,\varphi ^{(d)}}^{(d+1)}$$
by (\ref{eq:expr_amalgama}) and (\ref{eq:amalgama_sinx}), and hence

$$\mathcal{G}_{X_0,\varphi }^{(1)}=K[[t]][t^{\alpha _0}W]\odot \mathcal{G}_{X_0,\varphi ^{(d)}}^{(1)}\mbox{,}$$
and
\begin{equation}\label{orden_amalgama}
r_{X,\varphi }=\mathrm{ord}_{\xi }(\mathcal{G}_{X_0,\varphi }^{(1)})=\mathrm{min}\left\{ \alpha _0,\mathrm{ord}_{\xi }(\mathcal{G}_{X_0,\varphi ^{(d)}}^{(1)})\right\} \mbox{.}
\end{equation}
\end{Par}
\vspace{0.5cm}

\subsubsection*{Auxiliary results}
The following Lemma shows that, in fact, $\alpha _0$ is not important for $r_{X,\varphi }$.

\begin{Lemma}\label{lemma:ord_amalgama_final}
Let $X$ be as in Section \ref{subsubsec:setting_hyp}. Let $\xi \in \mathrm{\underline{Max}\; mult }(X)$. Then for any arc $\varphi \in \mathcal{L}(X)$ through $\xi $ as in \ref{subsubsec:setting_hyp}:
$$\mathrm{ord}_{\xi }(\mathcal{G}_{X_0,\varphi }^{(1)})=\mathrm{ord}_{\xi }(\mathcal{G}_{X_0,\varphi ^{(d)}}^{(1)}) \mbox{.}$$
\end{Lemma}
\begin{proof}
Assume that $X$ is given by $f$ as in (\ref{ec:f_Tsch2}). Let us suppose that $\varphi $ is given by $(\varphi _{x},\varphi _{z_1},\ldots ,\varphi _{z_d})=(u_0t^{\alpha _0},u_1t^{\alpha _1},\ldots ,u_dt^{\alpha _d})$, with $u_0,\ldots ,u_d$ units in $K[[t]]$ and $\alpha _0,\ldots ,\alpha _d$ positive integers, and recall that, since $\varphi \in \mathcal{L}(X)$,
\begin{equation}\label{f_se_anula:ord_amalgama_final}
\varphi (f)=\varphi \left( x^b+\sum _{i=0}^{b-2}{B_{i}x^{i}}\right)=0\mbox{.}
\end{equation}
By (\ref{orden_amalgama}), it suffices to prove that 
\begin{equation}\label{desig1:proof_lemma:ord_amalgma_final}
\alpha _0\geq \mathrm{ord}_{\xi }(\mathcal{G}_{X_0,\varphi ^{(d)}}^{(1)})\mbox{.}
\end{equation} 
On the other hand, from Lemma \ref{lema:generadores_algebra2} and diagram (\ref{diag:comm_compl-elim}) we know that 
$$\mathcal{G}_{\tilde{X}_0}^{(d+1)}=\mathrm{Diff}(\mathcal{O}_{\tilde{V}_0^{(d+1)},\xi _0^{(d+1)}}[B_iW^{b-i}: i=0,\ldots b-2])\mbox{.}$$
Denote 
$$\mathcal{H}=\mathcal{O}_{\tilde{V}_0^{(d+1)},\xi _0^{(d+1)}}[B_iW^{b-i}: i=0,\ldots b-2]\subset \mathcal{G}_{\tilde{X}_0}^{(d+1)}\mbox{.}$$
The inclusion holds after restricting both algebras to $C_0^{(d)}$, and hence
$$\mathrm{ord}_{\xi }(\varphi ^{(d)}(\mathcal{H}))\geq \mathrm{ord}_{\xi }(\varphi ^{(d)}(\mathcal{G}_{\tilde{X}_0}^{(d+1)}))=\mathrm{ord}_{\xi }(\mathcal{G}_{X_0,\varphi ^{(d)}}^{(1)})\mbox{.}$$
We will show now that 
\begin{equation}\label{paso_intermedio:ord_amalgama_final}
\alpha _0\geq \mathrm{ord}_{\xi }(\varphi ^{(d)}(\mathcal{H}))\mbox{,}
\end{equation}
which implies (\ref{desig1:proof_lemma:ord_amalgma_final}). On the contrary, let us suppose that
$$\alpha _0< \mathrm{ord}_{\xi }(\varphi ^{(d)}(\mathcal{H}))=\min _{i=0,\ldots ,b-2} \left\{ \frac{\mathrm{ord}_{t}\varphi ^{(d)}(B_i))}{b-i}\right\} \mbox{.}$$ 
That is,
\begin{equation*}
\alpha _0< \left( \frac{\mathrm{ord}_{t}(\varphi ^{(d)}(B_i))}{b-i}\right) \mbox{,\; for \; }i=0,\ldots b-2\mbox{,}
\end{equation*}
or equivalently
\begin{equation}\label{desig:individual:ord_amalgama_final}
(b-i)\alpha _0 <\mathrm{ord}_{t}(\varphi ^{(d)}(B_i)))\mbox{,\; for \; }i=0,\ldots b-2\mbox{.}
\end{equation}
Now observe that this implies
\begin{equation*}
 \varphi (f-x^b)=\mathrm{ord}_{t}( \sum _{i=0}^{b-2}\varphi ^{(d)}(B_i)) u_0^{i}t^{i\alpha _0}) \geq \min _{i=0,\ldots ,b-2} \left\{ \mathrm{ord}_{t}(\varphi ^{(d)}(B_i)))+i\cdot \alpha _0\right\} > b\cdot \alpha _0\mbox{.}
\end{equation*}
But this contradicts (\ref{f_se_anula:ord_amalgama_final}), so necessarily (\ref{paso_intermedio:ord_amalgama_final}) holds, concluding the proof of the Lemma. 
\end{proof}
\vspace{0.5cm}

We know now that we can just focus on the projection of $X$ over $S$, for the computation of the order of contact. We need to know now how the induced projection of arcs (\ref{diag:proj_arc}) behaves.

\begin{Def}
We say that an arc $\varphi \in \mathcal{L}(V^{(d)})$ through $\xi ^{(d)}\in V^{(d)}$ is a \textit{diagonal arc} if there exists a regular system of parameters $\left\{ z_1,\ldots ,z_d\right\} \in \mathcal{O}_{V^{(d)},\xi ^{(d)}}$, units $u_1,\ldots ,u_d\in K[[t]]$ and $\alpha \in \mathbb{N}$ such that $\varphi (z_i)=u_it^{\alpha }$ for $i=1,\ldots ,d$.
\end{Def}

\begin{Rem}\label{rem:eq_def_diagonal_arc} The following definition is equivalent to the previous one:\\
We say that an arc $\varphi \in \mathcal{L}(V^{(d)})$ through $\xi ^{(d)}\in V^{(d)}$ is a diagonal arc if there exists a regular system of parameters $\left\{ z_1,\ldots ,z_d\right\} \in \mathcal{O}_{V^{(d)},\xi ^{(d)}}$ inducing a diagram

\begin{equation}
\xymatrix{
0 \ar[r] & \mathrm{Ker}(\Gamma _0) \ar[r] & \mathcal{O}_{V_0^{(d+1)},\xi _0^{(d+1)}} \ar[rd]^{\Gamma _0} & \\
0 \ar[r] & \mathrm{Ker}(\varphi ) \ar[r] & \mathcal{O}_{V^{(d)},\xi ^{(d)}} \ar[r]^{\varphi } \ar[u]^{p^*} \ar[d]^{\delta } & K[[t]] \\
0 \ar[r] & \mathrm{Ker}(\tilde{\Gamma }_0) \ar[r] & (\mathcal{O}_{V^{(n)},\xi }\otimes _{k} K[[t]])_{\xi _0} \ar[ru]_{\tilde{\Gamma }_0} &
}
\end{equation}

where the ideal $\mathrm{Ker}(\tilde{\Gamma }_0)\subset (\mathcal{O}_{V^{(n)},\xi }\otimes _{k} K[[t]])_{\xi _0}$ is generated by elements of the form $(u_jz_i-u_iz_j)$, where $u_l\in K[[t]]$ are units for $l=1,\ldots ,d$.\footnote{Note that $\mathrm{Ker}(\tilde{\Gamma }_0)=\mathrm{Ker}(\Gamma _0)(\mathcal{O}_{V^{(n)},\xi }\otimes _{k} K[[t]])_{\xi _0}$.}
\end{Rem}

\begin{Rem}\label{rem:comp_diag_arcs}
Let $\varphi $ and $\varphi '$ be two arcs in $\mathcal{L}(V^{(d)})$ through $\xi \in V^{(d)}$ whose respective graphs are $\Gamma _0$ and $\Gamma _0'$. If $\varphi $ is diagonal and $\mathrm{Ker}(\Gamma _0)=\mathrm{Ker}(\Gamma _0')$, then $\varphi '$ is also diagonal. Moreover, since $\varphi $ is given by $\varphi (z_i)=u_it^{\alpha }$ for some regular system of parameters $\{ z_1,\ldots ,z_d\} $, where $u_1, \ldots ,u_d$ are units in $K[[t]]$ and $\alpha $ is come positive integer, then $\varphi '$ is given as $\varphi '(z_i)=u_ig'(t)$ for some $g'(t)\in K[[t]]$.
\end{Rem}

\begin{Lemma}\label{lema:arcos_diagonales} Let $X$ be as in \ref{subsubsec:setting_hyp} and let $\varphi ^{(d)}$ be an arc in $V^{(d)}$ through $\xi ^{(d)}\in V^{(d)}$. Then
	\begin{equation}\label{ec:lema_arcos_diagonales}\mathrm{ord}_{\xi }(\mathcal{G}_{X_0,\varphi ^{(d)}}^{(1)})\geq \mathrm{ord}_{\xi }(\mathcal{G}_X^{(d)})\cdot \mathrm{ord}(\varphi ^{(d)}) \mbox{.}\end{equation}
	\end{Lemma}
	
	\begin{proof}
Suppose, contrary to our claim, that $\mathrm{ord}_{\xi }(\mathcal{G}_{X_0,\varphi ^{(d)}}^{(1)})<\mathrm{ord}_{\xi }(\mathcal{G}_X^{(d)})\cdot \alpha $, where $\alpha =\mathrm{ord}(\varphi ^{(d)})$. Let $\varphi ^{(d)}$ be given by $\varphi ^{(d)}(z_i)=u_it^{\alpha _i}$ for some regular system of parameters $\{ z_1,\ldots ,z_d\} $ in $\mathcal{O}_{V^{(d)},\xi ^{(d)}}$, units $u_1,\ldots ,u_d\in K[[t]]$ and positive integers $\alpha _1,\ldots ,\alpha _d$. Then for some $qW^{l}\in \mathcal{G}_{X}^{(d)}$,
	\begin{equation}\label{ec:lema_diagonales1}\frac{\mathrm{ord}_{t}(\varphi ^{(d)}(q))}{l}<\mathrm{ord}_{\xi }(\mathcal{G}_X^{(d)})\cdot \alpha \mbox{.}\end{equation}
	But $\mathrm{ord}_{t}(\varphi ^{(d)}(q))\geq \alpha \cdot \mathrm{ord}_{\xi }(q)$, and hence
	$$\frac{\mathrm{ord}_{t}(\varphi ^{(d)}(q))}{l}\geq \frac{\alpha \cdot \mathrm{ord}_{\xi }(q)}{l}\geq \alpha \cdot \mathrm{ord}_{\xi }(\mathcal{G}_{X}^{(d)})\mbox{,}$$
	leading to a contradiction, and proving the Lemma.\\
	\end{proof}
	
	Note that in the Lemma $\varphi ^{(d)}$ is any arc in $\mathcal{L}_{\xi }(V^{(d)})$, not necessarily the projection of any arc $\varphi \in \mathcal{L}_{\xi }(X)$.\\

	\begin{Def}
	Let $\mathcal{G}^{(d)}$ be a Rees algebra over $V^{(d)}$. We say that an arc $\varphi ^{(d)}\in \mathcal{L}_{\xi }(V^{(d)})$ is \textit{generic for $\mathcal{G}^{(d)}$} if
	$$\mathrm{ord}_{\xi }(\left. (\mathcal{G}^{(d)}\odot \mathcal{G}_{\varphi ^{(d)}}^{(d+1)})\right| _{C_0^{(d)}})=\mathrm{ord}(\varphi ^{(d)})\cdot \mathrm{ord}_{\xi }(\mathcal{G}^{(d)})\mbox{.}$$
	If $\varphi ^{(d)}$ is also diagonal, we say that it is \textit{diagonal-generic}.
	\end{Def}
	\vspace{0.2cm}
	
	\begin{Rem}\label{obs:desig_lema_arcos_diagonales}
	In the situation of Lemma \ref{lema:arcos_diagonales}, an arc for which (\ref{ec:lema_arcos_diagonales}) is an equality is a generic arc for $\mathcal{G}^{(d)}_X$: $\mathcal{G}_X^{(d)}\odot \left. \mathcal{G}_{\varphi ^{(d)}}^{(d+1)}\right| _{C_0^{(d)}}=\left. \mathcal{G}_{X_0,\varphi ^{(d)}}^{(d+1)}\right| _{C_0^{(d)}}=\mathcal{G}_{X_0,\varphi ^{(d)}}^{(1)}$ shows it. Note that such an arc can always be found, by just considering a diagonal arc $\varphi ^{(d)}$ in $V^{(d)}$ through $\xi ^{(d)}\in V^{(d)}$ given, in some regular system of parameters $\{ z_1,\ldots ,z_d\} $, by $(u_1t^{\alpha },\ldots ,u_dt^{\alpha })$, for some positive integer $\alpha $ and units $u_1,\ldots ,u_d\in k$ such that there exists some element $qW^l\in \mathcal{G}_X^{(d)}$ with $\frac{\mathrm{ord}_{\xi }(q)}{l}=\mathrm{ord}_{\xi }(\mathcal{G}_X^{(d)})$, and for which\footnote{If $q\in R$ for a regular local ring $R$ with maximal ideal $\mathcal{M}$, then we denote by $\mathrm{in}_{\xi }(q)$ the \textit{initial part of $q$ at the closed point $\xi $}, meaning the equivalence class of $q$ in the quotient $\mathcal{M}^n/\mathcal{M}^{n+1}$, where $n$ is such that $q\in \mathcal{M}^n$ but $q\notin \mathcal{M}^{n+1}$. Therefore $\mathrm{in}_{\xi }(q)\in \mathrm{Gr}_{R_{\mathcal{M}}}\cong k'[z_1,\ldots ,z_d]$ is a homogeneous polynomial of degree $n$.} $(\mathrm{in}_{\xi }(q))(u_1,\ldots ,u_d)\neq 0$. For this arc, 
	$$\mathrm{ord}_{t}(\varphi ^{(d)}(q))=\alpha \cdot \mathrm{ord}_{\xi }(q)\mbox{,}$$
	and hence
	$$\mathrm{ord}_{\xi }(\mathcal{G}_{X_0,\varphi ^{(d)}}^{(1)})\leq \frac{\mathrm{ord}_{t}(\varphi ^{(d)}(q)))}{l}=\frac{\alpha \cdot \mathrm{ord}_{\xi }(q)}{l}=\alpha \cdot \mathrm{ord}_{\xi }(\mathcal{G}_X^{(d)})\mbox{,}$$
	but Lemma \ref{lema:arcos_diagonales} forces the last inequality to be an equality.
	\end{Rem}
	\vspace{0.2cm}

Even though in this section we are always under the assumption of $X$ being locally a hypersurface, the following Lemma will be stated and proved for a variety of arbitrary codimension, since no extra work is needed and this generality will be necessary in the next section.

\begin{Lemma}\label{lemma:lifting_arcs}
Let $X$ be a variety of dimension $d$ over $k$. With the notation from \ref{subsubsec:setting_hyp}, let $\bar{\varphi }^{(d)}$ be a diagonal arc in $V^{(d)}$ through $\xi ^{(d)}\in V^{(d)}$ which is diagonal-generic for $\mathcal{G}_X^{(d)}$. Then it is possible to find an arc $\varphi \in \mathcal{L}(X)$ through $\xi $ whose projection $\varphi ^{(d)}$ onto $V^{(d)}$ via $\beta _X$ is a diagonal arc which is also diagonal-generic for $\mathcal{G}_X^{(d)}$.
\end{Lemma}

\begin{proof} Consider a local presentation as in Example \ref{ex3} for $X$ at $\xi $ attached to the multiplicity. Let us recall that not every arc in $\left\{ f_1=\ldots =f_{n-d}=0\right\} $ is an arc in $X$, since
$$(f_1, \ldots ,f_{n-d})\subset I(X) \Longrightarrow X\subset \left\{ f_1=\ldots =f_{n-d}=0\right\} \mbox{.}$$
Assume that $\bar{\varphi }^{(d)}(z_i)=u_it^{\alpha }$, $i=1,\ldots ,d$ for some units $u_1,\ldots ,u_d\in K[[t]]$. We need to choose an arc $\varphi $ such that $\varphi \in \mathcal{L}(V(f))$ for all $f\in I(X)$, or equivalently an arc such that $\mathrm{Ker}(\varphi )\supset I(X)$. Consider the following diagram
\footnotesize
$$
\xymatrix{ 
\mathcal{O}_{X,\xi }\cong \mathcal{O}_{V^{(d)},\xi ^{(d)}}[x_1,\ldots  ,x_{n-d}]/I(X) & & \mathcal{O}_{V^{(d)},\xi ^{(d)}}[x_1,\ldots ,x_{n-d}] \ar[ll] \\
 & & \\
 & \mathcal{O}_{V^{(d)},\xi ^{(d)}} \ar[uul]_{\beta ^*_X} \ar[uur]_{\beta ^*}  & 
}
$$
\normalsize
where $\beta ^*_X$ (induced by $\beta _X$ from (\ref{eq:beta_X})) is a finite morphism. Let $\mathcal{P}=\mathrm{Ker}(\overline{\varphi }^{(d)})\subset \mathcal{O}_{V^{(d)},\xi ^{(d)}}$. There is a prime ideal $\mathcal{Q}\subset \mathcal{O}_{X,\xi }$ such that $\mathcal{Q}\cap \mathcal{O}_{V^{(d)},\xi ^{(d)}}=\mathcal{P}$. Note that $\mathcal{Q}$ is lifted to a unique ideal $\mathcal{Q}'\subset \mathcal{O}_{V^{(d)},\xi ^{(d)}}[x_1,\ldots ,x_{n-d}]$, with the property that $I(X)\subset \mathcal{Q}'$. We have the following diagram
$$
\xymatrix{ 
\mathcal{Q} \subset \mathcal{O}_{X,\xi } \ar[r] & \mathcal{O}_{X,\xi }/\mathcal{Q} \\
\mathcal{P} \subset \mathcal{O}_{V^{(d)},\xi ^{(d)}} \ar[r] \ar@<-2ex>[u]^{\beta _X^*} & \mathcal{O}_{V^{(d)},\xi ^{(d)}}/\mathcal{P} \ar[u] 
}
$$
where the left vertical arrow is a finite morphism, forcing the right vertical one to be also finite. Then, the two rings in the right side of the diagram have the same dimension, and thus $\mathcal{Q}$ defines a closed set of dimension $1$ in $X$, $C$. There is an arc $\varphi $ (different from the morphism $0$) in $C$ through $\xi $, and we know that, locally in a neighbourhood of $\xi $, $\mathcal{Q}=\mathrm{Ker}(\varphi )$ and that $\mathrm{Ker}(\varphi )\cap \mathcal{O}_{V^{(d)},\xi ^{(d)}}= \mathrm{Ker}(\varphi ^{(d)})=\mathrm{Ker}(\bar{\varphi }^{(d)})$, so the projection of $\varphi $ onto $V^{(d)}$, $\varphi ^{(d)}$, is diagonal by Remark \ref{rem:comp_diag_arcs}. To see that it is generic for $\mathcal{G}_X^{(d)}$, note that there exists some element $qW^l\in \mathcal{G}_X^{(d)}$ with $\frac{\mathrm{ord}_{\xi }(q)}{l}=\mathrm{ord}_{\xi }(\mathcal{G}_X^{(d)})$ for which $(\mathrm{in}_{\xi }(q))(u_1,\ldots ,u_d)\neq 0$. By passing to the completion of $(\mathcal{O}_{V^{(n)},\xi }\otimes _{k} K[[t]])_{\xi _0}$ at its maximal ideal (see Remark \ref{rem:eq_def_diagonal_arc}) and using Remark \ref{rem:comp_diag_arcs}, it can be checked that this implies that $\varphi ^{(d)}$ is also generic for $\mathcal{G}_X^{(d)}$.

\end{proof}

\begin{Rem}\label{rem:lifting_arcs}
The arc obtained in Lemma \ref{lemma:lifting_arcs} is given (as in (\ref{diag:triangle})) by
\begin{equation}\label{eq:arco_generico_diagonal}
\varphi =(g_1(t),\ldots ,g_{n-d}(t),u_1g'(t),\ldots ,u_dg'(t))
\end{equation}
for some $g_1(t),\ldots ,g_{n-d}(t),g'(t)\in K[[t]]$ and $u_1,\ldots ,u_d\in K[[t]]$, because $\mathrm{Ker}(\varphi )\cap \mathcal{O}_{V^{(d)},\xi ^{(d)}}=\mathrm{Ker}(\bar{\varphi }^{(d)})=\mathrm{Ker}(\varphi ^{(d)})$ and $\varphi ^{(d)}$ is diagonal (see Remark \ref{rem:comp_diag_arcs}).\\

\end{Rem}
\vspace{0.2cm}

\subsubsection*{Results for hypersurfaces}
Now we return to the hypersurface case, and we have enough tools to prove the following theorem:

\begin{Thm} \label{thm:desigualdad}
	Let $X$ be a variety of dimension $d$ which is locally a hypersurface at $\xi \in \mathrm{\underline{Max}\, mult}(X)$. For any $\varphi \in \mathcal{L}(X)$ through $\xi $, with the notation from section \ref{subsubsec:setting_hyp},
	\begin{equation}\label{ec:desig_hypers}\overline{r}_{X,\varphi }\geq \mathrm{ord}_{\xi }(\mathcal{G}_{X}^{(d)})\mbox{.}\end{equation}
\end{Thm}

\begin{proof}
	We can assume that $X$ is given locally by $f$ is as in (\ref{ec:f_Tsch2}). Let us write $\alpha =\mathrm{ord}(\varphi )=\mathrm{min}\left\{ \alpha _0 ,\ldots ,\alpha _d\right\} $. From Lemma \ref{lema:arcos_diagonales}, for any diagonal arc $\tilde{\varphi }$, given as $(\tilde{u}_0t^{\alpha },\ldots ,\tilde{u}_dt^{\alpha })$
	$$\alpha \cdot \mathrm{ord}_{\xi }(\mathcal{G}_X^{(d)})\leq \mathrm{ord}_{\xi }(\mathcal{G}_{X_0,\tilde{\varphi }^{(d)}}^{(1)})\mbox{.}$$
	It suffices to show that it is possible to choose units $\tilde{u}_i\in K[[t]]$ for $i=0,\ldots ,d$ so that
	\begin{equation}\label{objetivo1}
	\mathrm{ord}_{\xi }(\mathcal{G}_{X_0,\tilde{\varphi }^{(d)}}^{(1)}) \leq \mathrm{ord}_{\xi }(\mathcal{G}_{X_0,\varphi ^{(d)}}^{(1)})\mbox{.}
	\end{equation}
This, together with Lemma \ref{lemma:ord_amalgama_final}, would imply that
\begin{equation*}
\alpha \cdot \mathrm{ord_{\xi }}(\mathcal{G}_X^{(d)})\leq \mathrm{ord}_{\xi }(\mathcal{G}_{X_0,\varphi ^{(d)}}^{(1)})= \mathrm{ord}_{\xi }(\mathcal{G}_{X_0,\varphi }^{(1)})\mbox{,}
\end{equation*}
and complete the proof of the Theorem.\\
\\
In order to prove (\ref{objetivo1}), let us consider a finite set of generators  of $\mathcal{G}_X^{(d)}$, $\left\{ g_iW^{l_i}\right\} _{i=1,\ldots ,r}$. Since this set is finite and $k$ is infinite, it is possible to choose units $\tilde{u}_1,\ldots ,\tilde{u}_d \in k$ in a way such that
$$\mathrm{in}_{\xi }(g_i)(\tilde{u}_1,\ldots ,\tilde{u}_d) \neq 0\mbox{\; for \; }i=1,\ldots ,r\mbox{.}$$
Let $\lambda _i=\mathrm{ord}_{\xi }(g_i)$ for $i=1,\ldots ,r$. As $\mathrm{in}_{\xi }(g_i)$ is a homogeneous polynomial, 
$$\mathrm{in}_{\xi }(\tilde{\varphi }^{(d)}(g_i))=t^{\alpha \cdot \lambda _i}\cdot \mathrm{in}_{\xi }(g_i)(\tilde{u}_1,\ldots ,\tilde{u}_d)$$
and
$$\mathrm{ord}_{t}(\tilde{\varphi }^{(d)}(g_i))=\alpha \cdot \lambda _i\mbox{.}$$
On the other hand, observe that
$$\varphi ^{(d)}(g_i)\in <t^{\alpha \cdot \lambda _i}>\mbox{,}$$
so 
\begin{equation}\label{desig:proof_desigualdad_hyp1}\mathrm{ord}_{t}(\varphi ^{(d)}(g_i))\geq \alpha \cdot \lambda _i=\mathrm{ord}_{t}(\tilde{\varphi }^{(d)}(g_i))\mbox{.}\end{equation}
Since (\ref{desig:proof_desigualdad_hyp1}) holds for all $i\in \left\{ 1,\ldots ,r\right\} $, and for some $k\in \left\{ 1,\ldots ,r\right\} $, 
$$\frac{\mathrm{ord}_{t}(\varphi ^{(d)}(g_k))}{l_k}=\mathrm{ord}_{\xi }(\mathcal{G}_{X_0,\varphi ^{(d)}}^{(1)})\mbox{,}$$
it follows that
$$\mathrm{ord}_{\xi }(\mathcal{G}_{X_0,\varphi ^{(d)}}^{(1)})=\frac{\mathrm{ord}_{t}(\varphi ^{(d)}(g_k)))}{l_k}\geq \frac{\mathrm{ord}_{t}(\tilde{\varphi }^{(d)}(g_k))}{l_k}\geq \mathrm{ord}_{\xi }(\mathcal{G}_{X_0,\tilde{\varphi }^{(d)}}^{(1)})$$
concluding the proof of (\ref{objetivo1}), and the proof of the Theorem.\\
\end{proof}

For the proof of the existence of an arc giving an equality in (\ref{ec:desig_hypers}), we will use the following Lemma:

\begin{Lemma}\label{lemma:tau_mayor}
Let $X$ be as in Section \ref{subsubsec:setting_hyp}, and let $\varphi $ be an arc in $X$ through $\xi \in \mathrm{\underline{Max}\; mult}(X)$ with the notation used there where $\varphi (x)=g_1(t)$ and $\varphi (z_i)=u_ig'(t)$, $u_i$ a unit in $K[[t]]$, for $i=1,\ldots ,d$. Assume that $\varphi $ is such that the projection $\varphi ^{(d)}$ on $V^{(d)}$ is a diagonal-generic arc for $\mathcal{G}_X^{(d)}$.\footnote{We know that such an arc exists by Remark \ref{rem:lifting_arcs}.} If $\mathrm{ord}(\varphi)=\mathrm{ord}_{t}(g_1(t))$, then
$$\overline{r}_{X,\varphi }=\mathrm{ord}_{\xi }(\mathcal{G}_X^{(d)})=1\mbox{.}$$
\end{Lemma}

\begin{proof}
Let us suppose that $g'(t)=t^L$ for some positive integer $L$, that is, $\varphi _{z_i}=u_it^L$ for $i=1,\ldots ,d$. By Lemma \ref{lemma:ord_amalgama_final},
$$\mathrm{ord}_{\xi }(\mathcal{G}_{X_0,\varphi }^{(1)})=\mathrm{ord}_{\xi }(\mathcal{G}_{X_0,\varphi ^{(d)}}^{(1)})\mbox{,}$$
and since $\varphi ^{(d)}$ is generic for $\mathcal{G}_X^{(d)}$, Remark \ref{obs:desig_lema_arcos_diagonales} yields
\begin{equation}\label{ec:amalgama_lema}\mathrm{ord}_{\xi }(\mathcal{G}_{X_0,\varphi }^{(1)})=L\cdot \mathrm{ord}_{\xi }(\mathcal{G}_X^{(d)})\mbox{.}\end{equation}
It suffices to prove that
\begin{equation}\label{desig:s_1}\mathrm{ord}_{t}(g_1(t))\geq L\cdot \mathrm{ord}_{\xi }(\mathcal{G}_X^{(d)})\mbox{,}\end{equation}
since it implies
\begin{equation}\label{conclusion:teorema_igualdad}1\leq \mathrm{ord}_{\xi }(\mathcal{G}_X^{(d)})\leq \overline{r}_{X,\varphi }=\frac{L\cdot \mathrm{ord}_{\xi }(\mathcal{G}_X^{(d)})}{\mathrm{ord}_{t}(g_1(t))}\leq 1\mbox{,}\end{equation}
where we have used Theorem \ref{thm:desigualdad} for the second inequality and (\ref{ec:amalgama_lema}) together with the definition of $\overline{r}_{X,\varphi }$ for the equality. Hence $\mathrm{ord}_{\xi }(\mathcal{G}_{X}^{(d)})=\overline{r}_{X,\varphi }=1$, concluding the proof of the Lemma.
In order to prove (\ref{desig:s_1}), let us suppose that our claim is false, that is:
\begin{equation}\label{desig:s_2}\mathrm{ord}_{t}(g_1(t))< L\cdot \mathrm{ord}_{\xi }(\mathcal{G}_X^{(d)})\mbox{.}\end{equation}
Then, in particular,
\begin{equation}\label{eq:aux_thm_tau_mayor}
\mathrm{ord}_{t}(g_1(t))<L\cdot \frac{\mathrm{ord}_{\xi }(B_i)}{b-i}\leq \frac{\mathrm{ord}_{t}(\varphi ^{(d)}(B_i))}{b-i}\mbox{\; for \; }i=0,\ldots ,b-2
\end{equation}
where the first inequality follows from the same argument used in the proof of Lemma \ref{lemma:ord_amalgama_final}. Therefore
$$\mathrm{ord}_{t}(\varphi ^{(d)}(B_i))>\mathrm{ord}_{t}(g_1(t))(b-i)$$
and
\begin{align*}
 \mathrm{ord}_{t}(\varphi (f-x^b)) & =\mathrm{ord}_{t}\left( \sum _{i=0}^{b-2}\varphi ^{(d)}(B_i)g_1(t)^i\right) \geq \min _{i=0,\ldots ,b-2}\left\{ \mathrm{ord}_{t}(\varphi ^{(d)}(B_i))+i\cdot \mathrm{ord}_{t}(g_1(t))\right\} > \\
 > & \min _{i=0,\ldots ,b-2}\left\{ \mathrm{ord}_{t}(g_1(t))(b-i)+i\cdot \mathrm{ord}_{t}(g_1(t))\right\} =b\cdot \mathrm{ord}_{t}(g_1(t))\mbox{,}
\end{align*}
where (\ref{eq:aux_thm_tau_mayor}) is needed for the second inequality. But this contradicts $\varphi (f)=0$ and hence the fact that $\varphi \in \mathcal{L}_{\xi }(X)$, so necessarily (\ref{desig:s_1}) holds, concluding the proof.
\end{proof}

\vspace{0.3cm}

\begin{Thm} \label{thm:igualdad}
Let $X$ be a $d$-dimensional variety over a field $k$ of characteristic zero which is locally a hypersurface in a neighbourhood of $\xi \in \mathrm{\underline{Max}\; mult }(X)$. Then there exists some $\varphi \in \mathcal{L}(X)$ through $\xi $, with the notation from Section \ref{subsubsec:setting_hyp} such that 
\begin{equation}\label{ec:equality_hyper}\overline{r}_{X,\varphi }=\mathrm{ord}_{\xi }(\mathcal{G}_X^{(d)})\mbox{.}\end{equation}
\end{Thm}

\begin{proof}
We can assume again that $X$ is locally given by $f$ as in (\ref{ec:f_Tsch2}). Pick a diagonal-generic arc for $\mathcal{G}^{(d)}_X$ (see Remark \ref{obs:desig_lema_arcos_diagonales} for the existence). By Lemma \ref{lemma:lifting_arcs} it can be lifted to an arc $\varphi $ in $X$ through $\xi $ whose projection $\varphi ^{(d)}$ onto $V^{(d)}$ is diagonal generic for $\mathcal{G}_X^{(d)}$. Remark \ref{rem:lifting_arcs} shows that $\varphi $ is given (as in (\ref{diag:triangle})) by 
\begin{equation}\label{eq:arco_identidad_hip}(g(t),u_1g'(t),\ldots ,u_dg'(t))\end{equation}
for some $g(t),g'(t)\in K[[t]]$ and $u_1,\ldots ,u_d\in k$. We only need to check that for such an arc (\ref{ec:equality_hyper}) holds. Let $N=\mathrm{ord}_{t}(g'(t))$. Note that, since $\varphi ^{(d)}$ is generic for $\mathcal{G}_X^{(d)}$, $\mathrm{ord}_{\xi }(\mathcal{G}_{X_0,\varphi ^{(d)}}^{(1)})=N\cdot \mathrm{ord}_{\xi }(\mathcal{G}_X^{(d)})$. By Lemma \ref{lemma:ord_amalgama_final},
\begin{equation}\label{ec:amalgama_teorema_igualdad} \mathrm{ord}_{\xi }(\mathcal{G}_{X_0,\varphi }^{(1)})=N\cdot \mathrm{ord}_{\xi }(\mathcal{G}_X^{(d)}) \mbox{.}\end{equation}
Consider now two possible situations, depending on whether $\mathrm{ord}(\varphi )=\mathrm{ord}_{t}(g(t))$ or not. If $\mathrm{ord}(\varphi )=\mathrm{ord}_{t}(g(t))$, then Lemma \ref{lemma:tau_mayor} implies
$$1=\mathrm{ord}_{\xi }(\mathcal{G}_{X}^{(d)})=\overline{r}_{X,\varphi }\mbox{.}$$
Otherwise $\mathrm{ord}(\varphi )=N$, and by definition of $\bar{r}_{X,\varphi }$ and (\ref{ec:amalgama_teorema_igualdad}), $\overline{r}_{X,\varphi }=\frac{N\cdot \mathrm{ord}_{\xi }(\mathcal{G}_X^{(d)})}{N}$, completing the proof.
\end{proof}
\vspace{0.2cm}

\begin{Rem}\label{rem:order_arc_id}
Under the assumptions of Theorem \ref{thm:igualdad}, let $\varphi $ be the arc (\ref{eq:arco_identidad_hip}) given by the proof. For this arc
\begin{equation}
\mathrm{ord}(\varphi )=N\mbox{.}
\end{equation}
To see this we observe that, since we have proved that $\overline{r}_{X,\varphi }=\mathrm{ord}_{\xi }(\mathcal{G}_X^{(d)})$, it follows easily from (\ref{ec:amalgama_teorema_igualdad}) that:
$$\mathrm{ord}_{\xi }(\mathcal{G}_X^{(d)})=\overline{r}_{X,\varphi }=\frac{\mathrm{ord}_{\xi }(\mathcal{G}_{X_0,\varphi }^{(1)})}{\mathrm{ord}(\varphi )}=\frac{N\cdot \mathrm{ord}_{\xi }(\mathcal{G}_X^{(d)})}{\mathrm{ord}(\varphi )} \Rightarrow \frac{N}{\mathrm{ord}(\varphi )}=1\mbox{.}$$\\
\end{Rem}
\vspace{0.2cm}

\subsection{Rees algebras and orders for the general case}\label{par:RA_orders_codim_gral}

As we have just done for the proof of Theorem \ref{thm:main} for hypersurfaces, we will use that we can find, in an \'etale neighbourhood of each point $\xi $ of $X$, a local presentation (as in Example \ref{ex3}) given by a collection of hypersurfaces and integers. For each of these hypersurfaces we will assume a nice expression in the line of \ref{subsubsec:setting_hyp}. As a consequence, for any arc $\varphi $ in $X$ through $\xi $ we will be able to give an expression of the algebra of contact of $\varphi $ with $\mathrm{\underline{Max}\; mult}(X)$ in terms of some algebras of contact of arcs with hypersurfaces. This will lead to an easy formula for $r_{X,\varphi }$. With these tools, we will prove in Theorem \ref{thm:desigualdad_gral} that $\mathrm{ord}_{\xi }\mathcal{G}_X^{(d)}$ is again a lower bound for $\bar{r}_{X,\varphi }$ for any arc $\varphi $, and that $\bar{r}_X$ is also a minimum in this case in Theorem \ref{thm:igualdad_gral}. They will come naturally from Theorems \ref{thm:desigualdad} and \ref{thm:igualdad} respectively.\\

\begin{Par}\textbf{Notation and hypothesis for the general case\\}\label{subsubsec:setting_gen}
Let $X$ be a variety of dimension $d$, and let $\xi $ be a point in $\mathrm{\underline{Max}\; mult}(X)$. We already explained in Example \ref{ex3} that, in an \'etale neighbourhood of $\xi $, we can find a local presentation for $X$ attached to the multiplicity, meaning an immersion in $V^{(n)}$, elements $f_i\in \mathcal{O}_{V^{(n)},\xi }=\mathcal{O}_{V^{(d)},\xi ^{(d)}}[x_1,\ldots ,x_{n-d}]$ and positive integers $b_i$ for $i=1,\ldots ,n-d$ as in (\ref{sep_var}), such that 
\begin{equation}\label{eq:G_gral}\mathcal{G}_X^{(n)}=\mathrm{Diff}(\mathcal{O}_{V^{(n)},\xi }[f_1W^{b_1},\ldots ,f_{n-d}W^{b_{n-d}}])\end{equation}
represents the function $\mathrm{mult}(X)$. Consider the differential closure of the $\mathcal{O}_{V^{(n+1)}_0,\xi _0^{(n+1)}}$-Rees algebra generated by the $f_i$, $\mathcal{G}_{X_0}^{(n+1)}$. We already mentioned that $f_i$ is the minimal polynomial of $\theta _i$ over $\mathcal{O}_{V^{(d)},\xi ^{(d)}}$, where $\mathcal{O}_{X,\xi }=\mathcal{O}_{V^{(d)},\xi ^{(d)}}[\theta _1,\ldots ,\theta _{n-d}]$, and we can assume (by \ref{ex2}) that each $f_i$ is of the form:
$$f_i=x_i^{b_i}+B_{\left\{ i\right\} ,b_i-2}x_i^{b_i-2}+\ldots +B_{\left\{ i\right\} ,0}\in \mathcal{O}_{V^{(d)},\xi ^{(d)}}[x_i]\subset \mathcal{O}_{V^{(d)},\xi ^{(d)}}[x_1,\ldots ,x_{n-d}]\mbox{,}$$
where $\left\{ z_1,\ldots ,z_d,t\right\} $ is a regular system of parameters in $\mathcal{O}_{V^{(d+1)}_0,\xi _0}$ and $\left\{ x_1,\ldots ,x_{n-d},z_1,\ldots ,z_d,t \right\} $ a regular system of parameters in $(\mathcal{O}_{V_{i,0}^{(e)},\xi }\otimes _{k} K[[t]])_{\xi _0}$, $B_{\left\{ i\right\} ,b_i-j}\in \mathcal{O}_{V^{(d)},\xi ^{(d)}}$ and $\mathrm{ord}_{\xi }(B_{\left\{ i\right\} ,b_i-j})\geq j$ for $j=2,\ldots ,b_i$, $i=1,\ldots ,n-d$.\\

By Example \ref{ex3}, we know that
$$\mathcal{G}_X^{(n)}=\mathrm{Diff}(\mathcal{O}_{V^{(n)},\xi }[f_1W^{b_1},\ldots ,f_{n-d}W^{b_{n-d}}])=\mathrm{Diff}(\mathcal{O}_{V^{(n)},\xi }[f_1W^{b_1}])\odot \ldots \odot \mathrm{Diff}(\mathcal{O}_{V^{(n)},\xi }[f_{n-d}W^{b_{n-d}}])\mbox{,}$$
where each $\mathrm{Diff}(\mathcal{O}_{V^{(n)},\xi }[f_{i}W^{b_{i}}])$ is the smallest differentially closed $\mathcal{O}_{V^{(n)},\xi }$-Rees algebra with the property of containing the algebra $\mathrm{Diff}(\mathcal{O}_{V^{(d)},\xi ^{(d)}}[x_i][f_iW^i])$, since $f_i\in \mathcal{O}_{V^{(d)},\xi ^{(d)}}[x_i]$. Therefore we can write
\begin{equation}\label{eq:diff_en_trozos}
\mathcal{G}_X^{(n)}=\mathrm{Diff}(\mathcal{O}_{V^{(d)},\xi ^{(d)}}[x_1][f_1W^{b_1}])\odot \ldots \odot \mathrm{Diff}(\mathcal{O}_{V^{(d)},\xi ^{(d)}}[x_{n-d}][f_{n-d}W^{b_{n-d}}])\mbox{.}
\end{equation}
Observe that, for each $f_i$, $H_i=\left\{ f_i=0\right\} $ is a hypersurface in $V_i^{(e)}=\mathrm{Spec}(\mathcal{O}_{V^{(d)},\xi ^{(d)}}[x_i])$, where $e=d+1$. Using the hypersurface case, the Rees algebra 
\begin{equation}\label{eq:G_i}
\mathcal{G}_{H_i}^{(e)}=\mathrm{Diff}(\mathcal{O}_{V^{(d)},\xi ^{(d)}}[x_i][f_iW^{b_i}])=\mathcal{O}_{V^{(d)},\xi ^{(d)}}[x_i][x_iW]\odot \mathcal{G}^{(d)}_{H_i}
\end{equation}
represents $\mathrm{mult}(H_i)$ (see Remark \ref{rem:ex1}).
\vspace{0.2cm}
\end{Par}

\begin{Rem}\label{rem:rem:relation_varieties_hypersurfaces0}
Using (\ref{eq:diff_en_trozos}) we can rewrite $\mathcal{G}_{X}^{(n)}$ in terms of the $\mathcal{G}_{H_i}^{(e)}$ for $i=1,\ldots ,n-d$:
\begin{align}\label{eq:rel_G_i-G}
& \mathcal{G}^{(n)}_{X}=\mathcal{G}^{(e)}_{H_{1}}\odot \ldots \odot \mathcal{G}^{(e)}_{H_{n-d}}= \notag \\
= \mathcal{O}_{V^{(d)},\xi ^{(d)}}[x_1][x_1W] & \odot \mathcal{G}_{H_1}^{(d)}\odot \ldots \odot \mathcal{O}_{V^{(d)},\xi ^{(d)}}[x_{n-d}][x_{n-d}W]\odot \mathcal{G}_{H_{n-d}}^{(d)}= \\
=\mathcal{O}_{V^{(n)},\xi } & [x_1W,\ldots ,x_{n-d}W]\odot \mathcal{G}_{H_1}^{(d)}\odot \ldots \odot \mathcal{G}_{H_{n-d}}^{(d)}\mbox{.} \notag
\end{align}

If one goes back to diagram (\ref{diag:comm_compl-elim}), using the factorization
\begin{equation}\label{diag:factor_Hi}
\xymatrix{
\mathcal{O}_{V^{(n)},\xi } \ar[r]^{\varphi } & K[[t]] \\
& \mathcal{O}_{V_i^{(e)},\xi }=\mathcal{O}_{V^{(d)},\xi ^{(d)}}[x_i]  \ar[ul] \ar[u]_{\varphi _i} \\
\mathcal{O}_{V^{(d)},\xi ^{(d)}} \ar[ur] \ar[uu]_{\beta ^*} &
}
\end{equation}
one can consider also the Rees algebras $\mathcal{G}_{H_{i,0}}^{(e+1)}$ and $\mathcal{G}_{\tilde{H}_{i,0}}^{(e+1)}$ induced by $\mathcal{G}_{H_i}^{(e)}$ over $\mathcal{O}_{V_{i,0}^{(e+1)},\xi_0 }=\mathcal{O}_{V_0^{(d+1)},\xi _0^{(d+1)}}[x_i]$ and $(\mathcal{O}_{V_{i,0}^{(e)},\xi }\otimes _{k} K[[t]])_{\xi _0}$ respectively.\\
\\
\end{Rem}
Consider now an arc $\varphi \in \mathcal{L}(X)$ through $\xi $, and the $\mathcal{O}_{\tilde{V}_0^{(n+1)},\xi _0}$-Rees algebra $\mathcal{G}_{X_0,\varphi }^{(n+1)}$ of contact of $\varphi $ with $\mathrm{\underline{Max}\; mult}(X)$. Let us suppose that $\varphi $ is given by $(\varphi _{x_1},\ldots ,\varphi _{x_{n-d}},\varphi _{z_1},\ldots ,\varphi _{z_{d}})$ as in (\ref{diag:triangle}). At the same time, for $i=1,\ldots ,n-d$, the projection of $\varphi $ onto $V_i^{(e)}$ by (\ref{diag:factor_Hi}) is an arc $\varphi _i$ given by $(\varphi _{x_i},\varphi _{z_1},\ldots ,\varphi _{z_d})$ in $\mathcal{L}(H_i)$. Therefore we can define 
\begin{equation}\label{eq:G_i-arc}
\mathcal{G}^{(e+1)}_{H_{i,0},\varphi _i}=\mathrm{Diff}((\mathcal{O}_{V^{(n)},\xi }\otimes _{k} K[[t]])_{\xi _0}[f_iW^{b_i},h_iW,h_{n-d+1}W,\ldots ,h_nW])=\mathcal{G}_{H_i}^{(e)}\odot \mathcal{G}_{\varphi _i}^{(e+1)}\mbox{,}
\end{equation}
where $h_i=x_i-\varphi _{x_i}$ for $i=1,\ldots ,n-d$ and $h_{n-d+j}=z_{j}-\varphi _{z_j}$ for $j=1,\ldots ,d$, and
\begin{align}\label{eq:RA_arc_gral_decomp}
\mathcal{G}_{\varphi }^{(n+1)} & =(\mathcal{O}_{V_{i,0}^{(e)},\xi }\otimes _{k} K[[t]])_{\xi _0} [h_1W,\ldots ,h_nW]=\notag \\
= (\mathcal{O}_{V_{i,0}^{(e)},\xi }\otimes _{k} K[[t]])_{\xi _0}[h_1W,h_{n-d+1}W, & \ldots ,h_nW]\odot \ldots \odot (\mathcal{O}_{V_{i,0}^{(e)},\xi }\otimes _{k} K[[t]])_{\xi _0}[h_{n-d}W,h_{n-d+1}W,\ldots ,h_nW]= \\
& =\mathcal{G}_{\varphi _1}^{(e+1)}\odot \ldots \odot \mathcal{G}_{\varphi _{n-d}}^{(e+1)}\mbox{.}\notag
\end{align}
Now we can use the result for hypersurfaces in Theorem \ref{thm:desigualdad} to assert that, for $i=1,\ldots ,n-d$,
$$\frac{\mathrm{ord}_{\xi }(\mathcal{G}_{H_{i,0},\varphi _i}^{(1)})}{\mathrm{ord}(\varphi _i)}\geq \mathrm{ord}_{\xi }(\mathcal{G}_{H_{i}}^{(d)})\mbox{.}$$
Note that 
\begin{equation}\label{orders_arcs_i}
\mathrm{ord}(\varphi )=\min _{i=1,\ldots ,n-d} \left\{ \mathrm{ord}(\varphi _i)\right\} \mbox{.}
\end{equation}
The following remark will be important for the generalization of Theorem \ref{thm:desigualdad}.
\vspace{0.2cm}

\begin{Rem}\label{rem:relation_varieties_hypersurfaces}The Rees algebra $\mathcal{G}^{(n+1)}_{X_0,\varphi }$ can be written in terms of the $\mathcal{G}^{(e+1)}_{H_{i,0},\varphi _i}$, by (\ref{eq:thm_amalgama}), (\ref{eq:rel_G_i-G}), (\ref{eq:RA_arc_gral_decomp}) and (\ref{eq:G_i-arc}):
\begin{align}\label{eq:rel_G_i-arc-G-arc}
\mathcal{G}_{X_0,\varphi }^{(n+1)}=\mathcal{G}_{X}^{(n+1)}\odot \mathcal{G}_{\varphi }^{(n+1)}= \mathcal{O}_{V_0^{(n+1)},\xi _0^{(n+1)}} & [x_1W,\ldots ,x_{n-d}W]\odot \mathcal{G}_{H_{1}}^{(d)}\odot \ldots \odot \mathcal{G}_{H_{n-d}}^{(d)} \odot \mathcal{G}_{\varphi }^{(n+1)}= \notag\\
=\mathcal{G}_{H_1}^{(e)}\odot \ldots \odot \mathcal{G}_{H_{n-d}}^{(e)}\odot \mathcal{G}_{\varphi _1}^{(e+1)}\odot \ldots \odot  & \mathcal{G}_{\varphi _{n-d}}^{(e+1)}=\mathcal{G}_{H_1}^{(e)}\odot \mathcal{G}_{\varphi _1}^{(e+1)}\odot \ldots \odot \mathcal{G}_{H_{n-d}}^{(e)}\odot \mathcal{G}_{\varphi _{n-d}}^{(e+1)} =\\
=\mathcal{G}_{H_{1,0},\varphi _1}^{(e+1)}\odot & \ldots \odot \mathcal{G}_{H_{n-d,0},\varphi _{n-d}}^{(e+1)} \mbox{.} \notag
\end{align}
\end{Rem}
By expressing the algebras $\mathcal{G}^{(n+1)}_{X_0}$ and $\mathcal{G}^{(n+1)}_{X_0,\varphi }$ in terms of Rees algebras attached to hypersurfaces as we have done in (\ref{eq:rel_G_i-G}) and (\ref{eq:rel_G_i-arc-G-arc}), it is easy to establish a relation among the order of all Rees algebras involved in both cases, as the following Lemma states:
\vspace{0.2cm}

\begin{Lemma}\label{lemma:rel_orders} Let $X$ be a $d$-dimensional variety.\\

\begin{enumerate}
	\item Let $\mathcal{G}_{X}^{(n)}$ and $\mathcal{G}_{H_{i}}^{(e)}$ be as in (\ref{eq:G_gral}) and (\ref{eq:G_i}). Let $\mathcal{G}_X^{(d)}$ and $\mathcal{G}_{H_i}^{(d)}$ be respectively the elimination Rees algebras associated to their projection over $V^{(d)}$. Then
\begin{equation}\label{eq:rel_elim_G_i-G}
\mathcal{G}_X^{(d)}=\mathcal{G}_{H_1}^{(d)}\odot \ldots \odot \mathcal{G}_{H_{n-d}}^{(d)}\mbox{,}
\end{equation}
and thus
\begin{equation}\label{eq:rel_ord_G_i-G}
\mathrm{ord}_{\xi }(\mathcal{G}_X^{(d)})=\min _{i=1,\ldots ,n-d}\left\{ \mathrm{ord}_{\xi }(\mathcal{G}_{H_i}^{(d)})\right\} \mbox{.}
\end{equation}

	\item Let $\mathcal{G}_{X_0,\varphi }^{(n+1)}$ and $\mathcal{G}_{H_{i,0},\varphi _i}^{(e+1)}$ be as in (\ref{eq:rel_G_i-arc-G-arc}) and (\ref{eq:G_i-arc}). Let $\mathcal{G}_{X_0,\varphi }^{(1)}$ and $\mathcal{G}^{(1)}_{H_{i,0},\varphi _i}$ be respectively their restrictions to the curves defined by the arcs $\varphi ,\varphi _1,\ldots \varphi _{n-d}$ (as in Proposition \ref{prop:eq_amalgama_ordenes}). Then
\begin{equation}\label{eq:rel_elim_G_i-arc-G-arc}
\mathcal{G}_{X_0,\varphi }^{(1)}=\mathcal{G}_{H_{1,0},\varphi _1}^{(1)}\odot \ldots \odot \mathcal{G}_{H_{n-d,0},\varphi _{n-d}}^{(1)}\mbox{.}
\end{equation}
As a consequence
\begin{equation}\label{eq:rel_ord_G_i-arc-G-arc}
\mathrm{ord}_{\xi }(\mathcal{G}_{X_0,\varphi }^{(1)})=\min _{i=1,\ldots ,n-d}\left\{ \mathrm{ord}_{\xi }(\mathcal{G}_{H_{i,0},\varphi _i}^{(1)})\right\} \mbox{.}
\end{equation}
\end{enumerate}

\end{Lemma}

\begin{proof} Part (1) follows from the elimination of $\mathcal{G}_X^{(n)}$ associated to the projection $V^{(n)}\longrightarrow V^{(d)}$, using the expression in (\ref{eq:rel_G_i-G}). For (2), one must note, by looking at the expression in (\ref{eq:rel_G_i-arc-G-arc}), that the restriction of $\mathcal{G}_{X_0,\varphi }^{(n+1)}$ to the curve defined by $\varphi $ equals the smallest algebra containing the restrictions of the $\mathcal{G}_{H_{i,0},\varphi _{i}}^{(e+1)}=\mathcal{O}_{V_0^{(n+1)},\xi _0}[x_iW]\odot \mathcal{G}_{H_{i}}^{(d)}\odot \mathcal{G}_{\varphi _i}^{(e+1)}$ to the respective curves defined by the $\varphi _i$, since all the Rees algebras are differentially closed. 
\end{proof}
\vspace{0.5cm}

\subsubsection*{Results for the general case}

\begin{Thm}\label{thm:desigualdad_gral}
Let $X$ be a variety as in Section \ref{subsubsec:setting_gen}, let $\xi \in \mathrm{\underline{Max}\; mult}(X)$ and let $\varphi $ be an arc in $X$ through $\xi $ with the notation used there. Then 
\begin{equation}\label{eq:r_gral}
\overline{r}_{X,\varphi }\geq \mathrm{ord}_{\xi }(\mathcal{G}_X^{(d)})\mbox{.} 
\end{equation}
\end{Thm}

\begin{proof} From (\ref{eq:rel_ord_G_i-arc-G-arc}) we obtain
\begin{equation*}
\overline{r}_{X,\varphi }=\frac{\mathrm{ord}_{\xi }(\mathcal{G}_{X_0,\varphi }^{(1)})}{\mathrm{ord}(\varphi )}=\frac{\min _{i=1,\ldots ,n-d}\left\{ \mathrm{ord}_{\xi }(\mathcal{G}_{H_{i,0},\varphi _i}^{(1)})\right\} }{\mathrm{ord}(\varphi )}\mbox{.}
\end{equation*}
For every $i\in \left\{ 1,\ldots ,n-d\right\} $, Theorem \ref{thm:desigualdad} gives
$$\frac{\mathrm{ord}_{\xi }(\mathcal{G}_{H_{i,0},\varphi _i}^{(1)})}{\mathrm{ord}(\varphi _i)}\geq \mathrm{ord}_{\xi }(\mathcal{G}_{H_i}^{(d)})\mbox{,}$$
and this together with (\ref{orders_arcs_i}) and (\ref{eq:rel_ord_G_i-G}) implies
\begin{equation*}
\frac{\mathrm{ord}_{\xi }(\mathcal{G}_{H_{i,0},\varphi _i}^{(1)})}{\mathrm{ord}(\varphi )}\geq \frac{\mathrm{ord}_{\xi }(\mathcal{G}_{H_{i,0},\varphi _i}^{(1)})}{\mathrm{ord}(\varphi _i)}\geq \mathrm{ord}_{\xi }(\mathcal{G}_{H_i}^{(d)})\geq \mathrm{ord}_{\xi }(\mathcal{G}_{X}^{(d)})\mbox{,\; \;}\forall i=1,\ldots ,n-d\mbox{.}
\end{equation*}
As a consequence, we get
\begin{equation*}
\overline{r}_{X,\varphi }=\frac{\min _{i=1,\ldots ,n-d}\left\{ \mathrm{ord}_{\xi }(\mathcal{G}_{H_{i,0},\varphi _i}^{(1)})\right\} }{\mathrm{ord}(\varphi )}\geq \mathrm{ord}_{\xi }(\mathcal{G}_{X}^{(d)})\mbox{,}
\end{equation*}
concluding the proof of the Theorem.
\end{proof}
\vspace{0.2cm}

\begin{Rem}\label{rem:lifting_arcs_gral}
If $k$ is a field of characteristic zero, it is always possible to find a diagonal arc $\bar{\varphi }^{(d)}$ which is diagonal-generic for $\mathcal{G}_{H_i}^{(d)}$ for $i=1,\ldots ,n-d$. As we did in Remark \ref{obs:desig_lema_arcos_diagonales}, one needs only to consider for each $i\in \left\{ 1,\ldots ,n-d\right\} $, an element $p_iW^{l_i}\in \mathcal{G}_{H_i}^{(d)}$ such that $\mathrm{ord}_{\xi }(\mathcal{G}_{H_i}^{(d)})=\frac{\mathrm{ord}_{\xi }(p_i)}{l_i}$ and find units $u_1,\ldots ,u_d\in k$ such that $\mathrm{in}_{\xi }(p_i)(u_1,\ldots ,u_d)\neq 0$ for $i=1,\ldots ,n-d$, which is possible because we are considering once more a finite set of elements $\left\{ p_1,\ldots ,p_{n-d}\right\} $ and an infinite field $k$. Now, the arc $\bar{\varphi }^{(d)}$ given as $(u_1t^{\alpha },\ldots ,u_dt^{\alpha })$,
where $\alpha $ is some positive integer, is diagonal-generic for $\mathcal{G}_{H_i}^{(d)}$ for all $i=1,\ldots ,n-d$. Note that, in particular, $\bar{\varphi }^{(d)}$ is diagonal-generic for $\mathcal{G}_X^{(d)}$ (this follows from Lemma \ref{lemma:rel_orders}). Note also that by Lemma \ref{lemma:lifting_arcs}, $\bar{\varphi }^{(d)}$ can be lifted to an arc $\varphi $ in $X$ and the projection of $\varphi $ onto $V^{(d)}$ is diagonal-generic for every $\mathcal{G}_{H_i}^{(d)}$, $i=1,\ldots ,n-d$.
\end{Rem}

\begin{Thm}\label{thm:igualdad_gral}
Let $X$ be a variety as in Section \ref{subsubsec:setting_gen} and let $\xi \in \mathrm{\underline{Max}\; mult}(X)$. There exists an arc $\varphi \in \mathcal{L}(X)$ through $\xi $ such that
\begin{equation}\label{eq:igualdad_gral}
\overline{r}_{X,\varphi }=\mathrm{ord}_{\xi }(\mathcal{G}_{X}^{(d)})\mbox{.}
\end{equation}
\end{Thm}

\begin{proof}
By Remark \ref{rem:lifting_arcs_gral}, we can choose a diagonal arc which is diagonal-generic for $\mathcal{G}_{H_1}^{(d)},\ldots ,\mathcal{G}_{H_{n-d}}^{(d)}$ and $\mathcal{G}_X^{(d)}$. Let us denote it by $\bar{\varphi }^{(d)}$. We can lift $\bar{\varphi }^{(d)}$ to an arc $\varphi \in \mathcal{L}(X)$ through $\xi $. By Remark \ref{rem:lifting_arcs} we know that $\varphi $ is given (as in (\ref{diag:triangle})) by
\begin{equation*}
(g_1(t),\ldots ,g_{n-d}(t),u_1g'(t),\ldots ,u_dg'(t))
\end{equation*}
for some $g_1(t),\ldots g_{n-d}(t),g'(t)\in K[[t]]$ and some $u_1,\ldots ,u_d\in k$ by Lemma \ref{lemma:lifting_arcs}. By Remark \ref{rem:lifting_arcs_gral}, $\varphi ^{(d)}$ is also generic for $\mathcal{G}_{H_i}^{(d)}$, $i=1,\ldots ,n-d$. The proof will be complete by showing that any arc of this form satisfies (\ref{eq:igualdad_gral}).\\
\\
Let us denote $N=\mathrm{ord}_{t}(g'(t))$. As in (\ref{diag:fact_hyp}), $\beta $ factorizes via $\mathcal{O}_{H_i,\xi }$ for $i=1,\ldots ,n-d$:
\begin{equation}
\xymatrix{ 
\mathcal{O}_{X^{(d)},\xi }\cong \mathcal{O}_{V^{(d)},\xi ^{(d)}}[x_1,\ldots ,x_{n-d}]/I(X) & & \mathcal{O}_{V^{(d)},\xi ^{(d)}}[x_1,\ldots ,x_{n-d}] \ar[ll] \\
\mathcal{O}_{H_i,\xi }\cong \mathcal{O}_{V^{(d)},\xi ^{(d)}}[x_i]/(f_i) \ar[u] & & \\
& \mathcal{O}_{V^{(d)},\xi ^{(d)}} \ar[uul]_{\beta _X^*} \ar[uur]^{\beta ^*}  \ar[ul] & 
}
\end{equation}
and hence the projection $\varphi _i$ of $\varphi $ onto $V_i^{(d+1)}$ is, in particular, a lifting of $\bar{\varphi }^{(d)}$ to $H_i$, and the projection of each $\varphi _i$ to $V^{(d)}$ is $\varphi ^{(d)}$, which is diagonal-generic for $\mathcal{G}_{H_i}^{(d)}$. Thus, the result of Theorem \ref{thm:igualdad} holds for each $H_i$, as well as Remark \ref{rem:order_arc_id}, implying 
$$\mathrm{ord}_{\xi }(\mathcal{G}_{H_{i,0},\varphi _i}^{(1)})=\mathrm{ord}(\varphi _i)\cdot \mathrm{ord}_{\xi }(\mathcal{G}_{H_i}^{(d)})=N\cdot \mathrm{ord}_{\xi }(\mathcal{G}_{H_i}^{(d)})$$
for $i=1,\ldots ,n-d$. By (\ref{orders_arcs_i}) we also know that $\mathrm{ord}(\varphi )=N$. From this, together with Lemma \ref{lemma:rel_orders}, it follows that
\begin{align*}
&\bar{r}_{X,\varphi }=\frac{\mathrm{ord}_{\xi }(\mathcal{G}_{X_0,\varphi }^{(1)})}{\mathrm{ord}(\varphi )}=\frac{\min _{i=1,\ldots ,n-d}\left\{ \mathrm{ord}_{\xi }(\mathcal{G}_{H_{i,0},\varphi _i}^{(1)})\right\} }{N}= \\
& = \frac{\min _{i=1,\ldots ,n-d}\left\{ N\cdot \mathrm{ord}_{\xi }(\mathcal{G}_{H_i}^{(d)})\right\} }{N}=\min _{i=1,\ldots ,n-d}\left\{ \mathrm{ord}_{\xi }(\mathcal{G}_{H_i}^{(d)})\right\} =\mathrm{ord}_{\xi }(\mathcal{G}_X^{(d)})\mbox{,}
\end{align*}
which completes the proof.
\end{proof}

\subsection{Consequences of the main result}

When we first presented our results in Section \ref{subsec:RAandNash}, we gave there a version of Theorem \ref{thm:main}, which relates the invariants $\bar{r}_{X}$ and $\mathrm{ord}_{\xi }\mathcal{G}_{X}^{(d)}$ for any $X$. It is clear now that the statement there is a consequence of Theorems \ref{thm:desigualdad_gral} and \ref{thm:igualdad_gral}. In addition, we claimed to know a relation between $\mathrm{ord}_{\xi }\mathcal{G}_{X}^{(d)}$ and $\rho _{X,\varphi }$ for any arc $\varphi $ in $X$ through $\xi $. The following theorem shows this relation, which is just a small step more than a consequence of Proposition \ref{thm:r_elim_amalgama} and Theorem \ref{thm:main}.\\
\\
\begin{Thm}\label{thm:main_order-rho}
Let $X$ be a variety of dimension $d$. Let $\xi $ be a point in $\mathrm{\underline{Max}\, mult}(X)$. For any arc $\varphi $ in $X$ through $\xi $, 
$$\rho _{X,\varphi }\geq \left[ \mathrm{ord}_{\xi }\mathcal{G}_X^{(d)} \right] \cdot \mathrm{ord}(\varphi )\mbox{,}$$
where $\mathcal{G}_X^{(d)}$ is the elimination algebra described in Example \ref{ex:elim_gral}. Moreover, 
$$\inf _{\varphi \in \mathcal{L}_{\xi }(X)}\left\{ [\bar{\rho }_{X,\varphi }]\right\} =\left[ \mathrm{ord}_{\xi }\mathcal{G}_X^{(d)}\right] \mbox{.}$$
one can find an arc $\varphi _0$ in $X$ through $\xi $ satisfying
$$\bar{\rho }_{X,\varphi _0}=\mathrm{ord}_{\xi }\mathcal{G}_X^{(d)}\mbox{.}$$
\end{Thm}
\vspace{0.5cm}
\begin{proof}
For the first formula we use Proposition \ref{thm:r_elim_amalgama} and Theorem \ref{thm:desigualdad_gral}
$$\rho _{X,\varphi }=\left[ r_{X,\varphi } \right] = \left[ \frac{r_{X,\varphi }}{\mathrm{ord}(\varphi )}\cdot \mathrm{ord}(\varphi ) \right] \geq \left[ \frac{r_{X,\varphi }}{\mathrm{ord}(\varphi )}  \right] \cdot \mathrm{ord}(\varphi ) \geq \left[ \mathrm{ord}_{\xi }\mathcal{G}_X^{(d)} \right] \cdot \mathrm{ord}(\varphi )\mbox{.}$$
As a consequence, of this result,
$$\frac{r_{X,\varphi }}{\mathrm{ord}(\varphi )}\geq \frac{[r_{X,\varphi }]}{\mathrm{ord}(\varphi )}=\frac{\rho _{X,\varphi }}{\mathrm{ord}(\varphi )}\geq \frac{\left[ \mathrm{ord}_{\xi }\mathcal{G}_X^{(d)} \right] \cdot \mathrm{ord}(\varphi )}{\mathrm{ord}(\varphi )}=\left[ \mathrm{ord}_{\xi }\mathcal{G}_X^{(d)} \right] \mbox{.}$$
That is,
$$\bar{r}_{X,\varphi }\geq \bar{\rho }_{X,\varphi }\geq \left[ \mathrm{ord}_{\xi }\mathcal{G}_X^{(d)} \right] \mbox{,}$$
where we may take integral parts and then the minimums over all arcs in $X$ through $\xi $, obtaining 
$$\mathrm{min}_{\varphi \in \mathcal{L}_{\xi }(X)}\left\{ [\bar{r}_{X,\varphi }]\right\} \geq \mathrm{min}_{\varphi \in \mathcal{L}_{\xi }(X)}\left\{ [\bar{\rho }_{X,\varphi }]\right\} \geq \left[ \mathrm{ord}_{\xi }\mathcal{G}_X^{(d)} \right] =\mathrm{min}_{\varphi \in \mathcal{L}_{\xi }(X)}\left\{ [\bar{r}_{X,\varphi }]\right\} \mbox{,}$$
which implies the second formula of the Theorem.\\
\\
Finally, for the third formula, let us go back to the proof of Theorem \ref{eq:igualdad_gral}. It allows us to find an arc $\varphi _1$ in $X$ through $\xi $ satisfying $\bar{r}_{X,\varphi _1}=\mathrm{ord}_{\xi }\mathcal{G}_X^{(d)}$. This arc will be given by $(g_1(t),\ldots ,g_{n-d}(t),u_1g'(t),\ldots ,u_dg'(t))$ for some $g_1(t),\ldots ,g_{n-d}(t),g'(t)\in K[[t]]$ and some $u_1,\ldots ,u_d\in k$, and the projection $\varphi _1^{(d)}$ given by $(u_1g'(t),\ldots ,u_dg'(t))$ will be diagonal generic for $\mathcal{G}_X^{(d)}$. Let us choose $\varphi _0$ as the arc in $X$ through $\xi $ given by 
$$(g_1(t^b),\ldots ,g_{n-d}(t^b),u_1g'(t^b),\ldots ,u_dg'(t^b))\mbox{,}$$ 
for which the projection $\varphi _0^{(d)}$ is also diagonal generic for $\mathcal{G}_X^{(d)}$, so it is also valid for Theorem \ref{eq:igualdad_gral}, having $\bar{r}_{X,\varphi _0}=\mathrm{ord}_{\xi }\mathcal{G}_X^{(d)}$. In particular, this implies that $\bar{r}_{X,\varphi _0}=\bar{r}_{X,\varphi _1}$. Note also that $\mathrm{ord}(\varphi _0)=\mathrm{ord}(\varphi _1)\cdot b$.

We have found an arc such that 
$$r_{X,\varphi _0}=\mathrm{ord}_{\xi }\mathcal{G}_X^{(d)}\cdot \mathrm{ord}(\varphi _0)\mbox{,}$$ 
and for which 
$$r_{X,\varphi _0}=\left[ r_{X,\varphi _0}\right] =\rho _{X,\varphi _0}\mbox{,}$$ 
since $\mathcal{G}_X^{(d)}\in \frac{1}{b}\cdot \mathbb{Z}_{>0}$ and $\mathrm{ord}(\varphi _0)\in b\cdot \mathbb{Z}_{>0}$, concluding the proof.
\end{proof}

\vspace{1cm}

The following Corollary gives a characterization of $\mathrm{ord}_{\xi }\mathcal{G}^{(d)}_X$ in terms of the $\bar{\rho }_{X,\varphi }$.

\begin{Cor}
Let $X$ be a variety of dimension $d$. Let $\xi $ be a point in $\mathrm{\underline{Max}\, mult}(X)$. Consider the subset $\mathcal{C}\subset \mathcal{L}_{\xi }(X)$ of all arcs $\varphi $ satisfying $\bar{r}_{X,\varphi }=\mathrm{ord}_{\xi }\mathcal{G}^{(d)}_X$. Then:
$$\mathrm{ord}_{\xi }\mathcal{G}^{(d)}_X=\mathrm{max}_{\varphi \in \mathcal{C}}\left\{ \bar{\rho }_{X,\varphi }\right\}\mbox{.}$$
\end{Cor}

\begin{proof}
For any arc $\varphi  \in \mathcal{C}$, 
$$\rho _{X,\varphi }=\left[ r_{X,\varphi }\right] =\left[ \mathrm{ord}_{\xi }\mathcal{G}^{(d)}_X\cdot \mathrm{ord}(\varphi )\right] \mbox{.}$$
It follows that 
$$\frac{\rho _{X,\varphi }}{\mathrm{ord}(\varphi )}=\frac{\left[ \mathrm{ord}_{\xi }\mathcal{G}^{(d)}_X\cdot \mathrm{ord}(\varphi )\right] }{\mathrm{ord}(\varphi )}\leq \mathrm{ord}_{\xi }\mathcal{G}^{(d)}_X\mbox{.}$$
The result is a consequence of this together with Theorem \ref{thm:main_order-rho}.
\end{proof}

\vspace{1cm}

The following relations hold for every arc $\varphi \in \mathcal{L}(X)$ through $\xi $:

\begin{Cor}
For $X$ as in Proposition \ref{thm:r_elim_amalgama}, and for every arc $\varphi \in \mathcal{L}(X)$ through $\xi $:
\begin{enumerate}
 
	\item $\bar{r}_{X,\varphi }\geq \bar{\rho }_{X,\varphi }$
	\item $\bar{\rho }_{X,\varphi }\geq [ \mathrm{ord}_{\xi }\mathcal{G}^{(d)}_X] $
	\item Since $\bar{r}_{X,\varphi }\geq \mathrm{ord}_{\xi }\mathcal{G}^{(d)}_X$ and $\bar{r}_{X,\varphi }\geq \bar{\rho }_{X,\varphi}\geq [ \mathrm{ord}_{\xi }\mathcal{G}^{(d)}_X] $, two possible situations can happen for $\bar{\rho }_{X,\varphi }$ and $\mathrm{ord}_{\xi }\mathcal{G}^{(d)}_X$, namely:
	
	\begin{itemize}
		\item $\bar{r}_{X,\varphi } \geq \mathrm{ord}_{\xi }\mathcal{G}^{(d)}_X\geq \bar{\rho }_{X,\varphi }\geq [ \mathrm{ord}_{\xi }\mathcal{G}^{(d)}_X] $
		\item $\bar{r}_{X,\varphi }\geq \bar{\rho }_{X,\varphi } > \mathrm{ord}_{\xi }\mathcal{G}^{(d)}_X\geq [ \mathrm{ord}_{\xi }\mathcal{G}^{(d)}_X] $
	\end{itemize}
\end{enumerate}
\end{Cor}

\vspace{0.2cm}

\begin{proof}
\begin{enumerate}
\item Follows from the definitions of $\bar{r}_{X,\varphi }$ and $\bar{\rho }_{X,\varphi }$ toghether with Proposition \ref{thm:r_elim_amalgama}.
\item By Definition \ref{def:rho_bar}, Proposition \ref{thm:r_elim_amalgama}, Theorem \ref{thm:desigualdad_gral}:
$$\bar{\rho }_{X,\varphi }=\frac{\rho _{X,\varphi }}{\mathrm{ord}\varphi }=\frac{[r_{X,\varphi }]}{\mathrm{ord}\varphi }\geq \frac{ [ \mathrm{ord}_{\xi }\mathcal{G}^{(d)}_X\cdot \mathrm{ord}\varphi ] }{\mathrm{ord}\varphi }\geq [ \mathrm{ord}_{\xi }\mathcal{G}^{(d)}_X] \mbox{.}$$
\item This is just an observation which follows from (2) and (3).
\end{enumerate}
\end{proof}

\end{document}